\documentclass[11pt]{article}
\usepackage[tbtags]{amsmath}
\usepackage{amssymb}
\usepackage{amsthm}
\usepackage[misc]{ifsym}
\usepackage{cases}
\usepackage{mathrsfs}
\usepackage{subfigure}
\usepackage{mathtools}
\usepackage{float}
\usepackage{color}

\numberwithin{equation}{section}
\normalsize
\title{\bf Robust Incentive Stackelberg Mean Field Stochastic Linear-Quadratic Differential Game with Model Uncertainty \thanks{This work is supported by National Key R\&D Program of China (2022YFA1006104), National Natural Science Foundations of China (12471419, 12271304), and Shandong Provincial Natural Science Foundation (ZR2024ZD35).}}
\author{\normalsize Na Xiang\thanks{\it School of Mathematics, Shandong University, Jinan 250100, P.R. China, E-mail: 202211967@mail.sdu.edu.cn},\quad Jingtao Shi\thanks{\it Corresponding author, School of Mathematics, Shandong University, Jinan 250100, P.R. China, E-mail: shijingtao@sdu.edu.cn}}
\newtheorem{mypro}{Proposition}[section]
\newtheorem{mythm}{Theorem}[section]
\newtheorem{mydef}{Definition}[section]
\newtheorem{mylem}{Lemma}[section]
\newtheorem{Remark}{Remark}[section]
\newtheorem{mycor}{Corollary}[section]
\newtheorem{Example}{Example}[section]

\begin{document}

\maketitle
	
\noindent{\bf Abstract:}\quad This paper investigates a robust incentive Stackelberg stochastic differential game problem for a linear-quadratic mean field system, where the model uncertainty appears in the drift term of the leader's state equation. Moreover, both the state average and control averages enter into the leader's dynamics and cost functional. Based on the zero-sum game approach, mean field approximation and duality theory, firstly the representation of the leader's limiting cost functional and the closed-loop representation of decentralized open-loop saddle points are given, via decoupling methods. Then by convex analysis and the variational method, the decentralized strategies of the followers' auxiliary limiting problems and the corresponding consistency condition system are derived. Finally, applying decoupling technique, the leader's approximate incentive strategy set is obtained, under which the asymptotical robust incentive optimality of the decentralized mean field strategy is verified. A numerical example is given to illustrate the theoretical results.

\vspace{2mm}
	
	
\vspace{2mm}
	
	
\section{Introduction}

The Stackelberg game, also called the leader-follower game, was introduced by von Stackelberg \cite{Stackelberg34} in 1934, which is one of the important topics in the dynamic games. For the Stackelberg stochastic differential game, pioneer work was done by Bagchi and Ba\c{s}ar \cite{Bagchi81}, where the diffusion term of the state equation did not contain the state and control variables. Yong \cite{Yong02} discussed the open-loop solution of \emph{stochastic linear quadratic} (SLQ) nonzero-sum Stackelberg differential game, where the coefficients of the system with both state-dependent and control-dependent noise are random, and the weight matrices for the controls in the cost functionals are not necessarily positive definite. A lot of results have been made on this issue, one can also refer to Bensoussan et al. \cite{Bensoussan15}, Shi et al. \cite{Shi16,Shi17}, Li and Yu \cite{Li18}, Li et al. \cite{Li21}.
However, since the leader's strategy is affected by the follower's behavior, it is not easy for the leader to achieve his global optimum. Therefore, an incentive strategy is claimed by the leader ahead of time, which makes the follower and the leader cooperate to achieve the leader's team-optimum. This kind of game is called the incentive Stackelberg game.
The study of incentive design problems has a long history, such as \cite{Zheng82,Zheng84,Basar84,Cansever85,Basar89,Basar99}. Ho et al. \cite{Ho82} investigated the deterministic and stochastic versions of the incentive problem, and discussed their relationship
to economic literature. Mizukami and Wu \cite{Mizukami87,Mizukami88} discussed the derivation of the sufficient conditions for the LQ incentive Stackelberg game with multi-players in a two-level hierarchy. The three-level incentive Stackelberg strategy in a nonlinear differential game were analyzed in \cite{Ishida87,Ishida83}. Li et al. \cite{Li02} studied the team-optimal state feedback Stackelberg strategy of a class of discrete-time two-person nonzero-sum LQ dynamic games. Mukaidani et al. \cite{Mukaidani17-1} discussed the incentive Stackelberg game for a class of Markov jump linear stochastic systems with multiple leaders and followers. In \cite{Yin22,Gao23,Gao24}, the incentive feedback strategy was discussed in Stackelberg game for some kinds of stochastic systems. Sanjari et al. \cite{Sanjari25} studied incentive designs for a class of stochastic Stackelberg games with one leader and a large number of followers.

For the research on incentive Stackelberg games, there is a fair amount of literature on model uncertainty. To capture the idea that the leader and follower only access the approximation model rather than the realistic one, the external disturbance may be viewed as a model uncertainty synthesized in state dynamics. For the stochastic version, Mukaidani et al. \cite{Mukaidani18-1} discussed the incentive Stackelberg game with one leader and one follower subject to external disturbance by means of static output-feedback. The incentive Stackelberg game with one leader and multiple non-cooperative followers subjected to the $H_\infty$ constraint was discussed in \cite{Mukaidani17-2,Ahmed17-2,Mukaidani18-2,Mukaidani19}. Multi-leader-follower incentive Stackelberg games for SLQ systems with $H_\infty$ constraint were investigated in \cite{Ahmed17-1,Mukaidani18-3,Mukaidani20}. Subsequently, Mukaidani et al. \cite{Mukaidani22} studied the static output feedback strategy for robust incentive Stackelberg games with a large population for mean field stochastic systems. Xiang and Shi \cite{Xiang24} concerned with a three-level multi-leader-follower incentive Stackelberg game with $H_\infty$ constraint.

The theory of \emph{mean field games} (MFGs) was originally formulated by Larsy and Lions \cite{Lasry07} and simultaneously by Huang et al. \cite{Huang06}. In recent years, it has wide applications in many fields, such as economics, engineering, biology, physics, etc. In this context, because of the highly complicated coupling structure, it is infeasible and ineffective to obtain the classical centralized strategies based on the information of all peers. As a substitute, it is more effective and tractable to discuss the associated MFG to determine an approximate equilibrium by considering its own individual dynamic information and some off-line quantities only. For further research on MFGs and related topics, the readers can refer to \cite{Bensoussan13,Moon17,Wang17,Hu18,Carmona18}.

Building on the above discussion, this paper investigates a class of linear-quadratic incentive Stackelberg MFGs with one leader and $N$ followers. This framework is strongly motivated by practical applications in economics and game theory. For example, in federal tax policy design, the government leverages announced tax rules and citizens’ reported tax filings to identify policies that steer individuals toward socially equitable tax compliance, thereby achieving efficient tax collection \cite{Sanjari25}. Similarly, ride-hailing platforms such as Didi and Uber employ dynamic pricing and subsidy schemes to coordinate a large population of drivers, boosting order volumes and strengthening platform influence and reputation \cite{Ratliff19}. Furthermore, regulators design carbon taxes or green subsidies to incentivize numerous enterprises: while reducing carbon emissions, these firms also engage in strategic market competition. Notably, the model of our system exhibit several distinctive characteristics: (\romannumeral1) As for the leader, state and control average terms of  the $N$ followers are not only included in the drift term of the leader's state equation, but also affect its cost functional, which indicating that the $N$ followers' states and strategies directly influence the leader. This is motivated by real-world scenarios where all followers may have instantaneous and immediate effects on the leader's state and cost functional. Furthermore, the leader's control components' average enters into the drift term and diffusion term of its state equation and cost functional. This model setup more accurately addresses the incentive Stackelberg mean field games from the perspective of the leader's incentive designs. (\romannumeral2) As for the $N$ followers, the drift term of each follower's state equation and the cost functional include the leader's corresponding control over him, reflecting the direct impact of the leader on the follower, a dynamic dictated by the structural nature of the incentive system. (\romannumeral3) The external disturbance $v(\cdot)$ is viewed as a model uncertainty synthesized in the drift term of the leader's state dynamics, which represent some external influence from the common environment for decision-making. In order to consider the robustness, we utilize $H_\infty$ control theory and take a zero-sum game approach by considering the external disturbance as the system of an adversarial player. Different from the $H_2/H_\infty$ control theory based on the nonzero-sum Nash equilibrium condition, a soft-constrained zero-sum Nash equilibrium is considered to avoid dealing with coupled Riccati equations. Therefore, the auxiliary limiting problem of the leader $\mathcal{A}_0$ is essentially a zero-sum game. Based on duality theory, we give a representation of the limiting cost functional to discuss its properties. To obtain uniform concavity of the cost functional, we analyze the disturbance attenuation parameter $\gamma$ and establish their relationships (see Figure \ref{1}). Then, we obtain (decentralized) saddle point strategies and ensure the solvabilities of four across-coupled Riccati equations by expanding the dimension. For the i-th follower's auxiliary limiting problem, using the  variational method, we obtain the decentralized optimal control of the follower $\mathcal{A}_i$ and the {\it consistency condition} (CC) system. By the decoupling method, we establish the sufficient criteria for the leader $\mathcal{A}_0$ to achieve its (approximate) incentive goal and obtain the (approximate) incentive strategy set (Theorem \ref{thm4.3}). Finally, we give the definition of asymptotic robust incentive strategy, and prove that the corresponding decentralized mean-field strategy satisfies the asymptotic optimality.

The main contributions of this paper can be summarized as follows.

\begin{itemize}
\item[$\bullet$] Firstly, we study a class of incentive Stackelberg SLQ large-population problems with model uncertainty by the fixed-point method, where the state average and control averages are not only included in the leader's state equation, but also affect its cost functional, and the external disturbance enters into the drift term of the leader's state equation.

\item[$\bullet$] Secondly, for the auxiliary limiting problem of the leader, by using $H_\infty$ control theory, the zero-sum game approach and duality theory, we give the representation of the limiting cost functional and obtain the decentralized saddle point strategies. Moreover, we discuss the interrelations among uniform concavity of the limiting cost functional, solvability of the corresponding Riccati equation, and disturbance attenuation parameter $\gamma$ (see Figure \ref{1}). By the decoupling technique, we derive the optimal closed-loop representation of open-loop saddle point via some Riccati equations.

\item[$\bullet$] Thirdly, applying convex analysis and the variational method, we get the decentralized optimal control of the i-th follower for its auxiliary limiting problem. The related consistency condition system by the fixed-point argument is established.

\item[$\bullet$] Finally, by the decoupling method, we obtain the leader's approximate incentive strategy set (Theorem \ref{thm4.3}), under which the definition of asymptotic robust incentive strategy is given and the asymptotic optimality of the mean-field strategy is proved.
\end{itemize}

The rest of the paper is organized as follows. In Section 2, we introduce some preliminary notations and formulate the SLQ incentive Stackelberg MFGs with model uncertainty. Section 3 discusses the  leader's limiting auxiliary problem. The auxiliary optimal control for the i-th follower is obtained, and the CC system and the approximate incentive strategy are established in Section 4. In Section 5, the asymptotic optimality of the mean-field strategy is obtained. Two numerical example are given in Section 6. Section 7 concludes this paper.

\section{Preliminaries}

Let ($\Omega,\mathcal{F},\mathbb{P}$) be a complete filtered probability space and an $N+1$-dimensional standard Brownian motion $\{W_i\}_{i=0}^N$ is defined on it. Let $\mathcal{F}_t$, $\mathcal{F}_t^i$ and $\mathcal{G}_t^i$ denote the $\sigma$-algebra generated by $\left\lbrace W_i(s),0\leq s\leq t, 0\leq i\leq N\right\rbrace$, $\left\lbrace W_i(s),0\leq s\leq t\right\rbrace$ and $\left\lbrace W_0(s), W_i(s),0\leq s\leq t, 0\leq i\leq N\right\rbrace$, for $0\leq i\leq N$, respectively and augmented by all the $\mathbb{P}$-null sets in $\mathcal{F}$. Throughout the paper, let $\mathbb{R}^n$ denote the \emph{n}-dimensional Euclidean space with standard Euclidean norm $\vert\cdot\vert $ and standard Euclidean inner product $\left\langle\cdot,\cdot\right\rangle $. The transpose of a vector (or matrix) $x$ is denoted by $\mathbf{\emph{x}}^\top$. $\mbox{Tr}(A)$ denotes the trace of a square matrix $A$. Let $\mathbb{R}^{n\times m}$ be the Hilbert space consisting of all $n\times m$-matrices with the inner product $\left\langle A,B\right\rangle := \mbox{Tr}(AB^\top$) and the Frobenius norm $\vert A \vert:=\langle A,A\rangle^\frac{1}{2}$. Denote the set of symmetric $n\times n$ matrices with real elements by $\mathbb{S}^n$. If $M\in\mathbb{S}^n$ is positive (semi-) definite, we write $M > (\geq) 0$. If there exists a constant $\delta>0$ such that $M\geq\delta I$, we write $M\gg0$.	

For a fixed $T>0$, on a finite time horizon $[0,T]$, we introduce the following spaces for a given Hilbert space $\mathbb{H}$:
\begin{equation*}\hspace{-2mm}	
\begin{aligned}
	&L_{\mathcal{F}}^2(0,T;\mathbb{H}):=\left\{\phi :[0,T]\times\Omega \mapsto \mathbb{H}\,\Big|\,\phi\mbox{ is }\mathcal{F}_t\mbox{-progressively measurable}, \mathbb{E}\int_0^T |\phi(t)|^2 dt < \infty \right\},\\
	&L^\infty(0,T;\mathbb{H}):=\left\{\phi :[0,T] \mapsto \mathbb{H}\,\big|\,\phi\mbox{ is uniformly bounded}\right\},\\
	&L_{\mathcal{F}}^2(\Omega;C([0,T];\mathbb{H})):=\bigg\{\phi :[0,T]\times\Omega \mapsto \mathbb{H}\,\big|\,\phi\mbox{ is }\mathcal{F}_t\mbox{-adapted, continuous}, \mathbb{E}\Big[\sup_{0\leq t\leq T} \vert\phi(t)\vert^2\Big] < \infty \bigg\},\\
	&C([0,T];\mathbb{H}:=\left\{\phi :[0,T]\mapsto \mathbb{H}\,\big|\,\phi\mbox{ is continuous}\right\},\\
	&L_{\mathcal{F}_T}^2(\Omega,\mathbb{H}):=\left\{\xi :\Omega \mapsto \mathbb{H}\,\big|\,\xi\mbox{ is }\mathcal{F}_T\mbox{-measurable random variable}, \mathbb{E}\vert\xi\vert^2 < \infty \right\}.
\end{aligned}
\end{equation*}

Consider a weakly coupled large-population system with a leader $\mathcal{A}_0$ and $N$ individual followers $\{\mathcal{A}_i:1\leq i\leq N\}$. The dynamics of the leader $\mathcal{A}_0$ and the i-th follower $\mathcal{A}_i$ are as follows:
\begin{equation}\label{state0}
\left\{\begin{aligned}
	dx_0(t)&=\left[A(t)x_0(t)+B(t)u_0^{(N)}(t)+F(t)x^{(N)}(t)+H(t)u_1^{(N)}(t)+E(t)v(t)\right]dt\\
	&\quad +\left[C(t)x_0(t)+D(t)u_0^{(N)}(t)+O(t)x^{(N)}(t)\right]dW_0(t),\\
	x_0(0)&=\xi\in\mathbb{R}^n,
\end{aligned}\right.
\end{equation}
\begin{equation}\label{statei}
\left\{\begin{aligned}
	dx_i(t)&=\left[\tilde{A}(t)x_i(t)+\tilde{B}(t)u_{1i}(t)+\tilde{F}(t)x^{(N)}(t)+\tilde{H}(t)u_{0i}(t)\right]dt+\sigma(t)dW_i(t),\\
	x_i(0)&=x\in\mathbb{R}^n,
\end{aligned}\right.
\end{equation}
where $x_0(\cdot)$ and $x_i(\cdot)$ denote state processes of the leader $\mathcal{A}_0$ and the i-th follower $\mathcal{A}_i$, respectively.  $u_0(\cdot):=\mathbf{col}\left[u_{01}(\cdot),\cdots,u_{0N}(\cdot)\right]$ represents control process of the leader $\mathcal{A}_0$, where $u_{0i}(\cdot)\in\mathbb{R}^{m_L}$ is the leader's control input for the i-th follower. $u_{1i}(\cdot)\in\mathbb{R}^{m_F}$ is the control strategy of the i-th follower. Let $u_1(\cdot):=\mathbf{col}\left[u_{11}(\cdot),\cdots,u_{1N}(\cdot)\right]$, and define the followers' state-average, leader's control-average and followers' control-average as $x^{(N)}(\cdot):=\frac{1}{N}\sum_{i=1}^{N}x_i(\cdot)$, $u_0^{(N)}(\cdot):=\frac{1}{N}\sum_{i=1}^{N}u_{0i}(\cdot)$ and $u_1^{(N)}(\cdot):=\frac{1}{N}\sum_{i=1}^{N}u_{1i}(\cdot)$, respectively. $v(\cdot)\in\mathbb{R}^{n_v}$ represents the external unknown disturbance of the leader $\mathcal{A}_0$, which is a drift uncertainty. Moreover, we use  $u_{-0i}(\cdot)=\mathbf{col}\left[u_{01}(\cdot),\cdots,u_{0i-1}(\cdot),u_{0i+1}(\cdot),\cdots,u_{0N}(\cdot)\right]$ to denote the strategy set of the leader $\mathcal{A}_0$ except $u_{0i}(\cdot)$, and $u_{-1i}(\cdot)=\mathbf{col}\left[u_{11}(\cdot),\cdots,u_{1i-1}(\cdot),u_{1i+1}(\cdot),\cdots,u_{1N}(\cdot)\right]$ to denote the followers' strategy set except the i-th follower $\mathcal{A}_i$. The
centralized strategy sets of the leader $\mathcal{A}_0$ are defined by
\begin{equation*}
	\mathcal{U}_{0c}:=\left\{(u_0(\cdot),u_1(\cdot))\Big|(u_0(\cdot),u_1(\cdot)) \mbox{ is }\mathcal{F}_t\mbox{-adapted}, \mathbb{E}\int_0^T |(u_0(t),u_1(t))|^2 dt < \infty\right\},
\end{equation*}
\begin{equation*}
	\mathcal{U}_{vc}:=\left\{v(\cdot)\Big|v(\cdot) \mbox{ is }\mathcal{F}_t\mbox{-adapted}, \mathbb{E}\int_0^T |v(t)|^2 dt < \infty\right\},
\end{equation*}
the centralized control set of the i-th follower $\mathcal{A}_i$ is given by
\begin{equation*}
 	\mathcal{U}_{ic}:=\left\{u_{1i}(\cdot)\Big|u_{1i}(\cdot) \mbox{ is }\mathcal{F}_t\mbox{-adapted}, \mathbb{E}\int_0^T |u_{1i}(t)|^2 dt < \infty\right\}.
\end{equation*}
The admissible decentralized strategy sets are defined as follows:
\begin{equation*}
	\mathcal{U}_{0d}:=\left\{(u_0(\cdot),u_1(\cdot))\Big|(u_0(\cdot),u_1(\cdot)) \mbox{ is }\mathcal{G}_t^0\mbox{-adapted}, \mathbb{E}\int_0^T |(u_0(t),u_1(t))|^2 dt < \infty\right\},
\end{equation*}
\begin{equation*}
	\mathcal{U}_{vd}:=\left\{v(\cdot)\Big|v(\cdot) \mbox{ is }\mathcal{G}_t^0\mbox{-adapted and}\; \mathbb{E}\int_0^T |v(t)|^2 dt < \infty\right\},
\end{equation*}
\begin{equation*}
	\mathcal{U}_{id}:=\left\{u_{1i}(\cdot)\Big|u_{1i}(\cdot) \mbox{ is }\mathcal{G}_t^i\mbox{-adapted}, \mathbb{E}\int_0^T |u_{1i}(t)|^2 dt < \infty\right\}.
\end{equation*}

The cost functions of the leader $\mathcal{A}_0$ and the i-th follower $\mathcal{A}_i$ are supposed to be
\begin{equation}\label{cost0}
\begin{aligned}
	&\mathcal{J}_0(u_0(\cdot),u_1(\cdot),v(\cdot))\\
    &=\mathbb{E}\bigg\{\int_0^T\left[ \left\vert x_0-\Gamma_1x^{(N)}\right\vert_Q^2+\left\vert u_0^{(N)}\right\vert_{R_0}^2+\left\vert u_1^{(N)}\right\vert_{R_1}^2\right](t)dt
    +\left\vert x_0(T)-\Gamma_2x^{(N)}(T)\right\vert_G^2 \bigg\},
\end{aligned}
\end{equation}
\begin{equation}\label{costi}
\begin{aligned}
	&\mathcal{J}_i(u_0(\cdot),u_1(\cdot),v(\cdot))\\
    &=\mathbb{E}\bigg\{\int_0^T\left[ \left\vert x_i-\tilde{\Gamma}_1x^{(N)}\right\vert_{\tilde{Q}}^2+\left\vert u_{0i}\right\vert_{\tilde{R}_0}^2+\left\vert u_{1i}\right\vert_{\tilde{R}_1}^2\right](t)dt+\left\vert x_i(T)-\tilde{\Gamma}_2x^{(N)}(T)\right\vert_{\tilde{G}}^2 \bigg\}.
\end{aligned}
\end{equation}

\begin{Remark}
(\romannumeral1) When there exists a large number of homogeneous followers in the system, the leader's strategy is no longer to tailor control or incentive scheme for each individual follower. Instead, the leader influences all followers by designing macroscopic rules, prices, or incentive mechanisms. Therefore, the leader’s averaged control input for the followers act as a direct control input for the leader’s own state (see \cite{Mukaidani22,Ozai25}).

(\romannumeral2) If the followers' control-average $u_1^{(N)}$ enters the diffusion term of the state equation of the leader $\mathcal{A}_0$, it will make the subsequent derivation formally more complicated. In the present paper, we consider the form of the leader's state equation (\ref{state0}). We will consider the general case in the future.

(\romannumeral3) For the state equation of the i-th follower $\mathcal{A}_i$, we can only consider the case of additive noise. If the state equation (\ref{statei}) also includes state-dependent and control-dependent diffusion terms, the law of large numbers will fail, and we cannot obtain the equation satisfied by the limit term, which renders the limiting leader's problem unsolvable. We wish to overcome this difficulty in the future.
\end{Remark}

For the coefficients of (\ref{state0})-(\ref{costi}), we give the following assumptions.

\textbf{(A1)} $A(\cdot)$, $F(\cdot)$, $C(\cdot)$, $\tilde{A}(\cdot)$, $\tilde{F}(\cdot)$ $\in	L^\infty(0,T;\mathbb{R}^{n\times n})$;\ $B(\cdot)$, $D(\cdot)$, $\tilde{H}(\cdot)$ $\in L^\infty(0,T;\mathbb{R}^{n\times m_L})$;\
$H(\cdot)$, $\tilde{B}(\cdot)$ $\in L^\infty(0,T;\mathbb{R}^{n\times m_F})$;\ $E(\cdot)$ $\in L^\infty(0,T;\mathbb{R}^{n\times m_v})$; $\sigma(\cdot)\in L^\infty(0,T;\mathbb{R}^n)$.

\textbf{(A2)} $Q(\cdot)$, $\tilde{Q}(\cdot)$ $\in L^\infty(0,T;\mathbb{S}^{n})$;\ $R_0(\cdot)$, $\tilde{R}_0(\cdot)$ $\in L^\infty(0,T;\mathbb{S}^{m_L})$;\ $R_1(\cdot)$, $\tilde{R}_1(\cdot)$ $\in L^\infty(0,T;\mathbb{S}^{m_F})$;\ $\Gamma_1(\cdot)$, $\tilde{\Gamma}_1(\cdot)$$\in L^\infty(0,T;\mathbb{R}^{n\times n})$; $\Gamma_2$, $G$, $\tilde{\Gamma}_2$, $\tilde{G}$ $\in\mathbb{S}^n$.

In $H_\infty$ control, for a fixed disturbance attenuation level $\gamma>0$, the leader $\mathcal{A}_0$ wants to find a strategy $(u_0(\cdot),u_1(\cdot))$ to satisfy
\begin{equation*}
	\frac{\mathcal{J}_0(u_0(\cdot),u_1(\cdot),v(\cdot))}{\mathbb{E}\int_0^T\vert v(t)\vert_{R_2}^2dt}\leq\gamma^2,
\end{equation*}
for all non-zero $v(\cdot)\in L_{\mathcal{F}}^2(0,T;\mathbb{R}^{n_v})$.

Therefore, we redefine the cost functional (\ref{cost0}) as the following zero-sum cost functional for the leader $\mathcal{A}_0$ and the external disturbance $v(\cdot)$:
\begin{equation}\label{cost0'}
\begin{aligned}
	&\mathcal{J}_0(u_0(\cdot),u_1(\cdot),v(\cdot))\\
&=\mathbb{E}\bigg\{\int_0^T\left[ \left\vert x_0-\Gamma_1x^{(N)}\right\vert_Q^2+\left\vert u_0^{(N)}\right\vert_{R_0}^2+\left\vert u_1^{(N)}\right\vert_{R_1}^2-\gamma^2\vert v\vert_{R_2}^2\right](t)dt+\left\vert x_0(T)-\Gamma_2x^{(N)}(T)\right\vert_G^2 \bigg\}.
\end{aligned}
\end{equation}

Under the above assumptions \textbf{(A1)}-\textbf{(A2)}, for any $(u_0(\cdot),u_1(\cdot))\in\mathcal{U}_{0c}$, $v(\cdot)\in\mathcal{U}_{vc}$, the state system (\ref{state0})-(\ref{statei}) admit unique solutions, then the cost
functionals (\ref{costi})-(\ref{cost0'}) are well-defined.

Now, we formulate the following robust incentive Stackelberg games with large-population problem.

Problem \textbf{(L1)}. To find a centralized saddle point $(u^*(\cdot),v^*(\cdot))=(u_0^*(\cdot),u_1^*(\cdot),v^*(\cdot))\in\mathcal{U}_{0c}\times\mathcal{U}_{vc}$, such that
\begin{equation*}
	\mathcal{J}_0(u_0^*(\cdot),u_1^*(\cdot),v^*(\cdot))=\inf_{(u_0(\cdot),u_1(\cdot))\in\,\mathcal{U}_{0c}}\sup_{v(\cdot)\in\,\mathcal{U}_{vc}}\mathcal{J}_0(u_0(\cdot),u_1(\cdot),v(\cdot)),
\end{equation*}
subjects to (\ref{state0})-(\ref{statei}) and (\ref{cost0'}). Here, the cost functional of the leader $\mathcal{A}_0$ under the worst-case disturbance is
\begin{equation*}
	\mathcal{J}_0^{wo}(u_0(\cdot),u_1(\cdot))=\sup_{v(\cdot)\in\,\mathcal{U}_{vc}}\mathcal{J}_0(u_0(\cdot),u_1(\cdot),v(\cdot)).
\end{equation*}

Problem \textbf{(F1)}. To find a centralized strategy set $u_1^+(\cdot)=(u_{11}^+(\cdot),\cdots,u_{1N}^+(\cdot))$, where $u_{1i}^+(\cdot) \in\mathcal{U}_{ic}$, such that
\begin{equation*}
\begin{aligned}
    &\mathcal{J}_i\left(u_{0i}(u_{1i}^+)(\cdot),u_{-0i}(u_{-1i}^+)(\cdot),u_{1i}^+(\cdot),u_{-1i}^+(\cdot)\right)\\
    &=\inf_{u_{1i}(\cdot)\in\,\mathcal{U}_{ic}}\mathcal{J}_i\left(u_{0i}(u_{1i})(\cdot),u_{-0i}(u_{-1i}^+)(\cdot),u_{1i}(\cdot),u_{-1i}^+(\cdot)\right),
\end{aligned}
\end{equation*}
for $1\leq i\leq N$, subjects to (\ref{state0})-(\ref{statei}) and (\ref{costi}). $u_{0i}(u_{1i})(\cdot)$ denotes the leader's incentive strategy for the i-th follower, whose explicit form is presented in Section 4. Moreover, $u_{-0i}(u_{-1i})(\cdot)$ denotes the leader’s control profile which does not include $u_{0i}(u_{1i})(\cdot)$.

The system (\ref{state0})-(\ref{statei}) are fully coupled due to the coupling state-average and control-average, thus each agent should access all the information to solve his/her optimization problem. However, there are some difficulties making it unsolvable. Firstly, in many practical applications, each agent can only access his/her own information, while the information of other agents may be inaccessible. Secondly, due to the complicated coupling mechanism, the dynamic optimization is plagued by the curse of dimensionality and complexity. Inspired by \cite{Sanjari25} and combined with MFG theory, we present the following definitions.

\begin{mydef}
A set of strategies $(\hat{u}(\cdot),\hat{v}(\cdot))=(\hat{u}_0(\cdot),\hat{u}_1(\cdot),\hat{v}(\cdot))\in\mathcal{U}_{0c}\times\mathcal{U}_{vc}$ constitutes an $\epsilon_0$-leader robust team-optimal solution with respect to the cost functional $\mathcal{J}_0$ if
\begin{equation*}
	\bigg\vert \mathcal{J}_0(\hat{u}_0(\cdot),\hat{u}_1(\cdot),\hat{v}(\cdot))-\inf_{(u_0(\cdot),u_1(\cdot))\in\,\mathcal{U}_{0c}}\sup_{v(\cdot)\in\,\mathcal{U}_{vc}}\mathcal{J}_0(u_0(\cdot),u_1(\cdot),v(\cdot))\bigg\vert=O(\epsilon_0).
\end{equation*}
\end{mydef}

\begin{mydef}
The strategy profile $\check{u}_1(\cdot)=(\check{u}_{11}(\cdot),\cdots,\check{u}_{1N}(\cdot))$, where $\check{u}_{1i}(\cdot)\in\mathcal{U}_{ic}$, $1\leq i\leq N$, is called an $\hat{\epsilon}$-followers' Nash equilibrium with respect to the cost functional $\mathcal{J}_i$ if there exists an $\hat{\epsilon}=\hat{\epsilon}(N)\geq0$, $\lim\limits_{N\to\infty}\hat{\epsilon}(N)=0$, such that
\begin{equation*}
	\mathcal{J}_i({u_{0i}(\check{u}_{1i})(\cdot),u_{-0i}(\check{u}_{-1i})(\cdot)},\check{u}_{1i}(\cdot),\check{u}_{-1i}(\cdot))\leq\mathcal{J}_i({u_{0i}(u_{1i})(\cdot),u_{-0i}(\check{u}_{-1i})(\cdot)},u_{1i}(\cdot),\check{u}_{-1i}(\cdot))+\hat{\epsilon},
\end{equation*}
where $u_{1i}(\cdot)\in\mathcal{U}_{ic}$ is any alternative strategy applied by the i-th follower $\mathcal{A}_i$.
\end{mydef}

\begin{mydef}\label{def}
Given $\epsilon=(\epsilon_0,\hat{\epsilon})\geq0$, $(\hat{u}(\cdot),\hat{v}(\cdot))=(\hat{u}_0(\cdot),\hat{u}_1(\cdot),\hat{v}(\cdot))\in\mathcal{U}_{0c}\times\mathcal{U}_{vc}$ is called an $\epsilon$-robust incentive strategy if
	
(\romannumeral1) $(\hat{u}(\cdot),\hat{v}(\cdot))$ is an $\epsilon_0$-leader robust team-optimal strategy.
	
(\romannumeral2) $\bar{\hat{u}}_1(\cdot)=\bar{\check{u}}_1(\cdot)$, where $\check{u}_1(\cdot)$ is an $\hat{\epsilon}$-followers' Nash equilibrium, $\bar{\hat{u}}_1(\cdot)=\lim\limits_{N\to\infty}\frac{1}{N}\sum_{i=1}^{N}\hat{u}_{1i}(\cdot)$, $\bar{\check{u}}_1(\cdot)=\lim\limits_{N\to\infty}\frac{1}{N}\sum_{i=1}^{N}\check{u}_{1i}(\cdot)$.
\end{mydef}

\begin{Remark}
Intuitively understanding Definition \ref{def} (\romannumeral2): when the limiting behavior $\bar{\check{u}}_1(\cdot)$ of the $\hat{\epsilon}$-followers' Nash equilibrium  coincides with the limiting behavior $\bar{\hat{u}}_1(\cdot)$ of the followers' control strategies in the $\epsilon_0$-leader robust team-optimal strategy, it indicates that the leader has achieved the incentive objective. Then the $\epsilon_0$-leader robust team-optimal strategy is an $\epsilon$-robust incentive strategy.
\end{Remark}

\section{The limiting leader's team-optimal problem}

As mentioned in the previous section, it is infeasible for each agent to obtain the centralized
strategies in the large-population system. Alternatively, we use MFG theory to derive the decentralized strategy through an auxiliary limiting problem, where the state-average and control-average should be fixed to their frozen limit terms. When $N\to\infty$, suppose that the followers' state-average $x^{(N)}(\cdot)$, leader's control-average $u_0^{(N)}(\cdot)$ and followers' control-average $u_1^{(N)}(\cdot)$ are approximated by the process
$m(\cdot)$, $\bar{u}_0(\cdot)$ and $\bar{u}_1(\cdot)$, respectively, which will be determined later by the CC system. We introduce the following auxiliary state for the leader $\mathcal{A}_0$ and the i-th follower $\mathcal{A}_i$:
\begin{equation}\label{limstate0}
\left\{\begin{aligned}
	dx_0(t)&=\left[A(t)x_0(t)+B(t)\bar{u}_0(t)+F(t)m(t)+H(t)\bar{u}_1(t)+E(t)v(t)\right]dt\\
	&\quad +\left[C(t)x_0(t)+D(t)\bar{u}_0(t)+O(t)m(t)\right]dW_0(t),\\
	\dot{m}(t)&=\left[\tilde{A}(t)+\tilde{F}(t)\right]m(t)+\tilde{B}(t)\bar{u}_1(t)+\tilde{H}(t)\bar{u}_0(t),\\
	x_0(0)&=\xi,\quad m(0)=x,
\end{aligned}\right.
\end{equation}
\begin{equation}\label{limstatei}
\left\{\begin{aligned}
	dx_i(t)&=\left[\tilde{A}(t)x_i(t)+\tilde{B}(t)u_{1i}(t)+\tilde{F}(t)m(t)+\tilde{H}(t)u_{0i}(t)\right]dt+\sigma(t)dW_i(t),\\
	\dot{m}(t)&=\left[\tilde{A}(t)+\tilde{F}(t)\right]m(t)+\tilde{B}(t)\bar{u}_1(t)+\tilde{H}(t)\bar{u}_0(t),\\
	x_i(0)&=x,\quad m(0)=x,
\end{aligned}\right.
\end{equation}
and the limiting cost functionals are given by
\begin{equation}\label{limcost0}
\begin{aligned}
	&J_0(\bar{u}_0(\cdot),\bar{u}_1(\cdot),v(\cdot))\\
&=\mathbb{E}\bigg\{\int_0^T\left[ \vert x_0-\Gamma_1m\vert_Q^2+\vert \bar{u}_0\vert_{R_0}^2+\vert \bar{u}_1\vert_{R_1}^2-\gamma^2\vert v\vert_{R_2}^2\right](t)dt+\vert x_0(T)-\Gamma_2m(T)\vert_G^2 \bigg\},
\end{aligned}
\end{equation}
\begin{equation}\label{limcosti}
\begin{aligned}
	J_i(u_{0i}(\cdot),u_{1i}(\cdot))&=\mathbb{E}\bigg\{\int_0^T\left[ \left\vert x_i-\tilde{\Gamma}_1m\right\vert_{\tilde{Q}}^2+\vert u_{0i}\vert_{\tilde{R}_0}^2+\vert u_{1i}\vert_{\tilde{R}_1}^2\right](t)dt+\left\vert x_i(T)-\tilde{\Gamma}_2m(T)\right\vert_{\tilde{G}}^2 \bigg\}.
\end{aligned}
\end{equation}

We formalize the auxiliary limiting problem for the leader $\mathcal{A}_0$ and the i-th follower $\mathcal{A}_i$ as follows.

Problem \textbf{(L2)}. To find a decentralized saddle point $(\bar{u}^*(\cdot),v^*(\cdot))=(\bar{u}_0^*(\cdot),\bar{u}_1^*(\cdot),v^*(\cdot))\in L_{\mathcal{G}^0}^2(0,T;\mathbb{R}^{m_L+m_F})\times\mathcal{U}_{vd}$, such that
\begin{equation*}
	J_0(\bar{u}_0^*(\cdot),\bar{u}_1^*(\cdot),v^*(\cdot))=\inf_{(\bar{u}_0(\cdot),\bar{u}_1(\cdot))\in\, L_{\mathcal{G}^0}^2(0,T;\mathbb{R}^{m_L+m_F})}\sup_{v(\cdot)\in\,\mathcal{U}_{vd}}J_0(\bar{u}_0(\cdot),\bar{u}_1(\cdot),v(\cdot)),
\end{equation*}
subjects to (\ref{limstate0}) and (\ref{limcost0}).

Problem \textbf{(F2)}. To find a decentralized strategy set $u_1^+(\cdot)=(u_{11}^+(\cdot),\cdots,u_{1N}^+(\cdot))$, where $u_{1i}^+(\cdot) \in\mathcal{U}_{id}$, such that
\begin{equation*}
\begin{aligned}
	&J_i\left(u_{0i}(u_{1i}^+)(\cdot),u_{-0i}(u_{-1i}^+)(\cdot),u_{1i}^+(\cdot),u_{-1i}^+(\cdot)\right)\\
&=\inf_{u_{1i}(\cdot)\in\,\mathcal{U}_{id}}J_i\left(u_{0i}(u_{1i})(\cdot),u_{-0i}(u_{-1i}^+)(\cdot),u_{1i}(\cdot),u_{-1i}^+(\cdot)\right),
\end{aligned}
\end{equation*}
for $1\leq i\leq N$, subjects to (\ref{limstatei}) and (\ref{limcosti}).

\begin{mydef}
\begin{itemize}
\item[(a)] For given $\xi\in\mathbb{R}^n$, Problem \textbf{(L2)} is said to be finite if the value function of Problem \textbf{(L2)} $V_0(\xi,x)$ is bounded from above and below, that is,
\begin{equation*}
	-\infty<V_0(\xi,x):=\inf_{(\bar{u}_0(\cdot),\bar{u}_1(\cdot))\in\, L_{\mathcal{G}^0}^2(0,T;\mathbb{R}^{m_L+m_F})}\sup_{v(\cdot)\in\,\mathcal{U}_{vd}}J_0(\bar{u}_0(\cdot),\bar{u}_1(\cdot),v(\cdot))<\infty.
\end{equation*}
\item[(b)] Problem \textbf{(L2)} is said to be (uniquely) solvable if there exists a (unique) saddle point $(\bar{u}^*(\cdot),v^*(\cdot))\in L_{\mathcal{G}^0}^2(0,T;\mathbb{R}^{m_L+m_F})\times \mathcal{U}_{vd}$ such that $J_0(\bar{u}_0^*(\cdot),\bar{u}_1^*(\cdot),v^*(\cdot))=V_0(\xi,x)$.
\end{itemize}
\end{mydef}

We will give a representation of cost functional $J_0(\bar{u}_0(\cdot),\bar{u}_1(\cdot),v(\cdot))$ for Problem \textbf{(L2)} based on duality theory to obtain basic properties.

\begin{mypro}\label{Prop3.1}
Let \textbf{(A1)}-\textbf{(A2)} hold. There exist two bounded self-adjoint linear operators $M_1:L_{\mathcal{G}^0}^2(0,T;\mathbb{R}^{m_L+m_F})\to L_{\mathcal{G}^0}^2(0,T;\mathbb{R}^{m_L+m_F})$, $M_4:\mathcal{U}_{vd}\to \mathcal{U}_{vd}$, bounded operator $M_2:\mathcal{U}_{vd}\to L_{\mathcal{G}^0}^2(0,T;\mathbb{R}^{m_L+m_F})$, $M_3:\mathbb{R}^{2n}\to L_{\mathcal{G}^0}^2(0,T;\mathbb{R}^{m_L+m_F})$, $M_5:\mathbb{R}^{2n}\to \mathcal{U}_{vd}$ and some $M_0\in\mathbb{R}$, depending on $\bar{\xi}$, such that
\begin{equation}\label{rep}
\begin{aligned}
	J_0(\bar{u}(\cdot),v(\cdot))&=\langle M_1(\bar{u})(\cdot),\bar{u}(\cdot)\rangle+2\langle M_2(v)(\cdot),\bar{u}(\cdot)\rangle+2\langle M_3(\bar{\xi}(\cdot),\bar{u}(\cdot))\\
    &\quad +\langle M_4(v)(\cdot),v(\cdot)\rangle+2\langle M_5(\bar{\xi})(\cdot),v(\cdot)\rangle+M_0(\bar{\xi}),
\end{aligned}
\end{equation}
where
\begin{equation*}
\begin{aligned}
	&\bar{u}(\cdot)=\mathbf{col}\left[\bar{u}_0(\cdot),\bar{u}_1(\cdot)\right],\quad \bar{\xi}=\mathbf{col}[\xi,x],\quad M_0(\bar{\xi})=\langle y^2(0),\xi\rangle+\langle p^2(0),x\rangle ,\\
\end{aligned}
\end{equation*}
\begin{equation*}
\begin{aligned}
	&M_1(\bar{u})(\cdot)=\begin{pmatrix}
		B^\top(\cdot)y^1(\cdot)+D^\top(\cdot)z^1(\cdot)+\tilde{H}^\top(\cdot)p^1(\cdot)+R_0(\cdot)\bar{u}_0(\cdot) \\
		H^\top(\cdot)y^1(\cdot)+\tilde{B}^\top(\cdot)p^1(\cdot)+R_1(\cdot)\bar{u}_1(\cdot) \\
	\end{pmatrix},\\
	&M_2(v)(\cdot)=\begin{pmatrix}
		B^\top(\cdot)y^3(\cdot)+D^\top(\cdot)z^3(\cdot)+\tilde{H}^\top(\cdot)p^3(\cdot) \\
		H^\top(\cdot)y^3(\cdot)+\tilde{B}^\top(\cdot)p^3(\cdot) \\
	\end{pmatrix},\\
	&M_3(\bar{\xi})(\cdot)=\begin{pmatrix}
		B^\top(\cdot)y^2(\cdot)+D^\top(\cdot)z^2(\cdot)+\tilde{H}^\top(\cdot)p^2(\cdot) \\
		H^\top(\cdot)y^2(\cdot)+\tilde{B}^\top(\cdot)p^2(\cdot) \\
	\end{pmatrix},\\ &M_4(v)(\cdot)=E^\top(\cdot)y^3(\cdot)-\gamma^2R_2(\cdot)v(\cdot),\quad M_5(\bar{\xi})(\cdot)=E^\top(\cdot)y^2(\cdot),
\end{aligned}
\end{equation*}
with $y^i(\cdot)$, $z^i(\cdot)$, $p^i(\cdot)$, $i=1,2,3$, satisfy the following backward-forward stochastic systems:
\begin{equation}
\left\{\begin{aligned}
	dx_0^1&=\left[Ax_0^1+B\bar{u}_0+Fm^1+H\bar{u}_1\right]dt+\left[Cx_0^1+D\bar{u}_0+Om^1\right]dW_0,\\
	\dot{m}^1&=\left(\tilde{A}+\tilde{F}\right)m^1+\tilde{B}\bar{u}_1+\tilde{H}\bar{u}_0,\\
	dy^1&=-\left[A^\top y^1+C^\top z^1+Qx_0^1-Q\Gamma_1 m^1\right]dt+z^1dW_0,\\
	dp^1&=-\left[\left(\tilde{A}+\tilde{F}\right)^\top p^1+F^\top y^1+O^\top z^1+\Gamma_1^\top Q\Gamma_1 m^1-\Gamma_1^\top Q x_0^1\right]dt+q^1dW_0,\\
	x_0^1(0)&=0,\quad m(0)=0,\\
    y^1(T)&=Gx_0^1(T)-G\Gamma_2m^1(T),\quad p^1(T)=\Gamma_2^\top G\Gamma_2m^1(T)-\Gamma_2^\top Gx_0^1(T),
\end{aligned}\right.
\end{equation}
\begin{equation}
\left\{\begin{aligned}
	dx_0^2&=\left[Ax_0^2+Fm^2\right]dt+\left[ Cx_0^2+Om^2\right]dW_0,\quad \dot{m}^2=\left(\tilde{A}+\tilde{F}\right)m^2,\\
    dy^2&=-\left[A^\top y^2+C^\top z^2+Qx_0^2-Q\Gamma_1 m^2\right]dt+z^2dW_0,\\
    dp^2&=-\left[\left(\tilde{A}+\tilde{F}\right)^\top p^2+F^\top y^2+O^\top z^2+\Gamma_1^\top Q\Gamma_1 m^2-\Gamma_1^\top Q x_0^2\right]dt+q^2dW_0,\\
    x_0^2(0)&=\xi,\quad m^2(0)=x,\\
    y^2(T)&=Gx_0^2(T)-G\Gamma_2m^2(T),\quad p^2(T)=\Gamma_2^\top G\Gamma_2m^2(T)-\Gamma_2^\top Gx_0^2(T),
\end{aligned}\right.
\end{equation}
\begin{equation}
\left\{\begin{aligned}
	dx_0^3&=\left[Ax_0^3+Ev\right]dt+Cx_0^3dW_0,\\
	dy^3&=-\left[A^\top y^3+C^\top z^3+Qx_0^3\right]dt+z^3dW_0,\\
	dp^3&=-\left[\left(\tilde{A}+\tilde{F}\right)^\top p^3+F^\top y^3+O^\top z^3-\Gamma_1^\top Q x_0^3\right]dt+q^3dW_0,\\
	x_0^3(0)&=0,\quad y^3(T)=Gx_0^3(T),\quad p^3(T)=-\Gamma_2^\top Gx_0^3(T).
\end{aligned}\right.
\end{equation}
\end{mypro}

\begin{proof}
The proof is similar as Chapter 6 of \cite{Yong99}, by using duality theory, we omit the details.
\end{proof}

For notational simplicity, we use $\langle\cdot,\cdot\rangle$ to denote all inner products in different Hilbert spaces which can be identified from the context. From the representation (\ref{rep}) of Proposition \ref{Prop3.1}, one has the following result, which is concerned with the convexity of the cost functional $J_0$ and the solvability of Problem \textbf{(L2)}.

\begin{mypro}\label{Prop3.2}
Let \textbf{(A1)}-\textbf{(A2)} hold.
\begin{itemize}
\item[(a)] Problem \textbf{(L2)} is finite only if $M_1\geq0$, $M_4\leq0$.
\item[(b)] Problem \textbf{(L2)} is (uniquely) solvable if and only if Problem \textbf{(L2)} satisfies the convexity-concavity condition $M_1\geq0$, $M_4\leq0$ and the following stationarity condition holds: there exists a (unique) $(\bar{u}^*(\cdot),v^*(\cdot))\in L_{\mathcal{G}^0}^2(0,T;\mathbb{R}^{m_L+m_F})\times \mathcal{U}_{vd}$ such that
\begin{equation}\label{staionary1}
\left\{
\begin{aligned}
	&M_1(\bar{u}^*)(\cdot)+M_2(v^*)(\cdot)+M_3(\bar{\xi})(\cdot)=0,\\
	&M_4(v^*)(\cdot)+M_2^*(\bar{u}^*)(\cdot)+M_5(\bar{\xi})(\cdot)=0.
\end{aligned}
\right.
\end{equation}
Moreover, (\ref{staionary1}) implies that $\mathcal{R}(M_2(v^*)+M_3(\bar{\xi}))\subset\mathcal{R}(M_1(\bar{u}^*))$, $\mathcal{R}(M_2^*(\bar{u}^*)+M_5(\bar{\xi}))\subset\mathcal{R}(M_4(v^*))$, where $\mathcal{R}(S)$ stands for the range of operator matrix $S$.
\item[(c)] Problem \textbf{(L2)} is uniformly convex-concave (i.e., $M_1\gg0$, $M_4\ll0$), then Problem \textbf{(L2)} admits a unique saddle point given by
\begin{equation}\label{rep'}
\left\{
\begin{aligned}
	\bar{u}^*(\cdot)&=-M_1^{-1}(M_2(v^*)+M_3(\bar{\xi}))(\cdot),\\
    v^*(\cdot)&=-M_4^{-1}(M_2^*(\bar{u}^*)+M_5(\bar{\xi}))(\cdot).
\end{aligned}
\right.
\end{equation}
Moreover,
\begin{equation}\label{rep''}
\left\{
\begin{aligned}
	\bar{u}^*(\cdot)&=-(M_1-M_2M_4^{-1}M_2^*)^{-1}(M_3-M_2M_4^{-1}M_5)(\bar{\xi})(\cdot),\\
	v^*(\cdot)&=[M_4^{-1}M_2^*(M_1-M_2M_4^{-1}M_2^*)^{-1}(M_3-M_2M_4^{-1}M_5)-M_4^{-1}M_5](\bar{\xi})(\cdot).
\end{aligned}
\right.
\end{equation}
\end{itemize}
\end{mypro}

\begin{proof}
The first result can be proved by contradiction. The proof is regular, thus we omit it here.
Let us prove the second result. If Problem \textbf{(L2)} is uniquely solvable with the saddle point $(\bar{u}^*(\cdot),v^*(\cdot))$. By the representations (\ref{rep}) of the cost functional $J_0$, for any $\lambda\in\mathbb{R}$ and $(\bar{u}(\cdot),v(\cdot))\in L_{\mathcal{G}^0}^2(0,T;\mathbb{R}^{m_L+m_F})\times \mathcal{U}_{vd}$, we have
\begin{equation*}
\begin{aligned}
	&J_0\left(\bar{u}^*(\cdot)+\lambda\bar{u}(\cdot),v^*(\cdot)\right)=\langle M_1(\bar{u}^*+\lambda\bar{u}),\bar{u}^*+\lambda\bar{u}\rangle+2\langle M_2(v^*),\bar{u}^*+\lambda\bar{u}\rangle\\
	&\qquad +2\langle M_3(\bar{\xi}),\bar{u}^*+\lambda\bar{u}\rangle+\langle M_4(v^*),v^*\rangle+2\langle M_5(\bar{\xi}),v^*\rangle+M_0(\bar{\xi})\\
	&=J_0(\bar{u}^*(\cdot),v^*(\cdot))+\lambda^2\langle M_1(\bar{u}),\bar{u}\rangle+2\lambda\langle M_1(\bar{u}^*)+M_2(v^*)+M_3(\bar{\xi}),\bar{u}\rangle.
\end{aligned}
\end{equation*}
Therefore, $\bar{u}^*(\cdot)$ is an open-loop optimal control with the cost functional $J_0(\bar{u}(\cdot),v^*(\cdot))$ and the corresponding state equation if and only if
\begin{equation*}
    J_0(\bar{u}^*(\cdot)+\lambda\bar{u}(\cdot),v^*(\cdot))-J_0(\bar{u}^*(\cdot),v^*(\cdot))\geq0,\quad\forall\lambda\in\mathbb{R},\;\bar{u}(\cdot)\in L_{\mathcal{G}^0}^2(0,T;\mathbb{R}^{m_L+m_F}).
\end{equation*}
Thus,
\begin{equation*}
	\lambda^2\langle M_1(\bar{u}),\bar{u}\rangle+2\lambda\langle M_1(\bar{u}^*)+M_2(v^*)+M_3(\bar{\xi}),\bar{u}\rangle
\end{equation*}
is a nonnegative and quadratic function of $\lambda$. Since $\bar{u}(\cdot)$ is arbitrary, we must have
\begin{equation*}
	M_1\geq0,\quad M_1(\bar{u}^*)+M_2(v^*)+M_3(\bar{\xi})=0.
\end{equation*}
Similarly,
\begin{equation*}
\begin{aligned}
	&J_0(\bar{u}^*(\cdot),v^*(\cdot)+\lambda v(\cdot))\\
	&=\langle M_1(\bar{u}^*),\bar{u}^*\rangle+2\langle M_2(v^*+\lambda v),\bar{u}^*\rangle+2\langle M_3(\bar{\xi}),\bar{u}^*\rangle+\langle M_4(v^*+\lambda v),v^*+\lambda v\rangle\\
	&\qquad +2\langle M_5(\bar{\xi}),v^*+\lambda v\rangle+M_0(\bar{\xi})\\
	&=J_0(\bar{u}^*(\cdot),v^*(\cdot))+\lambda^2\langle M_4(v),v\rangle+2\lambda\langle M_4(v^*)+M_2^*(\bar{u}^*)+M_5(\bar{\xi}),v\rangle,
\end{aligned}
\end{equation*}
where $M_2^*$ denotes the adjoint operators of bounded operator $M_2$. According to the definition of open-loop saddle point, $v^*(\cdot)$ is an open-loop optimal control for fixed $\bar{u}^*(\cdot)$ if and only if
\begin{equation*}
	J_0(\bar{u}^*(\cdot),v^*(\cdot)+\lambda v(\cdot))-J_0(\bar{u}^*(\cdot),v^*(\cdot))\leq0,\quad\forall\lambda\in\mathbb{R},\;v(\cdot)\in \mathcal{U}_{vd}.
\end{equation*}
Thus,
\begin{equation*}
	\lambda^2\langle M_4(v),v\rangle+2\lambda\langle M_4(v^*)+M_2^*(\bar{u}^*)+M_5(\bar{\xi}),v\rangle\leq0,\quad\forall\lambda\in\mathbb{R},\;v(\cdot)\in \mathcal{U}_{vd}.
\end{equation*}
Thus, we must have
\begin{equation*}
	M_4\leq0,\quad M_1(\bar{u}^*)+M_2(v^*)+M_3(\bar{\xi})=0.
\end{equation*}
Hence, the second result is derived.

Moreover, if $M_1\gg0$, $M_4\ll0$ hold, and $J_0$ is continuous in $\bar{u}(\cdot)$ and $v(\cdot)$, there exists a unique $(\bar{u}^*(\cdot),v^*(\cdot))$ such that $J_0(\bar{u}^*(\cdot),v^*(\cdot))=\inf_{\bar{u}(\cdot)\in\, L_{\mathcal{G}^0}^2(0,T;\mathbb{R}^{m_L+m_F})}J_0(\bar{u}(\cdot),v^*(\cdot))$ and $J_0(\bar{u}^*(\cdot),v^*(\cdot))=\sup_{v(\cdot)\in\, \mathcal{U}_{vd}}J_0(\bar{u}^*(\cdot),v(\cdot))$. By (\ref{rep}) of the second result, we can obtain (\ref{rep'}). Solving for $v^*(\cdot)$ from the second one (uniquely) of (\ref{staionary1}), and substituting it into the first one of (\ref{staionary1}), we obtain the single relation
\begin{equation*}
	\left(M_1-M_2M_4^{-1}M_2^*\right)(\bar{u}^*)+\left(M_3-M_2M_4^{-1}M_5\right)(\bar{\xi})=0,
\end{equation*}
which admits a unique solution, since $M_1-M_2M_4^{-1}M_2^*\gg0$ and thereby invertible. Then we get the first equation of (\ref{rep''}), substituting it into the second one of (\ref{rep'}), the representation of $v^*(\cdot)$ can be derived. The proof is complete.
\end{proof}

We give the following assumption.\\
\textbf{(A3)} $Q\geq0$, $G\geq0$, $R_0\gg0$, $R_1\gg0$, $R_2\gg0$.

\begin{Remark}
Under the assumption \textbf{(A3)}, $\bar{u}(\cdot)\mapsto J_0(\bar{u}(\cdot),v(\cdot))$ is uniformly convex (i.e., $M_1\gg0$).
\end{Remark}
Next, we would like to discuss the uniform concavity of the cost functional $J_0$ in $v$ for every fixed $\bar{u}(\cdot)\in L_{\mathcal{G}^0}^2(0,T;\mathbb{R}^{m_L+m_F})$, which is precisely the condition
of existence of a unique solution to the linear-quadratic optimal control problem:
\begin{equation*}
	\sup_{v(\cdot)\in\,\mathcal{U}_{vd}}J_0(\bar{u}(\cdot),v(\cdot)).
\end{equation*}

\begin{mycor}\label{cor3.1}
Let \textbf{(A1)}-\textbf{(A2)} hold. For $\bar{u}(\cdot)\in L_{\mathcal{G}^0}^2(0,T;\mathbb{R}^{m_L+m_F})$ and $\bar{\xi}\in\mathbb{R}^{2n}$, the following statements are equivalent:

(\romannumeral1) $v(\cdot)\mapsto J_0(\bar{u}(\cdot),v(\cdot))$ is uniformly concave.

(\romannumeral2) $M_4\ll0$.

(\romannumeral3) $J_0'(v(\cdot)):=\langle M_4(v)(\cdot),v(\cdot)\rangle=\mathbb{E}\int_0^T\left\langle E^\top y^3-\gamma^2R_2v,v\right\rangle(t)dt \leq -\alpha \mathbb{E}\int_0^T\vert v(t)\vert^2dt$, $\forall v(\cdot)\in \mathcal{U}_{vd}$, for some $\alpha>0$.
\end{mycor}

\begin{mypro}\label{Prop3.3}
For each fixed $\gamma>0$, the cost functional $J_0(\bar{u}(\cdot),v(\cdot))$, and under the state equation (\ref{limstate0}), is uniformly concave in $v(\cdot)\in \mathcal{U}_{vd}$ for $\bar{u}(\cdot)\in L_{\mathcal{G}^0}^2(0,T;\mathbb{R}^{m_L+m_F})$ if and only if the following generalized Riccati differential equation exists a unique symmetric nonnegative-definite solution
\begin{equation}\label{K}
\left\{\begin{aligned}
	&\dot{K}+KA+A^\top K+C^\top KC+Q+\gamma^{-2}KER_2^{-1}E^\top K=0,\\
	&K(T)=G.
\end{aligned}\right.
\end{equation}
\end{mypro}

\begin{proof}
In fact, the necessity is the {\it stochastic bounded real lemma} proved in Lemma 8.2.1 of \cite{Petersen00}.
Let us prove the sufficiency. Consider the following auxiliary cost functional:
\begin{equation*}
	J_0'(v(\cdot))=\mathbb{E}\int_0^T\left\langle E^\top y^3-\gamma^2R_2v,v\right\rangle(t) dt,
\end{equation*}
subject to
\begin{equation*}
\left\{\begin{aligned}
	dx_0^3(t)&=\left[Ax_0^3+Ev\right]dt+Cx_0^3dW_0(t),\\
	dy^3(t)&=-\left[A^\top y^3+C^\top z^3+Qx_0^3\right]dt+z^3dW_0(t),\\
	x_0^3(0)&=0,\quad y^3(T)=Gx_0^3(T).
\end{aligned}\right.
\end{equation*}
Applying It\^{o}'s formula to $\langle x_0^3(\cdot),y^3(\cdot)\rangle$, we can get
\begin{equation}\label{J_0''}
	J_0''(v(\cdot)):=-J_0'(v(\cdot))=\mathbb{E}\bigg\{\int_0^T\left[-\langle Q x_0^3,x_0^3\rangle+\gamma^2\langle R_2v,v\rangle\right](t) dt-\langle Gx_0^3(T),x_0^3(T)\rangle\bigg\},
\end{equation}
subject to
\begin{equation}
\left\{\begin{aligned}
	dx_0^3(t)&=\left[Ax_0^3+Ev\right]dt+Cx_0^3dW_0(t),\\
	x_0^3(0)&=0.
\end{aligned}\right.
\end{equation}
Since Riccati equation (\ref{K}) is solvable, applying It\^{o}'s formula to $\langle K(\cdot)x_0^3(\cdot),x_0^3(\cdot)\rangle$, integrating both sides on $[0, T]$, taking expectation and substituting it into (\ref{J_0''}), we have
\begin{equation*}
\begin{aligned}
	J_0''(v(\cdot))&=\gamma^2\mathbb{E}\int_0^T\left\vert v-\gamma^{-2}R_2^{-1}E^\top Kx_0^3\right\vert_{R_2}^2(t)dt\\
	&\geq\delta\gamma^2\mathbb{E}\int_0^T\left\vert v-\gamma^{-2}R_2^{-1}E^\top Kx_0^3\right\vert^2(t)dt,
\end{aligned}
\end{equation*}
where $R_2\geq\delta I$. Now we define a bounded linear operator $\mathcal{L}:\mathcal{U}_{vd}\mapsto \mathcal{U}_{vd}$ by
\begin{equation*}
	\mathcal{L}v:=v-\gamma^{-2}R_2^{-1}E^\top Kx_0^3.
\end{equation*}
It is easy to see that $\mathcal{L}$ is bijective, and that its inverse is given by
\begin{equation*}
	\mathcal{L}^{-1}v=v+\gamma^{-2}R_2^{-1}E^\top Kx_0^{(v)},
\end{equation*}
where $x_0^{(v)}$ is the solution to
\begin{equation*}
\left\{\begin{aligned}
	dx_0^{(v)}(t)&=\left[\left(A+\gamma^{-2}ER_2^{-1}E^\top K\right)x_0^{(v)}+Ev\right]dt+Cx_0^{(v)}dW_0(t),\\
	x_0^{(v)}(0)&=0.
\end{aligned}\right.
\end{equation*}
By the bounded inverse theorem, $\mathcal{L}^{-1}$ is bounded with $\Vert\mathcal{L}^{-1}\Vert>0$. Therefore,
\begin{equation*}
	\mathbb{E}\int_0^T\vert v(t)\vert^2dt=\mathbb{E}\int_0^T\vert( \mathcal{L}^{-1}\mathcal{L}v)(t)\vert^2dt\leq\Vert\mathcal{L}^{-1}\Vert^2\mathbb{E}\int_0^T\vert(\mathcal{L}v)(t)\vert^2dt.
\end{equation*}
Combining the two inequalities mentioned above, we obtain
\begin{equation*}
	J_0''(v(\cdot))\geq\delta\gamma^2\mathbb{E}\int_0^T\vert(\mathcal{L}v)(t)\vert^2dt\geq\frac{\delta\gamma^2}{\Vert\mathcal{L}^{-1}\Vert^2}\mathbb{E}\int_0^T\vert v(t)\vert^2dt,
\end{equation*}
for any $v(\cdot)\in \mathcal{U}_{vd}$, thus
\begin{equation*}
	J_0'(v(\cdot))=-J_0''(v(\cdot))\leq-\frac{\delta\gamma^2}{\Vert\mathcal{L}^{-1}\Vert^2}\mathbb{E}\int_0^T\vert v(t)\vert^2dt.
\end{equation*}
From Corollary \ref{cor3.1}, the cost functional $J_0$ is uniformly concave in $v$. The proof is complete.
\end{proof}

\begin{mycor}\label{cor3.2}
If (\ref{K}) admits a solution defined over $[0,T]$, then
\begin{equation*}
	\sup_{v(\cdot)\in\,\mathcal{U}_{vd}}J_0^{\gamma}(\bar{u}(\cdot),v(\cdot))<\infty.
\end{equation*}
\end{mycor}
In the above, since the cost functional $J_0$ is affected by the disturbance attenuation level $\gamma$, we have added a superscript to it for the convenience of subsequent discussions.

It is clear that the solution to the generalized Riccati equation (\ref{K}) depends on the parameter $\gamma>0$. We shall denote the solution to the equation by $K_\gamma(\cdot)$. (\ref{K}) has a solution in a left neighborhood of the terminal time $T$. However, it may have a finite escape time in $[0,T)$ \cite{Basar08} (its solution tends to infinity).

\begin{mypro}
For $\gamma$ being sufficiently large, the Riccati equation (\ref{K}) admits a solution over $[0,T]$.
\end{mypro}

\begin{proof}
Replace $\gamma^{-2}$ by $\epsilon$ in (\ref{K}). For $\epsilon=0$, this is a generalized Lyapunov differential equation and thus it admits a unique solution according to Lemma 7.3 of \cite{Yong99}. Moreover, the solution to (\ref{K}) is a continuous function of $\epsilon$, specially at zero. Thus it remains bounded for small enough values of $\epsilon$, and equivalently for large enough values of $\gamma$. The desired conclusion follows.
\end{proof}

From the above result, we define the following set, which is nonempty:
\begin{equation*}
	\hat{\Gamma}:=\left\{\tilde{\gamma}>0\,\big|\,\forall\gamma\geq\tilde{\gamma},\,\mbox{the Riccati equation (\ref{K}) admits a solution over}\,[0,T] \right\}.
\end{equation*}
Define $\hat{\gamma}$ as
\begin{equation}\label{hatgamma}
	\hat{\gamma}:=\inf\{\gamma:\gamma\in\hat{\Gamma}\}.
\end{equation}

\begin{mypro}\label{Prop3.5}
For $\gamma=\hat{\gamma}$, the Riccati equation (\ref{K}) has a finite escape time.
\end{mypro}

\begin{proof}
For any $(\gamma,t)$ for which $K_\gamma(t)$ exists, it is a continuous function of $\gamma$, and according to (\ref{K}), $\dot{K}_\gamma(t)$ is also a continuous function of $\gamma$. As a consequence, $K_\gamma$ is continuous in $\gamma$ uniformly in $t$ over $[0,T]$. As a matter of fact, for any $t\in[0,T]$ and $\epsilon>0$, $\exists$ $\delta(t)>0$, such that for any $\tilde{\gamma}\in(\gamma-\delta(t),\gamma+\delta(t))$
\begin{equation*}
	\Vert K_{\tilde{\gamma}}(t)-K_{\gamma}(t)\Vert<\frac{\epsilon}{3}.
\end{equation*}
For fixed $\gamma$, $K_\gamma(t)$ is a continuous function of $t$ over $[0,T]$. Thus, for the above $\epsilon>0$, $\exists$ $\delta_\gamma>0$, such that for any $t'\in(t-\delta_\gamma,t+\delta_\gamma)$, we have
\begin{equation*}
	\Vert K_{\gamma}(t')-K_{\gamma}(t)\Vert<\frac{\epsilon}{3},
\end{equation*}
and the same applies to $K_{\tilde{\gamma}}(t)$. Therefore, from the above two inequalities, we can get that for $\epsilon>0$, $\exists$ $\delta_0>0$, such that for any $t'\in(t-\delta_0,t+\delta_0)$,
\begin{equation*}
\begin{aligned}
	\Vert K_{\tilde{\gamma}}(t')-K_{\gamma}(t')\Vert&=\Vert K_{\tilde{\gamma}}(t')-K_{\tilde{\gamma}}(t)+K_{\tilde{\gamma}}(t)-K_{\gamma}(t)+K_{\gamma}(t)-K_{\gamma}(t')\Vert\\
	&\leq\Vert K_{\tilde{\gamma}}(t')-K_{\tilde{\gamma}}(t)\Vert+\Vert K_{\tilde{\gamma}}(t)-K_{\gamma}(t)\Vert+\Vert K_{\gamma}(t)-K_{\gamma}(t')\Vert<\epsilon.
\end{aligned}
\end{equation*}
This defines an open covering of the compact line segment $[0,T]$. There exits a finite covering. Extract it and pick $\delta>0$ as the minimum of the corresponding $\delta(t)$'s. For any $\tilde{\gamma}\in(\gamma-\delta,\gamma+\delta)$, $\forall$ $t\in[0,T]$, we have
\begin{equation*}
	\Vert K_{\tilde{\gamma}}(t)-K_{\gamma}(t)\Vert<\epsilon.
\end{equation*}
Thus, $K_\gamma(t)$ is continuous in $\gamma$ uniformly in $t$ over $[0,T]$, which provides an a priori bound for $K_{\tilde{\gamma}}(t)$ and thus ensures its existence over $[0,T]$. Then, the set $\{\gamma>0\,\big|\,K_\gamma \,\mbox{is defined over}\, [0,T]\}$ is open, and $\hat{\Gamma}$ is its connected component that contains $\infty$. Hence $\hat{\Gamma}$ is open, and the infimum of the open set $\hat{\Gamma}$ is not contained in the set, i.e., $\hat{\gamma}\notin\hat{\Gamma}$. The desired result then follows.
\end{proof}

\begin{Remark}
By Proposition \ref{Prop3.3} and Proposition \ref{Prop3.5}, we can show that if $\gamma>\hat{\gamma}$, the Riccati equation (\ref{K}) admits a non-negative solution, hence the cost functional $J_0^\gamma$ is uniformly concave in $v$. For all $\gamma\leq\hat{\gamma}$, (\ref{K}) has a finite escape time, and hence $J_0^\gamma$ is not uniformly concave in $v$.
\end{Remark}

Now we present the main result.

\begin{mythm}\label{thm3.1}
The cost functional $J_0^\gamma(\bar{u}(\cdot),v(\cdot))$ has a finite supremum in $v(\cdot)$ for all $\bar{u}(\cdot)$ and $\bar{\xi}$ if and only if $\gamma>\hat{\gamma}$.
\end{mythm}

\begin{proof}
Sufficiency. If $\gamma>\hat{\gamma}$, then $\gamma\in\hat{\Gamma}$. Assume the contrary, $\gamma$$\notin\hat{\Gamma}$, i.e., there exists a $\gamma_0\geq\gamma$ such that the equation (\ref{K}) has no solution over $[0,T]$. Thus, $\hat{\gamma}\geq\gamma_0\geq\gamma$, which contradicts $\gamma>\hat{\gamma}$. Therefore, for any $\gamma>\hat{\gamma}$, the Riccati equation (\ref{K}) admits a solution $K_\gamma(\cdot)$. The sufficiency has already been proven in Corollary \ref{cor3.2}.

Necessity. We prove it by contradiction. Suppose that $\gamma\leq\hat{\gamma}$. Let $\{\gamma_k\}_{k\geq0}$ be a monotonically decreasing sequence in $\mathbb{R}^+$ with limit point $\hat{\gamma}$ (i.e., $\gamma_k\downarrow\hat{\gamma}$), and use $K_k$ to denote $K_{\gamma_k}$. From Proposition \ref{Prop3.5}, there exists some $t^*\in[0,T)$ such that $\Vert K_k(t^*)\Vert\to\infty$ as $k\to\infty$, where $K_k(t^*)$ is nonnegative definite for each $k>0$, so that a valid norm. Hence, we have $Tr(K_k(t^*))\to\infty$. As a consequence, at least one of the diagonal elements is unbounded, since otherwise the trace would be bounded by the (finite) sum of these bounds. Picking $e\in\mathbb{R}^n$ as the basis vector associated with that diagonal element, we have
\begin{equation}\label{inftye}
	\vert e \vert_{K_k(t^*)}^2=\langle K_k(t^*)e,e\rangle\to\infty,\quad k\to\infty.
\end{equation}
Define the zero-extension of $v(\cdot)\in L_{\mathcal{G}^0}^2(t^*,T;\mathbb{R}^{n_v})$ as follows:
\begin{equation*}
	\left[0\mathbf{1}_{[0,t^*)}\oplus v(\cdot)\right](t)=
	\begin{cases}
	0,\quad t\in[0,t^*),\\
	v(\cdot),\quad t\in[t^*,T].
	\end{cases}
\end{equation*}
Clearly, $v(\cdot)\in \mathcal{U}_{vd}$. Let $\Upsilon=\{\Upsilon(t): 0\leq t\leq T\}$ be the solution to the linear matrix {\it stochastic differrential equation} (SDE):
\begin{equation*}
\left\{\begin{aligned}
	d\Upsilon(t)&=A(t)\Upsilon(t)dt+C(t)\Upsilon(t)dW_0(t),\\
	\Upsilon(0)&=I_n,
\end{aligned}\right.
\end{equation*}
and choose the initial value
\begin{equation*}
\begin{aligned}
	\xi&=\Upsilon(t^*)^{-1}e-\int_0^{t^*}\Upsilon(s)^{-1}\left[(B+CD)\bar{u}_0+H\bar{u}_1+Fm\right](s)ds\\
       &\quad-\int_{0}^{t^*}\Upsilon(s)^{-1}\left[ D\bar{u}_0+Om\right](s) dW_0(s),
\end{aligned}\end{equation*}
where $m(\cdot)$ is the solution to the second equation of (\ref{limstate0}) for given $\bar{u}(\cdot)$. Thus, $v(t)=0$ for $t\in[0,t^*)$ will yield $x_0(t^*)=e$.
For any $v(\cdot)\in L_{\mathcal{G}^0}^2(t^*,T;\mathbb{R}^{n_v})$,
\begin{equation}\label{J0kv}
\begin{aligned}
	&J_0^{\gamma_k}(\bar{u}(\cdot),0\mathbf{1}_{[0,t^*)}\oplus v(\cdot))\\
	&=\mathbb{E}\bigg\{\int_{t^*}^T\left[ \vert x_0-\Gamma_1m\vert_Q^2+\vert \bar{u}_0\vert_{R_0}^2+\vert \bar{u}_1\vert_{R_1}^2-\gamma_k^2\vert v\vert_{R_2}^2\right](t)dt+\vert x_0(T)-\Gamma_2m(T)\vert_G^2\\
	&\qquad+\int_0^{t^*}\left[ \vert x_0-\Gamma_1m\vert_Q^2+\vert \bar{u}_0\vert_{R_0}^2+\vert \bar{u}_1\vert_{R_1}^2\right](t)dt\bigg\}\\
	&\geq\mathbb{E}\bigg\{\int_{t^*}^T\left[ \vert x_0-\Gamma_1m\vert_Q^2+\vert \bar{u}_0\vert_{R_0}^2+\vert \bar{u}_1\vert_{R_1}^2-\gamma_k^2\vert v\vert_{R_2}^2\right](t)dt+\vert x_0(T)-\Gamma_2m(T)\vert_G^2 \bigg\}.
\end{aligned}
\end{equation}
Applying It\^{o}'s formula to $\langle K_k(\cdot)x_0(\cdot),x_0(\cdot)\rangle$, we can get
\begin{equation*}
\begin{aligned}
	&\mathbb{E}\left[\langle Gx_0(T),x_0(T)\rangle-\langle K_k(t^*)e,e\rangle\right]\\
	&=\mathbb{E}\bigg\{\int_{t^*}^T\bigg[ \left\langle\left(\dot{K}_k+K_kA+A^\top K_k+C^\top K_kC\right)x_0,x_0\right\rangle
     +2\left\langle H^\top K_kx_0,\bar{u}_1\right\rangle \\
	&\qquad +2\left\langle\big(B^\top K_k+D^\top K_kC\big)x_0+D^\top K_kOm,\bar{u}_0\right\rangle+2\left\langle E^\top K_kx_0,v\right\rangle \\
	&\qquad+2\left\langle \big(F^\top K_k+O^\top K_kC\big)x_0,m\right\rangle +\left\langle D^\top K_kD\bar{u}_0,\bar{u}_0\right\rangle
     +\left\langle O^\top K_kOm,m\right\rangle\bigg] dt\bigg\}.
\end{aligned}
\end{equation*}
Substituting the above equation into the right-hand side of the inequality (\ref{J0kv}), we have
\begin{equation*}
\begin{aligned}
	&J_0^{\gamma_k}(\bar{u}(\cdot),0\mathbf{1}_{[0,t^*)}\oplus v(\cdot))\\
	&\geq\mathbb{E}\left[\langle K_k(t^*)e,e\rangle-2\langle G\Gamma_2m(T),x_0(T)\rangle +\left\langle\Gamma_2^\top G\Gamma_2m(T),m(T)\right\rangle\right]\\
	&\quad +\mathbb{E}\bigg\{\int_{t^*}^T\bigg[ -\gamma_k^2\left\vert v-\gamma_k^{-2}R_2^{-1}E^\top K_kx_0\right\vert_{R_2}^2
     -2\left\langle\big(Q\Gamma_1-K_kF-C^\top K_kO\big)m,x_0\right\rangle \\
	&\qquad\qquad +\left\langle\big(\Gamma_1^\top Q\Gamma_1+O^\top K_kO\big)m,m\right\rangle
     +\left\langle \left(R_0+D^\top K_kD\right)\bar{u}_0,\bar{u}_0\right\rangle +\vert\bar{u}_1\vert_{R_1}^2\\
    &\qquad\qquad +2\left\langle \left(B^\top K_k+D^\top K_kC\right)x_0+D^\top K_kOm,\bar{u}_0\right\rangle +2\left\langle H^\top K_kx_0,\bar{u}_1\right\rangle\bigg] dt\bigg\},
\end{aligned}
\end{equation*}
which, together with (\ref{inftye}), implies
\begin{equation}\label{inftyJ0kv}
	\sup_{v(\cdot)\in\, L_{\mathcal{G}^0}^2(t^*,T;\mathbb{R}^{n_v})}J_0^{\gamma_k}(\bar{u}(\cdot),0\mathbf{1}_{[0,t^*)}\oplus v(\cdot))\to\infty.
\end{equation}

Moreover, for any $v(\cdot)\in L_{\mathcal{G}^0}^2(0,T;\mathbb{R}^{n_v})$, we also have
\begin{equation*}
\begin{aligned}
	J_0^\gamma(\bar{u}(\cdot),v(\cdot))&=\mathbb{E}\bigg\{\int_0^T\left[ \vert x_0-\Gamma_1m\vert_Q^2+\vert \bar{u}_0\vert_{R_0}^2+\vert \bar{u}_1\vert_{R_1}^2-\gamma_k^2\vert v\vert_{R_2}^2\right](t)dt\\
	&\qquad +\vert x_0(T)-\Gamma_2m(T)\vert_G^2 \bigg\}+(\gamma_k^2-\gamma^2)\mathbb{E}\int_0^T\vert v(t)\vert_{R_2}^2dt\geq J_0^{\gamma_k}(\bar{u},v).
\end{aligned}
\end{equation*}
Take the supremum over $v(\cdot)$ on both sides of the above inequality, we can get
\begin{equation*}
	\sup_{v(\cdot)\in\,\mathcal{U}_{vd}}J_0^\gamma(\bar{u}(\cdot),v(\cdot))\geq\sup_{v(\cdot)\in\,\mathcal{U}_{vd}}J_0^{\gamma_k}(\bar{u}(\cdot),v(\cdot)),\quad\forall k\geq0.
\end{equation*}
The above, together with (\ref{inftyJ0kv}), shows that $\sup_{v(\cdot)\in\,\mathcal{U}_{vd}}J_0^\gamma(\bar{u}(\cdot),v(\cdot))=\infty$, which is a contradiction. The proof is complete.
\end{proof}

Now we summarize the relevant results concerning Problem \textbf{(L2)} in the following diagram:
\begin{figure}[htbp]
	\centering
	\includegraphics[width=0.69\linewidth]{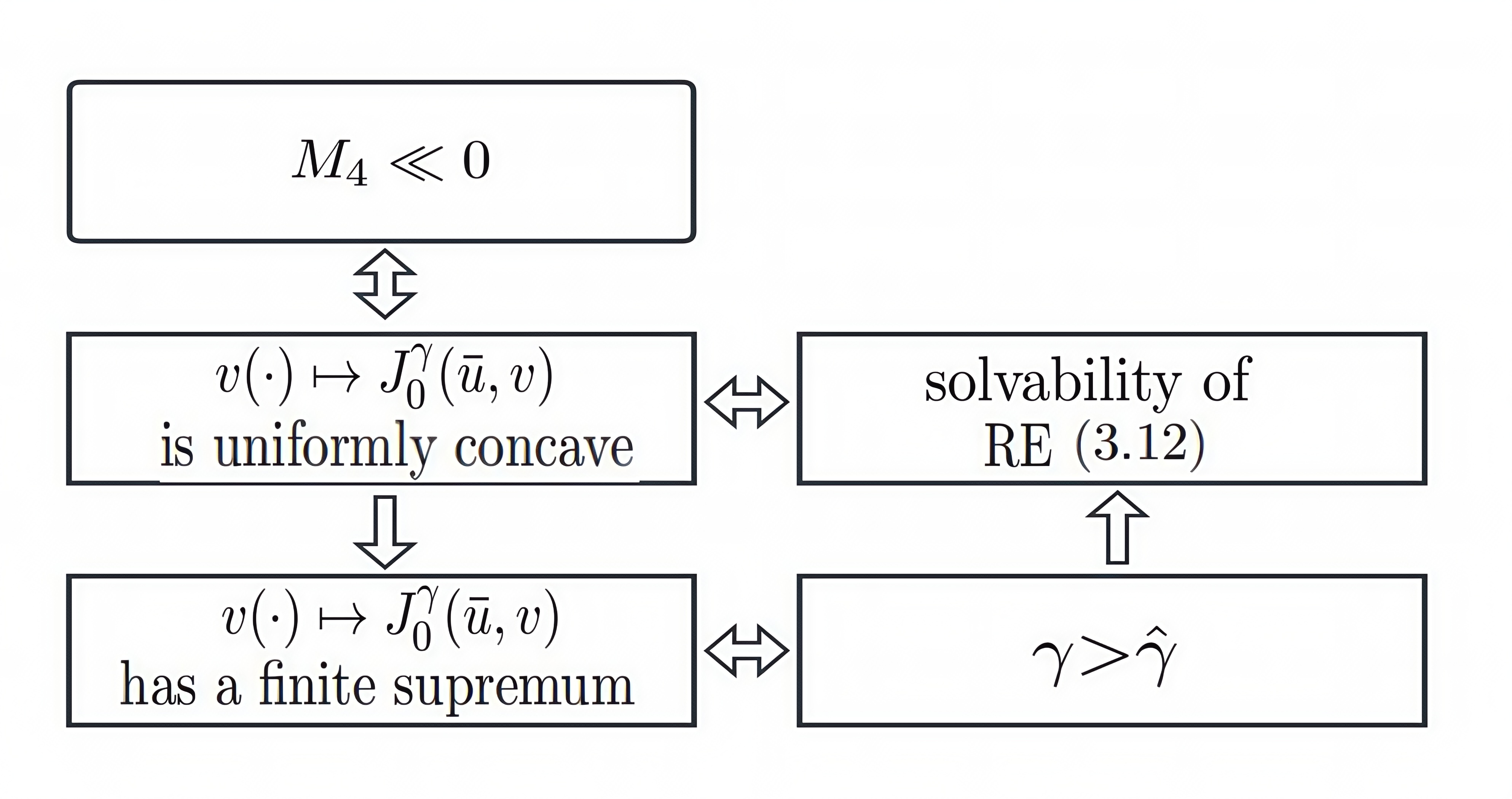}
	\caption{Relationships}
	\label{fig:diagram}
	\label{1}
\end{figure}

Applying It\^{o}'s formula to $\langle x_0^3(\cdot),y^1(\cdot)\rangle, \langle x_0^1(\cdot),y^3(\cdot)\rangle$ and $\langle m^1(\cdot),p^3(\cdot)\rangle$, then taking expectation, $M_2^*$ has the following representation based on $M_2$:
\begin{equation*}
	M_2^*(\bar{u})(\cdot)=E^\top y^1.
\end{equation*}
By Proposition \ref{Prop3.1}, (\ref{staionary1}) takes the following form:
\begin{equation*}
\left\{
\begin{aligned}
	&B^\top (y^1+y^2+y^3)+D^\top (z^1+z^2+z^3)+\tilde{H}^\top (p^1+p^2+p^3)+R_0\bar{u}_0^*=0,\\
	&H^\top (y^1+y^2+y^3)+\tilde{B}^\top (p^1+p^2+p^3)+R_1\bar{u}_1^*=0,\\
	&E^\top(y^1+y^2+y^3)-\gamma^2R_2v^*=0.
\end{aligned}
\right.
\end{equation*}
Define $y^*:=y^1+y^2+y^3$, $z^*:=z^1+z^2+z^3$, $p^*:=p^1+p^2+p^3$ and $q^*:=q^1+q^2+q^3$, we have the following solvability result in terms of {\it forward-backward SDEs} (FBSDEs).

\begin{mythm}
Let \textbf{(A1)}-\textbf{(A3)} hold, and $\gamma>\hat{\gamma}$. Then Problem \textbf{(L2)} is (uniquely) solvable if and only if there (uniquely) exists a 6-tuple $(x_0^*(\cdot),m^*(\cdot),y^*(\cdot),z^*(\cdot),p^*(\cdot),q^*(\cdot))$ and $(\bar{u}^*(\cdot),v^*(\cdot))$ satisfying FBSDEs
\begin{equation}\label{FBSDEs3.1}
\left\{\begin{aligned}
	dx_0^*(t)&=\left[Ax_0^*+B\bar{u}_0^*+Fm^*+H\bar{u}_1^*+Ev^*\right]dt+\left[Cx_0^*+D\bar{u}_0^*+Om^*\right]dW_0,\\
    \dot{m}^*(t)&=(\tilde{A}+\tilde{F})m^*+\tilde{B}\bar{u}_1^*+\tilde{H}\bar{u}_0^*,\\
    dy^*(t)&=-\left[A^\top y^*+C^\top z^*+Qx_0^*-Q\Gamma_1m^*\right]dt+z^*dW_0,\\
    dp^*(t)&=-\left[(\tilde{A}+\tilde{F})^\top p^*+F^\top y^*+O^\top z^*+\Gamma_1^\top Q\Gamma_1 m^*-\Gamma_1^\top Q x_0^*\right]dt+q^*dW_0,\\
    x_0^*(0)&=\xi,\quad m^*(0)=x,\\
    y^*(T)&=Gx_0^*(T)-G\Gamma_2m^*(T),\quad p^*(T)=\Gamma_2^\top G\Gamma_2m^*(T)-\Gamma_2^\top Gx_0^*(T),
\end{aligned}\right.
\end{equation}
such that
\begin{equation}\label{stationary1'}
\left\{
\begin{aligned}
	&B(t)^\top y^*(t)+D^\top(t) z^*(t)+\tilde{H}^\top(t) p^*(t)+R_0(t)\bar{u}_0^*(t)=0,\\
	&H^\top(t) y^*(t)+\tilde{B}^\top(t) p^*(t)+R_1(t)\bar{u}_1^*(t)=0,\\
	&E^\top(t) y^*(t)-\gamma^2R_2(t)v^*(t)=0,\quad t\in[0,T],\quad\mathbb{P}\mbox{-}a.s.,
\end{aligned}
\right.
\end{equation}
and
\begin{equation}\label{J0*}
	V_0(\xi,x)=\langle\xi,y^*(0)\rangle+\langle x,p^*(0)\rangle.
\end{equation}
\end{mythm}

\begin{proof}
Under \textbf{(A1)}-\textbf{(A3)}, by Propositions \ref{Prop3.1} and \ref{Prop3.2}, we obtain the stationarity condition (\ref{stationary1'}).  Applying It\^{o}'s formula to $\langle x_0^*(\cdot),y^*(\cdot)\rangle+\langle m^*(\cdot),p^*(\cdot)\rangle$, integrating both sides on $[0, T]$, taking expectation and substituting it into the cost functional (\ref{limcost0}), we can derive (\ref{J0*}).
\end{proof}

Using (\ref{stationary1'}), $(\bar{u}^*(\cdot),v^*(\cdot))$ can be represented in the open-loop form as follows:
\begin{equation}
\left\{
\begin{aligned}
	\bar{u}_0^*(t)&=-R_0(t)^{-1}\left[B(t)^\top y^*(t)+D^\top(t) z^*(t)+\tilde{H}^\top(t) p^*(t)\right],\\
	\bar{u}_1^*(t)&=-R_1(t)^{-1}\left[H^\top(t) y^*(t)+\tilde{B}^\top(t) p^*(t)\right],\\
	v^*(t)&=\gamma^{-2}R_2(t)^{-1}E^\top(t) y^*(t),\quad t\in[0,T],\quad\mathbb{P}\mbox{-}a.s.
\end{aligned}
\right.
\end{equation}

Inspired by the decoupling technique introduced in \cite{Ma94,Ma99}, we now consider the solvability of FBSDEs (\ref{FBSDEs3.1}), and get the closed-loop representation of open-loop saddle point $(\bar{u}^*(\cdot),v^*(\cdot))$.

For convenience, we have suppressed the superscript $^*$. To solve FBSDEs (\ref{FBSDEs3.1}), we make the ansatz that the adapted solution  $(x_0(\cdot),m(\cdot),y(\cdot),z(\cdot),p(\cdot),q(\cdot))$ to (\ref{FBSDEs3.1}) has the form
\begin{equation}
	y(\cdot)=P_1(\cdot)x_0(\cdot)+\Pi_1(\cdot)m(\cdot),\qquad p(\cdot)=P_2(\cdot)x_0(\cdot)+\Pi_2(\cdot)m(\cdot),
\end{equation}
where $P_1(\cdot)$, $\Pi_1(\cdot)$, $P_2(\cdot)$ and $\Pi_2(\cdot)$: $[0,T]\to\mathbb{R}^{n\times n}$ are differentiable maps to be determined. To match the terminal condition of (\ref{FBSDEs3.1}), we set that
\begin{equation*}
	P_1(T)=G,\qquad\Pi_1(T)=-G\Gamma_2,\qquad P_2(T)=-\Gamma_2^\top G,\qquad\Pi_2(T)=\Gamma_2^\top G\Gamma_2.
\end{equation*}
Applying It\^{o}'s formula to $y(\cdot)=P_1(\cdot)x_0(\cdot)+\Pi_1(\cdot)m(\cdot)$, we have
\begin{equation}\label{decouple y}
\begin{aligned}
	dy&=-\left[A^\top y+C^\top z+Qx_0-Q\Gamma_1m\right]dt+zdW_0\\
	&=\left\{\dot{P}_1x_0+P_1\left[Ax_0+B\bar{u}_0+Fm+H\bar{u}_1+Ev\right]+\dot{\Pi}_1m\right.\\
	&\qquad \left.+\Pi_1\left[(\tilde{A}+\tilde{F})m+\tilde{B}\bar{u}_1+\tilde{H}\bar{u}_0\right]\right\}+P_1[Cx_0+D\bar{u}_0+Om]dW_0.
\end{aligned}
\end{equation}
Hence, one should have
\begin{equation*}
	z(t)=P_1(t)\left[C(t)x_0(t)+D(t)\bar{u}_0(t)+O(t)m(t)\right],\quad a.e.\ t\in[0,T],\quad\mathbb{P}\mbox{-}a.s.
\end{equation*}
The stationarity condition (\ref{stationary1'}) then becomes
\begin{equation*}
\begin{aligned}
	0&=B^\top y+D^\top P_1(Cx_0+D\bar{u}_0+Om)+\tilde{H}^\top p+R_0\bar{u}_0\\
	&=B^\top(P_1x_0+\Pi_1m)+D^\top P_1(Cx_0+D\bar{u}_0+Om)+\tilde{H}^\top(P_2x_0+\Pi_2m)+R_0\bar{u}_0\\
	&=(R_0+D^\top P_1D)\bar{u}_0+\left(B^\top P_1+D^\top P_1C+\tilde{H}^\top P_2\right)x_0+\left(D^\top P_1O+B^\top \Pi_1+\tilde{H}^\top \Pi_2\right)m.
\end{aligned}
\end{equation*}
If the matrix $R_0+D^\top P_1D$ is invertible, then
\begin{equation*}
	\bar{u}_0(\cdot)=-\left(R_0+D^\top P_1D\right)^{-1}\left[\left(B^\top P_1+\tilde{H}^\top P_2+D^\top P_1C\right)x_0+\left(B^\top \Pi_1+\tilde{H}^\top \Pi_2+D^\top P_1O\right)m\right].
\end{equation*}
Therefore, $z(\cdot)$ can be represented as
\begin{equation*}
\begin{aligned}
	z(\cdot)&=\left[P_1C-P_1D\left(R_0+D^\top P_1D\right)^{-1}(B^\top P_1+D^\top P_1C+\tilde{H}^\top P_2)\right]x_0\\
	&\quad+\left[P_1O-P_1D\left(R_0+D^\top P_1D\right)^{-1}\left(B^\top \Pi_1+D^\top P_1O+\tilde{H}^\top \Pi_2\right)\right] m,\quad\mathbb{P}\mbox{-}a.s.
\end{aligned}
\end{equation*}

For $\bar{u}_1(\cdot)$, we have
\begin{equation*}
	H^\top(P_1x_0+\Pi_1m)+\tilde{B}^\top(P_2x_0+\Pi_2m)+R_1\bar{u}_1=0,\quad\mathbb{P}\mbox{-}a.s.
\end{equation*}
thus
\begin{equation*}
	\bar{u}_1(\cdot)=-R_1^{-1}\left[\left(H^\top P_1+\tilde{B}^\top P_2\right)x_0+\left(H^\top \Pi_1+\tilde{B}^\top \Pi_2\right)m\right],\quad\mathbb{P}\mbox{-}a.s.
\end{equation*}
Similarly,
\begin{equation*}
	v(\cdot)=\gamma^{-2}R_2^{-1}E^\top(P_1x_0+\Pi_1m),\quad\mathbb{P}\mbox{-}a.s.
\end{equation*}
Substituting the above equalities into  (\ref{decouple y}) and comparing coefficients in the drift term of (\ref{decouple y}), one gets
\begin{equation}\label{P1}
\left\{\begin{aligned}
	&\dot{P}_1+P_1A+A^\top P_1+C^\top P_1C+Q+\gamma^{-2}P_1ER_2^{-1}E^\top P_1\\
	&\ -\left(P_1B +\Pi_1\tilde{H}+C^\top P_1D\right)\left(R_0+D^\top P_1D\right)^{-1}\left(B^\top P_1+\tilde{H}^\top P_2+D^\top P_1C\right)\\
    &\ -\left(P_1H+\Pi_1\tilde{B}\right)R_1^{-1}\left(H^\top P_1+\tilde{B}^\top P_2\right)=0,\\
	&P_1(T)=G,
\end{aligned}\right.
\end{equation}
\begin{equation}\label{Pi1}
\left\{\begin{aligned}
	&\dot{\Pi}_1+\Pi_1\left(\tilde{A}+\tilde{F}\right)+A^\top\Pi_1+P_1F-Q\Gamma_1+\gamma^{-2}P_1ER_2^{-1}E^\top \Pi_1+C^\top P_1O\\
    &\ -\left(P_1B +\Pi_1\tilde{H}+C^\top P_1D\right)\left(R_0+D^\top P_1D\right)^{-1}\left(B^\top \Pi_1+\tilde{H}^\top \Pi_2+D^\top P_1O\right)\\
    &\ -\left(P_1H+\Pi_1\tilde{B}\right)R_1^{-1}\left(H^\top \Pi_1+\tilde{B}^\top \Pi_2\right)=0,\\
    &\Pi_1(T)=-G\Gamma_2.
\end{aligned}\right.
\end{equation}
In the same way, by applying It\^{o}'s formula to $p(\cdot)=P_2(\cdot)x_0(\cdot)+\Pi_2(\cdot)m(\cdot)$, we can obtain the following equations:
\begin{equation}\label{P2}
\left\{\begin{aligned}
	&\dot{P}_2+P_2A+\left(\tilde{A}+\tilde{F}\right)^\top P_2+F^\top P_1-\Gamma_1^\top Q+\gamma^{-2}P_2ER_2^{-1}E^\top P_1+O^\top P_1C\\
	&\ -\left(P_2B +\Pi_2\tilde{H}+O^\top P_1D\right)\left(R_0+D^\top P_1D\right)^{-1}\left(B^\top P_1+\tilde{H}^\top P_2+D^\top P_1C\right)\\
    &\ -\left(P_2H+\Pi_2\tilde{B}\right)R_1^{-1}\left(H^\top P_1+\tilde{B}^\top P_2\right)=0,\\
	&P_2(T)=-\Gamma_2^\top G,
\end{aligned}\right.
\end{equation}
\begin{equation}\label{Pi2}
\left\{\begin{aligned}
	&\dot{\Pi}_2+\Pi_2\left(\tilde{A}+\tilde{F}\right)+(\tilde{A}+\tilde{F})^\top\Pi_2+P_2F+F^\top\Pi_1+\Gamma_1^\top Q\Gamma_1+\gamma^{-2}P_2ER_2^{-1}E^\top \Pi_1\\
    &\ +O^\top P_1O-\left(P_2B +\Pi_2\tilde{H}+O^\top P_1D\right)\left(R_0+D^\top P_1D\right)^{-1}\left(B^\top \Pi_1+\tilde{H}^\top \Pi_2+D^\top P_1O\right)\\
    &\ -\left(P_2H+\Pi_2\tilde{B}\right)R_1^{-1}\left(H^\top \Pi_1+\tilde{B}^\top \Pi_2\right)=0,\\
    &\Pi_2(T)=\Gamma_2^\top G\Gamma_2.
\end{aligned}\right.
\end{equation}
From the above analysis, if Riccati equations (\ref{P1})-(\ref{Pi2}) admit solutions, then (\ref{FBSDEs3.1}) is decoupled and solvable. The wellposedness of (\ref{FBSDEs3.1}) is obtained and further allows a closed-loop representation of open-loop saddle point $(\bar{u}^*(\cdot),v^*(\cdot))$ given by
\begin{equation}\label{closed-loop rep}
\left\{
\begin{aligned}
	\bar{u}_0^*(x_0^*,m^*)(\cdot)&=-\left(R_0+D^\top P_1D\right)^{-1}\left[\left(B^\top P_1+\tilde{H}^\top P_2+D^\top P_1C\right)x_0^*\right.\\
    &\qquad +\left.\left(B^\top \Pi_1+\tilde{H}^\top \Pi_2+D^\top P_1O\right)m^*\right],\\
	\bar{u}_1^*(x_0^*,m^*)(\cdot)&=-R_1^{-1}\left[\left(H^\top P_1+\tilde{B}^\top P_2\right)x_0^*+\left(H^\top \Pi_1+\tilde{B}^\top \Pi_2\right)m^*\right],\\
	v^*(x_0^*,m^*)(\cdot)&=\gamma^{-2}R_2^{-1}E^\top(P_1x_0^*+\Pi_1m^*),\quad\mathbb{P}\mbox{-}a.s.,
\end{aligned}
\right.
\end{equation}
where $(x_0^*(\cdot),m^*(\cdot))$ is the solution to the following closed-loop stochastic system:
\begin{equation}\label{closed-loop system}
\left\{\begin{aligned}
	dx_0^*(t)&=\left[Ax_0^*+B\bar{u}_0^*(x_0^*,m^*)+Fm^*+H\bar{u}_1^*(x_0^*,m^*)+Ev^*(x_0^*,m^*)\right]dt\\
	&\quad+\left[Cx_0^*+D\bar{u}_0^*(x_0^*,m^*)+Om^*\right]dW_0(t),\\
	\dot{m}^*(t)&=(\tilde{A}+\tilde{F})m^*+\tilde{B}\bar{u}_1^*(x_0^*,m^*)+\tilde{H}\bar{u}_0^*(x_0^*,m^*),\\
	x_0^*(0)&=\xi,\quad m^*(0)=x.
\end{aligned}\right.
\end{equation}
Moreover,
\begin{equation}\label{J0*'}
	V_0(\xi,x)=\langle P_1(0)\xi,\xi\rangle+\langle(\Pi_1(0)+P_2(0)^\top)x,\xi\rangle+\langle\Pi_2(0)x,x\rangle.
\end{equation}

\begin{Remark}
(\romannumeral1) By comparing coefficients, we have $\Pi_1^\top(\cdot)=P_2(\cdot)$.

(\romannumeral2) By Theorem 8.3 of \cite{Sun19}, the closed-loop representation coincides with the outcome of the closed-loop saddle point. Therefore, $(\bar{u}^*(\cdot),v^*(\cdot))$ given by (\ref{closed-loop rep}) is essentially an outcome of a closed-loop saddle point of Problem \textbf{(L2)}.
\end{Remark}

Let us discuss the solvability of (\ref{P1})-(\ref{Pi2}). Set
\begin{equation*}
	P=\begin{pmatrix}
		P_1 & \Pi_1\\
		P_2 & \Pi_2
	\end{pmatrix},
\end{equation*}
which satisfies
\begin{equation}\label{P}
\left\{\begin{aligned}
	&\dot{P}+P\bar{A}+\bar{A}^\top P+\bar{C}^\top P\bar{C}+\bar{Q}+\gamma^{-2}P\bar{E}R_2^{-1}\bar{E}^\top P\\
	&\ -\left(P\bar{B}+\bar{C}^\top P\bar{D}\right)\left(\bar{R}+\bar{D}^\top P\bar{D}\right)^{-1}\left(\bar{B}^\top P+\bar{D}^\top P\bar{C}\right)=0,\\
	&P(T)=\bar{G},
\end{aligned}\right.
\end{equation}
where
\begin{equation*}
\begin{aligned}
	\bar{A}&=\begin{pmatrix}
		A & F\\
		0 & \tilde{A}+\tilde{F}\\
	\end{pmatrix},\quad
	\bar{B}=\begin{pmatrix}
		B & H\\
		\tilde{H} & \tilde{B}\\
	\end{pmatrix},\quad
	\bar{C}=\begin{pmatrix}
		C & O\\
		0 & 0\\
	\end{pmatrix},\quad
	\bar{D}=\begin{pmatrix}
		D & 0\\
		0 & 0\\
	\end{pmatrix},\\
	\bar{Q}&=\begin{pmatrix}
		Q & -Q\Gamma_1\\
		-\Gamma_1^\top Q & \Gamma_1^\top Q\Gamma_1\\
	\end{pmatrix},\quad
	\bar{R}=\begin{pmatrix}
		R_0 & 0\\
		0 & R_1\\
	\end{pmatrix},\quad
	\bar{G}=\begin{pmatrix}
		G & -G\Gamma_2\\
		-\Gamma_2^\top G & \Gamma_2^\top G\Gamma_2\\
	\end{pmatrix},\quad
	\bar{E}=\begin{pmatrix}
		E\\0
	\end{pmatrix}.
\end{aligned}
\end{equation*}
Therefore, the solvability of (\ref{P1})-(\ref{Pi2}) is equivalent to the solvability of the Riccati equation (\ref{P}).

\begin{Remark}
In fact, one can verify, (\ref{P}) coincides with the Riccati equation associated with the existence of a closed-loop saddle point for an augmented LQ (two forward SDEs) zero-sum game.
\end{Remark}

Now we present the main result of this section.

\begin{mythm}
Let \textbf{(A1)}-\textbf{(A3)} hold. For the linear-quadratic zero-sum differential game, let $\hat{\gamma}$ be as defined by (\ref{hatgamma}). Then,

(\romannumeral1) For $\gamma>\hat{\gamma}$, the Riccati equation (\ref{P}) admits a unique
solution on the interval $[0,T]$.

(\romannumeral2) For $\gamma>\hat{\gamma}$, there exists a unique saddle point, given by
\begin{equation}
\left\{
\begin{aligned}
	\bar{u}^*(t)&=-\left(\bar{R}+\bar{D}^\top P\bar{D}\right)^{-1}\left(\bar{B}^\top P+\bar{D}^\top P\bar{C}\right)X^*(t),\\
	v^*(t)&=\gamma^{-2}R_2^{-1}\bar{E}^\top PX^*(t),\quad t\in[0,T],\quad\mathbb{P}\mbox{-}a.s.,
\end{aligned}
\right.
\end{equation}
where $X^*=\begin{pmatrix}
	x_0^*\\m^*
\end{pmatrix}$ is the corresponding state trajectory, generated by
\begin{equation}
\left\{\begin{aligned}
	dX^*(t)&=\left[\bar{A}-\bar{B}\left(\bar{R}+\bar{D}^\top P\bar{D}\right)^{-1}\left(\bar{B}^\top P+\bar{D}^\top P\bar{C}\right)+\gamma^{-2}\bar{E}R_2^{-1}\bar{E}^\top P\right]X^*dt\\
	&\quad+\left[\bar{C}-\bar{D}\left(\bar{R}+\bar{D}^\top P\bar{D}\right)^{-1}\left(\bar{B}^\top P+\bar{D}^\top P\bar{C}\right)\right]X^*dW_0(t),\\
	X^*(0)&=\bar{\xi}.
\end{aligned}\right.
\end{equation}
Moreover, the saddle point value is
\begin{equation}
	V_0(\xi,x)=J_0(\bar{u}^*(\cdot),v^*(\cdot))=\bar{\xi}^\top P(0)\bar{\xi}.
\end{equation}

(\romannumeral3) For $\gamma\leq\hat{\gamma}$, there is some $\xi\in\mathbb{R}^n$ such that the upper value of the game $\inf\limits_{\bar{u}(\cdot)}\sup\limits_{v(\cdot)}J_0^\gamma(\bar{u}(\cdot),v(\cdot))$ is unbounded.
\end{mythm}

\begin{proof}
From the above result, when $\gamma>\hat{\gamma}$, $J_0^\gamma(\bar{u}(\cdot),v(\cdot))$ is uniformly concave in $v(\cdot)$. Under the assumptions \textbf{(A1)}-\textbf{(A3)}, Problem \textbf{(L2)} is uniformly convex-concave. Then Problem \textbf{(L2)} admits a (closed-loop) saddle point on $[0,T]$, which is equivalent to the solvability of the Riccati equation (\ref{P}). Thus, (\romannumeral1) holds. (\romannumeral2) can be obtained by rewriting (\ref{closed-loop rep})-(\ref{J0*'}). For part (\romannumeral3), one can drive it from the proof of Theorem \ref{thm3.1}. The proof is complete.
\end{proof}

\section{The auxiliary limiting  problem for followers}

For a finite number of followers, the leader can accurately obtain the team-optimal strategy, then design differentiated incentive form as follows. For $1\leq i\leq N$,
\begin{equation}
\begin{aligned}
	u_{0i}(t)&\equiv\Gamma_i(t,u_{1i}(\cdot),x_0(\cdot),x_i(\cdot),x^{(N)}(\cdot))=u_{0i}^*(x_0,x_i,x^{(N)})(t)+L(t)\left[u_{1i}-u_{1i}^*(x_0,x_i,x^{(N)})\right](t),
\end{aligned}
\end{equation}
where $(u_{01}^*(\cdot),\cdots,u_{0N}^*(\cdot),u_{11}^*(\cdot),\cdots,u_{1N}^*(\cdot))$ represents the leader's team-optimal solution, and $L(\cdot)$ is (unknown) deterministic matrix-valued function of proper dimension, which is called {\it the leader's incentive matrix}. Since the system is symmetric, the leader's incentive matrix for each follower is the same. If the population size $N$ is sufficiently large, the influence of an individual follower's behavior on the whole system tends to zero, and it is difficult to accurately obtain all the information of each follower. Therefore, the leader cannot derive the optimal strategy of the i-th follower $u_{1i}^*(x_0,x_i,x^{(N)})$ and the optimal control $u_{0i}^*(x_0,x_i,x^{(N)})$ that it should impose on the i-th follower. At this point, the leader tends to focus on the aggregate behavior of the followers (the empirical distribution of the followers' states and controls) and designs a macroscopic incentive strategy (such as benchmark interest rates, taxes, subsidies, etc.) by deriving an overall team-optimal strategy. Therefore, we set the approximation incentive form of the leader for a single representative follower as follows:\vspace{-2mm}
\begin{equation}\label{incentive form}
\begin{aligned}
	u_{0i}(t)&\equiv\Gamma_i(t,u_{1i}(\cdot),x_0(\cdot),m(\cdot))=\bar{u}_0^*(x_0,m)(t)+L(t)\left[u_{1i}(t)-\bar{u}_1^*(x_0,m)(t)\right]\\
	&=-\left(R_0+D^\top P_1D\right)^{-1}\Big[\left(B^\top P_1+\tilde{H}^\top P_2+D^\top P_1C\right)x_0(t)+\left(B^\top \Pi_1+\tilde{H}^\top \Pi_2\right.\\
	&\quad\,+\left.D^\top P_1O\right)m(t)\Big]+L(t)\left[u_{1i}(t) +R_1^{-1}\left(H^\top P_1+\tilde{B}^\top P_2\right)x_0(t)\right.\\
    &\quad +\left.R_1^{-1}\left(H^\top \Pi_1+\tilde{B}^\top \Pi_2\right)m(t)\right]:=L(t)u_{1i}(t)+\zeta(t)x_0(t)+\eta(t)m(t),
\end{aligned}
\end{equation}
where
\begin{equation*}
\begin{aligned}
	\zeta(\cdot)&:=-\left(R_0+D^\top P_1D\right)^{-1}\left(B^\top P_1+\tilde{H}^\top P_2+D^\top P_1C\right)+LR_1^{-1}\left(H^\top P_1+\tilde{B}^\top P_2\right),\\
	\eta(\cdot)&:=-\left(R_0+D^\top P_1D\right)^{-1}\left(B^\top \Pi_1+\tilde{H}^\top \Pi_2+D^\top P_1O\right)+LR_1^{-1}\left(H^\top \Pi_1+\tilde{B}^\top \Pi_2\right).
\end{aligned}
\end{equation*}
Obviously, $\bar{u}_0(t)=L(t)\bar{u}_1(t)+\zeta(t) x_0(t)+\eta(t)m(t)$ and
\begin{equation}\label{incentive rel}
\begin{aligned}
	\bar{u}_0^*(x_0^*,m^*)(t)&=\lim\limits_{N\to\infty}\frac{1}{N}\sum_{i=1}^{N}\Gamma_i\left(t,\bar{u}_1^*(\cdot),x_0^*(\cdot),m^*(\cdot)\right)\\
	&=L(t)\bar{u}_1^*(x_0^*,m^*)(t)+\zeta(t)x_0^*(t)+\eta(t)m^*(t).
\end{aligned}
\end{equation}

\begin{Remark}
If $N$ is sufﬁciently large, the approximate Nash equilibrium strategy of each follower and the leader's approximate incentive strategy for it are identical (see \cite{Mukaidani22,Sanjari25}). Therefore, it is reasonable to set the leader's incentive strategy in the above form for the limiting system.
\end{Remark}

Substituting the leader's incentive strategy $u_{0i}(\cdot)=Lu_{1i}+\zeta x_0+\eta m$ for each follower and $\bar{u}_0(\cdot)=L\bar{u}_1+\zeta x_0+\eta m$ into the auxiliary limiting state equation (\ref{limstatei}) and cost functional (\ref{limcosti}) for the i-th follower $\mathcal{A}_i$, we obtain
\begin{equation}\label{limstatei'}
\left\{\begin{aligned}
	dx_i(t)&=\left[\tilde{A}x_i+\left(\tilde{B}+\tilde{H}L\right)u_{1i}+\left(\tilde{F}+\tilde{H}\eta\right)m+\tilde{H}\zeta x_0\right]dt+\sigma dW_i(t),\\
	\dot{m}(t)&=\left(\tilde{A}+\tilde{F}+\tilde{H}\eta\right)m+\left(\tilde{B}+\tilde{H}L\right)\bar{u}_1+\tilde{H}\zeta x_0,\\
	dx_0(t)&=\left[\left(A+\gamma^{-2}ER_2^{-1}E^\top P_1+B\zeta\right)x_0+\left(F+\gamma^{-2}ER_2^{-1}E^\top \Pi_1+B\eta\right)m\right.\\
	&\quad\,+(H+BL)\bar{u}_1\Big]dt+\left[(C+D\zeta)x_0+(O+D\eta) m+DL\bar{u}_1\right]dW_0(t),\\
	x_i(0)&=x,\quad m(0)=x,\quad x_0(0)=\xi,
\end{aligned}\right.
\end{equation}
and
\begin{equation}\label{limcosti'}
\begin{aligned}
	&J_i'(u_{1i}(\cdot)):=J_i\left(u_{0i}(u_{1i})(\cdot),u_{1i}(\cdot)\right)\\
	&=\mathbb{E}\bigg\{\int_0^T\bigg[\left\langle\tilde{Q}x_i,x_i\right\rangle +\left\langle\left(\eta^\top\tilde{R}_0\eta+\tilde{\Gamma}_1^\top\tilde{Q}\tilde{\Gamma}_1\right)m,m\right\rangle +\left\langle\zeta^\top\tilde{R}_0\zeta x_0,x_0\right\rangle\\
	&\qquad\quad -2\left\langle\tilde{Q}\tilde{\Gamma}_1m,x_i\right\rangle +2\left\langle\eta^\top\tilde{R}_0\zeta x_0,m\right\rangle +\left\langle\left(\tilde{R}_1+L^\top\tilde{R}_0L\right)u_{1i},u_{1i}\right\rangle\\
	&\qquad\quad +2\left\langle L^\top\tilde{R}_0\zeta x_0+L^\top\tilde{R}_0\eta m,u_{1i}\right\rangle\bigg]dt +\left\langle\tilde{G}x_i(T),x_i(T)\right\rangle\\
	&\qquad\quad -2\left\langle\tilde{G}\tilde{\Gamma}_2m(T),x_i(T)\right\rangle +\left\langle\tilde{\Gamma}_2^\top\tilde{G}\tilde{\Gamma}_2m(T),m(T)\right\rangle\bigg\}.
\end{aligned}
\end{equation}
Fixing $\bar{u}_1(\cdot)$, the i-th follower $\mathcal{A}_i$ needs to consider the corresponding auxiliary limiting problem.

Problem \textbf{(F2)'}. For the i-th follower $\mathcal{A}_i$, find a decentralized control $u_{1i}^+(\cdot)\in\mathcal{U}_{id}$ such that
\begin{equation*}
	J_i'(u_{1i}^+(\cdot))=\inf_{u_{1i}(\cdot)\in\,\mathcal{U}_{id}}J_i'(u_{1i}(\cdot)),\quad 1\leq i\leq N,
\end{equation*}
subjects to (\ref{limstatei'}) and (\ref{limcosti'}).

We further give the following assumptions on coefficients:\\
\textbf{(A4)} $\tilde{Q}\geq0$, $\tilde{G}\geq0$, $\tilde{R}_0\gg0$, $\tilde{R}_1\gg0$.

\begin{mythm}\label{thm4.1}
Let \textbf{(A1)}-\textbf{(A4)} hold, and $\gamma>\hat{\gamma}$. For any given $\bar{u}_1^+(\cdot)\in L_{\mathcal{G}^0}^2(0,T;\mathbb{R}^{m_F})$, $(u_{1i}^+(\cdot),x_i^+(\cdot))$ is the decentralized optimal pair
of Problem \textbf{(F2)'} if and only if the adapted solution $(x_i^+(\cdot),\varphi_i^+(\cdot),\psi^+(\cdot),\psi_0^+(\cdot))$ to the following FBSDE:
\begin{equation}\label{FBSDEfollowers}
\left\{\begin{aligned}
	dx_i^+(t)&=\left[\tilde{A}x_i^++\left(\tilde{B}+\tilde{H}L\right)u_{1i}^++\left(\tilde{F}+\tilde{H}\eta\right)m^++\tilde{H}\zeta x_0^+\right]dt+\sigma dW_i(t),\\
	d\varphi_i^+(t)&=-\left[\tilde{A}^\top\varphi_i^++\tilde{Q}x_i^+-\tilde{Q}\tilde{\Gamma}_1m^+\right]dt+\psi^+dW_i+\psi_0^+dW_0,\\
	x_i^+(0)&=x,\quad\varphi_i^+(T)=\tilde{G}x_i^+(T)-\tilde{G}\tilde{\Gamma}_2m^+(T),
\end{aligned}\right.
\end{equation}
satisfies the following stationarity condition:\vspace{-1mm}
\begin{equation*}
	\left(\tilde{R}_1+L^\top\tilde{R}_0L\right)u_{1i}^++\left(\tilde{B}+\tilde{H}L\right)^\top\varphi^+_i+L^\top\tilde{R}_0\zeta x_0^++L^\top\tilde{R}_0\eta m^+=0,\quad a.e.\;t\in[0,T],\quad \mathbb{P}\mbox{-}a.s.,
\end{equation*}
where $(m^+(\cdot),x_0^+(\cdot))$ is the solution to the following coupled SDEs for fixing $\bar{u}_1^+(\cdot)$:\vspace{-1mm}
\begin{equation*}
\left\{\begin{aligned}
	\dot{m}^+(t)&=\left(\tilde{A}+\tilde{F}+\tilde{H}\eta\right)m^++\left(\tilde{B}+\tilde{H}L\right)\bar{u}_1^++\tilde{H}\zeta x_0^+,\\
    dx_0^+(t)&=\left[\left(A+\gamma^{-2}ER_2^{-1}E^\top P_1+B\zeta\right)x_0^++\left(F+\gamma^{-2}ER_2^{-1}E^\top \Pi_1+B\eta\right)m^+\right.\\
    &\quad\,+(H+BL)\bar{u}_1^+\Big]dt+\left[(C+D\zeta)x_0^++(O+D\eta) m^++DL\bar{u}_1^+\right]dW_0(t),\\
    m^+(0)&=x,\quad x_0^+(0)=\xi.
\end{aligned}\right.
\end{equation*}
\end{mythm}

\begin{proof}
Let $u_{1i}^+(\cdot)\in\mathcal{U}_{id}$ and $(x_i^+(\cdot),\varphi_i^+(\cdot),\psi^+(\cdot),\psi_0^+(\cdot))$ be the adapted solution to FBSDE (\ref{FBSDEfollowers}). For any $u_{1i}(\cdot)\in\mathcal{U}_{id}$ and $\lambda\in\mathbb{R}$, let $x_i^\lambda(\cdot)$ be the solution to the following perturbed state equation:
\begin{equation*}
\left\{\begin{aligned}
	dx_i^\lambda(t)&=\left[\tilde{A}x_i^\lambda+\left(\tilde{B}+\tilde{H}L\right)(u_{1i}^++\lambda u_{1i})+\left(\tilde{F}+\tilde{H}\eta\right)m^++\tilde{H}\zeta x_0^+\right]dt+\sigma dW_i(t),\\
	x_i^\lambda(0)&=x.
\end{aligned}\right.
\end{equation*}
Then denoting $\delta x_i(\cdot)$ the solution to the following equation:
\begin{equation*}
\left\{\begin{aligned}
	d\delta x_i(t)&=\left[\tilde{A}\delta x_i+\left(\tilde{B}+\tilde{H}L\right)u_{1i}\right]dt,\\
	\delta x_i(0)&=0,
\end{aligned}\right.
\end{equation*}
we have $x_i^\lambda(\cdot)=x_i^+(\cdot)+\lambda\delta x_i(\cdot)$ and
\begin{equation*}
\begin{aligned}
	&J_i'(u_{1i}^+(\cdot)+\lambda u_{1i}(\cdot))-J_i'(u_{1i}^+(\cdot))\\
&=\lambda^2\mathbb{E}\bigg\{\int_0^T\left[\left\langle\tilde{Q}\delta x_i,\delta x_i\right\rangle +\left\langle\left(\tilde{R}_1+L^\top\tilde{R}_0L\right)u_{1i},u_{1i}\right\rangle\right]dt +\left\langle\tilde{G}\delta x_i(T),\delta x_i(T)\right\rangle\bigg\}\\
	&\quad +2\lambda\mathbb{E}\bigg\{\int_0^T\left[\left\langle\left(\tilde{R}_1+L^\top\tilde{R}_0L\right)u_{1i}^++L^\top\tilde{R}_0\zeta x_0^++L^\top\tilde{R}_0\eta m^+,u_{1i}\right\rangle\right.\\
	&\qquad\qquad\qquad +\left\langle\tilde{Q}x_i^+-\tilde{Q}\tilde{\Gamma}_1m^+,\delta x_i\right\rangle\Big]dt +\left\langle\tilde{G}x_i^+(T)-\tilde{G}\tilde{\Gamma}_2m^+(T),\delta x_i(T)\right\rangle\bigg\}.
\end{aligned}
\end{equation*}
Applying It\^{o}'s formula to $\langle\varphi_i^+(\cdot),\delta x_i(\cdot)\rangle$, we obtain
\begin{equation*}
	\mathbb{E}\left\langle\varphi_i^+(T),\delta x_i(T)\right\rangle=\mathbb{E}\int_0^T\left[-\left\langle\tilde{Q}x_i^+-\tilde{Q}\tilde{\Gamma}_1m^+,\delta x_i\right\rangle+\left\langle\left(\tilde{B}+\tilde{H}L\right)^\top\varphi^+_i,u_{1i}\right\rangle\right]dt.
\end{equation*}
Hence,
\begin{equation*}
\begin{aligned}
	&J_i'(u_{1i}^+(\cdot)+\lambda u_{1i}(\cdot))-J_i'(u_{1i}^+(\cdot))\\
	&=\lambda^2\mathbb{E}\bigg\{\int_0^T\left[\left\langle\tilde{Q}\delta x_i,\delta x_i\right\rangle+\left\langle\left(\tilde{R}_1+L^\top\tilde{R}_0L\right)u_{1i},u_{1i}\right\rangle\right]dt+\left\langle\tilde{G}\delta x_i(T),\delta x_i(T)\right\rangle\bigg\}\\
    &\quad+2\lambda\mathbb{E}\int_0^T\left[\left\langle\left(\tilde{R}_1+L^\top\tilde{R}_0L\right)u_{1i}^++\left(\tilde{B}+\tilde{H}L\right)^\top\varphi^+_i+L^\top\tilde{R}_0\zeta x_0^++L^\top\tilde{R}_0\eta m^+,u_{1i}\right\rangle\right]dt.
\end{aligned}
\end{equation*}
It follows that
\begin{equation*}
	J_i'(u_{1i}^+(\cdot))\leq J_i'(u_{1i}^+(\cdot)+\lambda u_{1i}(\cdot)),\quad\forall u_{1i}(\cdot)\in\mathcal{U}_{id},\;\lambda\in\mathbb{R},\\
\end{equation*}
if and only if the following convexity condition holds:
\begin{equation*}
	\mathbb{E}\bigg\{\int_0^T\left[\left\langle\tilde{Q}\delta x_i,\delta x_i\right\rangle+\left\langle\left(\tilde{R}_1+L^\top\tilde{R}_0L\right)u_{1i},u_{1i}\right\rangle\right]dt+\left\langle\tilde{G}\delta x_i(T),\delta x_i(T)\right\rangle\bigg\}\geq0,\quad\forall u_{1i}(\cdot)\in\mathcal{U}_{id},
\end{equation*}
and
\begin{equation*}
	\left(\tilde{R}_1+L^\top\tilde{R}_0L\right)u_{1i}^++\left(\tilde{B}+\tilde{H}L\right)^\top\varphi^+_i+L^\top\tilde{R}_0\zeta x_0^++L^\top\tilde{R}_0\eta m^+=0,\quad a.e.\;t\in[0,T],\quad \mathbb{P}\mbox{-}a.s.
\end{equation*}
Obviously, \textbf{(A4)} can ensure that the above convexity condition holds. The proof is completed.
\end{proof}

From Theorem \ref{thm4.1}, the decentralized strategy of Problem \textbf{(F2)'} is given by
\begin{equation}\label{ui1+}
\begin{aligned}
	&u_{1i}^+(t)=-\left(\tilde{R}_1+L^\top\tilde{R}_0L\right)^{-1}\left[\left(\tilde{B}+\tilde{H}L\right)^\top\varphi^+_i+L^\top\tilde{R}_0\zeta x_0^++L^\top\tilde{R}_0\eta m^+\right],\\
	&\hspace{6cm} a.e.\;t\in[0,T],\quad \mathbb{P}\mbox{-}a.s.,
\end{aligned}
\end{equation}
and the corresponding stochastic Hamiltonian system is obtained:
\begin{equation}\label{Hamiltioni}
\left\{\begin{aligned}
	dx_i^+(t)&=\bigg\{\tilde{A}x_i^+-\left(\tilde{B}+\tilde{H}L\right)\left(\tilde{R}_1+L^\top\tilde{R}_0L\right)^{-1}\left(\tilde{B}+\tilde{H}L\right)^\top\varphi_i^+\\
	&\quad+\left[\tilde{F}+\tilde{H}\eta-\left(\tilde{B}+\tilde{H}L\right)\left(\tilde{R}_1+L^\top\tilde{R}_0L\right)^{-1}L^\top\tilde{R}_0\eta\right]m^+\\
	&\quad+\left[\tilde{H}\zeta-\left(\tilde{B}+\tilde{H}L\right)\left(\tilde{R}_1+L^\top\tilde{R}_0L\right)^{-1}L^\top\tilde{R}_0\zeta \right]x_0^+\bigg\}dt+\sigma dW_i(t),\\
    d\varphi_i^+(t)&=-\left[\tilde{A}^\top\varphi_i^++\tilde{Q}x_i^+-\tilde{Q}\tilde{\Gamma}_1m^+\right]dt+\psi^+dW_i+\psi_0^+dW_0,\\
    \dot{m}^+(t)&=\left(\tilde{A}+\tilde{F}+\tilde{H}\eta\right)m^++\left(\tilde{B}+\tilde{H}L\right)\bar{u}_1^++\tilde{H}\zeta x_0^+,\\
    dx_0^+(t)&=\left[\left(A+\gamma^{-2}ER_2^{-1}E^\top P_1+B\zeta\right)x_0^++\left(F+\gamma^{-2}ER_2^{-1}E^\top \Pi_1+B\eta\right)m^+\right.\\
    &\quad\,+(H+BL)\bar{u}_1^+\Big]dt+\left[(C+D\zeta)x_0^++(O+D\eta) m^++DL\bar{u}_1^+\right]dW_0(t),\\
    x_i^+(0)&=x,\quad\varphi_i^+(T)=\tilde{G}x_i^+(T)-\tilde{G}\tilde{\Gamma}_2m^+(T),\quad m^+(0)=x,\quad x_0^+(0)=\xi.\\
\end{aligned}\right.
\end{equation}

Sum the first $N$ terms, take the average, and let $N$ approach infinity, we have
\begin{equation}\label{u1+}
\begin{aligned}
	\bar{u}_1^+(t)&:=\lim\limits_{N\to\infty}\frac{1}{N}\sum_{i=1}^{N}u_{1i}^+(t)\\
	&=-\left(\tilde{R}_1+L^\top\tilde{R}_0L\right)^{-1}\left[\left(\tilde{B}+\tilde{H}L\right)^\top\lim\limits_{N\to\infty}\frac{1}{N}\sum_{i=1}^{N}\varphi^+_i+L^\top\tilde{R}_0\zeta x_0^++L^\top\tilde{R}_0\eta m^+\right],\\
 &\hspace{6cm} a.e.\;t\in[0,T],\quad \mathbb{P}\mbox{-}a.s.
\end{aligned}
\end{equation}
Set $\varphi^+(\cdot):=\lim\limits_{N\to\infty}\frac{1}{N}\sum_{i=1}^{N}\varphi^+_i(\cdot)$, by the law of large numbers, $\varphi^+(\cdot)$ satisfies the following BSDE:
\begin{equation}
\left\{\begin{aligned}
	d\varphi^+(t)&=-\left[\tilde{A}^\top\varphi^++\left(\tilde{Q}-\tilde{Q}\tilde{\Gamma}_1\right)m^+\right]dt+\varphi_0^+dW_0(t),\\
	\varphi^+(T)&=\left(\tilde{G}-\tilde{G}\tilde{\Gamma}_2\right)m^+(T).
\end{aligned}\right.
\end{equation}
Substituting (\ref{u1+}) into the equations satisfied by $(m^+(\cdot),x_0^+(\cdot))$ in the Hamiltonian system (\ref{Hamiltioni}), we deduce the following CC system:
\begin{equation}\label{cc}
\left\{\begin{aligned}
	\dot{m}^+(t)&=\hat{A}_1m^++\hat{B}_1x_0^++\hat{H}_1\varphi^+,\\
	dx_0^+(t)&=\left[\hat{A}_2x_0^++\hat{B}_2m^++\hat{H}_2\varphi^+\right]dt+\left[\hat{A}_3x_0^++\hat{B}_3m^++\hat{H}_3\varphi^+\right]dW_0(t),\\
	d\varphi^+(t)&=-\left[\tilde{A}^\top\varphi^++\left(\tilde{Q}-\tilde{Q}\tilde{\Gamma}_1\right)m^+\right]dt+\varphi_0^+dW_0(t),\\
	 x_0^+(0)&=\xi,\quad x_0^+(0)=\xi,\quad \varphi^+(T)=\left(\tilde{G}-\tilde{G}\tilde{\Gamma}_2\right)m^+(T),\\
\end{aligned}\right.
\end{equation}
where we have denoted
\begin{equation}\label{notations}
\begin{aligned}
	\hat{A}_1(\cdot)&:=\tilde{A}+\tilde{F}+\tilde{H}\eta-\left(\tilde{B}+\tilde{H}L\right)\left(\tilde{R}_1+L^\top\tilde{R}_0L\right)^{-1}L^\top\tilde{R}_0\eta,\\
	\hat{B}_1(\cdot)&:=\tilde{H}\zeta-\left(\tilde{B}+\tilde{H}L\right)\left(\tilde{R}_1+L^\top\tilde{R}_0L\right)^{-1}L^\top\tilde{R}_0\zeta,\\
	\hat{H}_1(\cdot)&:=-\left(\tilde{B}+\tilde{H}L\right)\left(\tilde{R}_1+L^\top\tilde{R}_0L\right)^{-1}(\tilde{B}+\tilde{H}L)^\top,\\
	\hat{A}_2(\cdot)&:=A+\gamma^{-2}ER_2^{-1}E^\top P_1+B\zeta-(H+BL)\left(\tilde{R}_1+L^\top\tilde{R}_0L\right)^{-1}L^\top\tilde{R}_0\zeta,\\
	\hat{B}_2(\cdot)&:=F+\gamma^{-2}ER_2^{-1}E^\top \Pi_1+B\eta-(H+BL)\left(\tilde{R}_1+L^\top\tilde{R}_0L\right)^{-1}L^\top\tilde{R}_0\eta,\\
	\hat{H}_2(\cdot)&:=-(H+BL)\left(\tilde{R}_1+L^\top\tilde{R}_0L\right)^{-1}\left(\tilde{B}+\tilde{H}L\right)^\top,\\
	\hat{A}_3(\cdot)&:=C+D\zeta-DL\left(\tilde{R}_1+L^\top\tilde{R}_0L\right)^{-1}L^\top\tilde{R}_0\zeta,\\
	\hat{B}_3(\cdot)&:=O+D\eta-DL\left(\tilde{R}_1+L^\top\tilde{R}_0L\right)^{-1}L^\top\tilde{R}_0\eta,\\
	\hat{H}_3(\cdot)&:=-DL\left(\tilde{R}_1+L^\top\tilde{R}_0L\right)^{-1}\left(\tilde{B}+\tilde{H}L\right)^\top.
\end{aligned}
\end{equation}

\begin{mythm}\label{feedback}
Let \textbf{(A1)}-\textbf{(A4)} hold, and $\gamma>\hat{\gamma}$. If $u_{1i}^+(\cdot)$ is the decentralized optimal control of Problem \textbf{(F2)'}, then $\bar{u}_1^+(\cdot)=\lim\limits_{N\to\infty}\frac{1}{N}\sum_{i=1}^{N}u_{1i}^+(\cdot)$ can be expressed as
\begin{equation}\label{u1+feedback}
\begin{aligned}
	\bar{u}_1^+(\cdot)&=-\left(\tilde{R}_1+L^\top\tilde{R}_0L\right)^{-1}\left\{\left[L^\top\tilde{R}_0\zeta+\left(\tilde{B}+\tilde{H}L\right)^\top\Theta\right]x_0^+\right.\\
    &\qquad \left.+\left[L^\top\tilde{R}_0\eta+\left(\tilde{B}+\tilde{H}L\right)^\top\Delta\right] m^+\right\},\quad a.e.\;t\in[0,T],\quad \mathbb{P}\mbox{-}a.s.,
\end{aligned}
\end{equation}
and
\begin{equation}\label{ui1+feedback}
\begin{aligned}
	u_{1i}^+(\cdot)&=-\left(\tilde{R}_1+L^\top\tilde{R}_0L\right)^{-1}\left\{\left(\tilde{B}+\tilde{H}L\right)^\top\Sigma x_i^++\left[L^\top\tilde{R}_0\zeta+\left(\tilde{B}+\tilde{H}L\right)^\top\Psi\right] x_0^+\right.\\
    &\qquad \left.+\left[L^\top\tilde{R}_0\eta+\left(\tilde{B}+\tilde{H}L\right)^\top\Phi\right] m^+\right\},\quad a.e.\;t\in[0,T],\quad \mathbb{P}\mbox{-}a.s.
\end{aligned}
\end{equation}
Here, $\Delta(\cdot)$ and $\Theta(\cdot)$ solve the following Riccati equations:
\begin{equation}\label{Delta-Theta}
\left\{\begin{aligned}
	&\dot{\Delta}+\Delta\hat{A}_1+\tilde{A}^\top\Delta+\Delta\hat{H}_1\Delta+\Theta\left(\hat{B}_2+\hat{H}_2\Delta\right)+\tilde{Q}-\tilde{Q}\tilde{\Gamma}_1=0,\\
	&\dot{\Theta}+\Theta\hat{A}_2+\tilde{A}^\top\Theta+\Theta\hat{H}_2\Theta+\Delta\left(\hat{B}_1+\hat{H}_1\Theta\right)=0,\\
	&\Delta(T)=\tilde{G}-\tilde{G}\tilde{\Gamma}_2,\quad\Theta(T)=0,
\end{aligned}\right.
\end{equation}
and $\Sigma(\cdot)$, $\Phi(\cdot)$ and $\Psi(\cdot)$ satisfy:
\begin{equation}\label{decoupleHamiltoni}
\left\{\begin{aligned}
	&\dot{\Sigma}+\Sigma\tilde{A}+\tilde{A}^\top\Sigma+\Sigma\hat{H}_1\Sigma+\tilde{Q}=0,\\
	&\dot{\Phi}+\Phi\left(\hat{A}_1+\hat{H}_1\Delta\right)+\tilde{A}^\top\Phi+\Sigma\hat{H}_1\Phi+\Sigma\left(\hat{A}_1-\tilde{A}\right)+\Psi\left(\hat{B}_2+\hat{H}_2\Delta\right)-\tilde{Q}\tilde{\Gamma}_1=0,\\
	&\dot{\Psi}+\Psi\left(\hat{A}_2+\hat{H}_2\Theta\right)+\tilde{A}^\top\Psi+\Sigma\hat{H}_1\Psi+\Sigma\hat{B}_1+\Phi\left(\hat{B}_1+\hat{H}_1\Theta\right)=0,\\
	&\Sigma(T)=\tilde{G},\quad\Phi(T)=-\tilde{G}\tilde{\Gamma}_2,\quad\Psi(T)=0.
\end{aligned}\right.
\end{equation}
\end{mythm}

\begin{proof}
According to the terminal condition of (\ref{cc}), we conjecture that
\begin{equation}
	\varphi^+(\cdot)=\Delta(\cdot) m^+(\cdot)+\Theta(\cdot) x_0^+(\cdot),
\end{equation}
with $\Delta(T)=\tilde{G}-\tilde{G}\tilde{\Gamma}_2$, $\Theta(T)=0$. Applying It\^{o}'s formula to $\varphi^+(\cdot)$, and by comparing its diffusion terms with the ones in the third equation of (\ref{cc}), we can obtain that $\Delta(\cdot)$ and $\Theta(\cdot)$ solve (\ref{Delta-Theta}).

If the non-symmetric coupled Riccati equations (\ref{Delta-Theta}) are solvable, then the related coupled CC system (\ref{cc}) are decoupled. We have
\begin{equation}
\left\{\begin{aligned}
	\dot{m}^+(t)&=\left(\hat{A}_1+\hat{H}_1\Delta\right)m^++\left(\hat{B}_1+\hat{H}_1\Theta\right)x_0^+,\\
	dx_0^+(t)&=\left[\left(\hat{A}_2+\hat{H}_2\Theta\right)x_0^++\left(\hat{B}_2+\hat{H}_2\Delta\right)m^+\right]dt\\
    &\quad +\left[\left(\hat{A}_3+\hat{H}_3\Theta\right)x_0^++\left(\hat{B}_3+\hat{H}_3\Delta\right)m^+\right]dW_0(t),\\
	m^+(0)&=x,\quad x_0^+(0)=\xi.
\end{aligned}\right.
\end{equation}
Next, let us decouple the Hamiltonian system (\ref{Hamiltioni}). By the above analysis, we only need to decouple the first two equations in (\ref{Hamiltioni}). We suppose
\begin{equation}\label{decouple varphi}
	\varphi_i^+(\cdot)=\Sigma(\cdot) x_i^+(\cdot)+\Phi(\cdot) m^+(\cdot)+\Psi(\cdot) x_0^+(\cdot),
\end{equation}
with $\Sigma(T)=\tilde{G}$, $\Phi(T)=-\tilde{G}\tilde{\Gamma}_2$, $\Psi(T)=0$. Applying It\^{o}'s formula to (\ref{decouple varphi}), it follows that
\begin{equation}
\begin{aligned}
	d\varphi_i^+&=\bigg\{\dot{\Sigma}x_i^++\dot{\Phi}m^++\dot{\Psi}x_0^++\Sigma\left[\tilde{A}x_i^++\left(\hat{A}_1-\tilde{A}\right)m^++\hat{B}_1x_0^++\hat{H}_1\varphi_i^+\right]\\
	&\qquad +\Phi\left[\left(\hat{A}_1+\hat{H}_1\Delta\right)m^++\left(\hat{B}_1+\hat{H}_1\Theta\right)x_0^+\right]\\
    &\qquad +\Psi\left[\left(\hat{A}_2+\hat{H}_2\Theta\right)x_0^++\left(\hat{B}_2+\hat{H}_2\Delta\right)m^+\right]\bigg\}dt\\
	&\quad+\Sigma\sigma dW_i(t)+\Psi\left[\left(\hat{A}_3+\hat{H}_3\Theta\right)x_0^++\left(\hat{B}_3+\hat{H}_3\Delta\right)m^+\right]dW_0(t).
\end{aligned}
\end{equation}
Comparing the above with the diffusion terms of the second equation in (\ref{Hamiltioni}), one gets
\begin{equation*}
\begin{aligned}
	\psi^+(\cdot)&=\Sigma\sigma,\quad \mathbb{P}\mbox{-}a.s.\\
	\psi_0^+(\cdot)&=\Psi\left[\left(\hat{A}_3+\hat{H}_3\Theta\right)x_0^+(\cdot)+\left(\hat{B}_3+\hat{H}_3\Delta\right)m^+(\cdot)\right],\quad \mathbb{P}\mbox{-}a.s.
\end{aligned}
\end{equation*}
By noting (\ref{decouple varphi}) and comparing the coefficients of the drift terms, we can get (\ref{decoupleHamiltoni}).
\end{proof}

\begin{Remark}
From Theorem \ref{feedback}, there exists the following relationship
\begin{equation}
	\Theta(\cdot)=\Psi(\cdot),\quad \Delta(\cdot)=\Sigma(\cdot)+\Phi(\cdot).
\end{equation}
\end{Remark}

For the target of (\ref{incentive form}) to hold as the incentive strategy of the leader $\mathcal{A}_0$, the limit of the arithmetic average of the followers' optimal controls (i.e., $\bar{u}_1^+(\cdot)$) must be
matched with the corresponding team-optimal strategies of the leader (i.e., $\bar{u}_1^*(\cdot)$) from some conditions. Therefore, we assume that the following equation holds:
\begin{equation}\label{achieve}
	\bar{u}_1^+(x_0^+,m^+)(t)=\bar{u}_1^*(x_0^*,m^*)(t),\quad t\in[0,T].
\end{equation}

In fact, from the incentive form (\ref{incentive form}), we have
\begin{equation}
	\bar{u}_0^+(x_0^+,m^+)(t)=L(t)\bar{u}_1^+(x_0^+,m^+)(t)+\zeta(t)x_0^+(t)+\eta(t)m^+(t).
\end{equation}
Subtracting the above equation from (\ref{incentive rel}) yields the following equation
\begin{equation}
\begin{aligned}
	&\bar{u}_0^+(x_0^+,m^+)(t)-\bar{u}_0^*(x_0^*,m^*)(t)\\
	&=L(t)\big[\bar{u}_1^+(x_0^+,m^+)(t)-\bar{u}_1^*(x_0^*,m^*)(t)\big]+\zeta(t)\big(x_0^+(t)-x_0^*(t)\big)+\eta(t)\big(m^+(t)-m^*(t)\big).
\end{aligned}
\end{equation}
Moreover, $(x_0^+(\cdot),m^+(\cdot))$ formally satisfies the following equations:
\begin{equation}
\left\{\begin{aligned}
	dx_0^+(t)&=\left[Ax_0^++B\bar{u}_0^+(x_0^+,m^+)+Fm^++H\bar{u}_1^+(x_0^+,m^+)+Ev^*(x_0^+,m^+)\right]dt\\
	&\quad+\left[Cx_0^++D\bar{u}_0^+(x_0^+,m^+)+Om^+\right]dW_0(t),\\
    \dot{m}^+(t)&=\left(\tilde{A}+\tilde{F}\right)m^++\tilde{B}\bar{u}_1^+(x_0^+,m^+)+\tilde{H}\bar{u}_0^+(x_0^+,m^+),\\
    x_0^+(0)&=\xi,\quad m^+(0)=x.
\end{aligned}\right.
\end{equation}
Subtracting the above equations from the closed-loop system (\ref{closed-loop system}), yields a homogeneous system with zero initial conditions, thus we can obtain
\begin{equation}
    x_0^+(t)=x_0^*(t),\quad m^+(t)=m^*(t),\quad t\in[0,T].
\end{equation}

From the above analysis, the second equation of (\ref{closed-loop rep}) and (\ref{u1+feedback}), the relation (\ref{achieve}) can be expressed as follows:
\begin{equation}
\begin{aligned}
	&-\left(\tilde{R}_1+L^\top\tilde{R}_0L\right)^{-1}\left\{\left[L^\top\tilde{R}_0\zeta+\left(\tilde{B}+\tilde{H}L\right)^\top\Theta\right] x_0^++\left[L^\top\tilde{R}_0\eta+\left(\tilde{B}+\tilde{H}L\right)^\top\Delta\right] m^+\right\}\\
	=&-R_1^{-1}\left[\left(H^\top P_1+\tilde{B}^\top P_2\right)x_0^*+\left(H^\top \Pi_1+\tilde{B}^\top \Pi_2\right)m^*\right],
\end{aligned}
\end{equation}
then we have
\begin{equation}
\begin{aligned}
	&\left(\tilde{R}_1+L^\top\tilde{R}_0L\right)^{-1}\left[L^\top\tilde{R}_0\zeta+\left(\tilde{B}+\tilde{H}L\right)^\top\Theta\right]-R_1^{-1}\left(H^\top P_1+\tilde{B}^\top P_2\right)=0,\\
	&\left(\tilde{R}_1+L^\top\tilde{R}_0L\right)^{-1}\left[L^\top\tilde{R}_0\eta+\left(\tilde{B}+\tilde{H}L\right)^\top\Delta\right]-R_1^{-1}\left(H^\top \Pi_1+\tilde{B}^\top \Pi_2\right)=0.
\end{aligned}
\end{equation}

Summarizing what is stated above, we obtain the following theorem with the incentive strategy of the leader $\mathcal{A}_0$ under the additional condition.
\begin{mythm}\label{thm4.3}
Let \textbf{(A1)}-\textbf{(A4)} hold, and $\gamma>\hat{\gamma}$. If the following CC-incentive system admits a solution $(\Delta^*(\cdot),\Theta^*(\cdot),L^*(\cdot))$:
\begin{equation}\label{cc-incentive}
\left\{\begin{aligned}
	&\dot{\Delta}+\Delta\hat{A}_1+\tilde{A}^\top\Delta+\Delta\hat{H}_1\Delta+\Theta\left(\hat{B}_2+\hat{H}_2\Delta\right)+\tilde{Q}-\tilde{Q}\tilde{\Gamma}_1=0,\\
	&\dot{\Theta}+\Theta\hat{A}_2+\tilde{A}^\top\Theta+\Theta\hat{H}_2\Theta+\Delta\left(\hat{B}_1+\hat{H}_1\Theta\right)=0,\\
    &\left(\tilde{R}_1+L^\top\tilde{R}_0L\right)^{-1}\left[L^\top\tilde{R}_0\zeta+\left(\tilde{B}+\tilde{H}L\right)^\top\Theta\right]-R_1^{-1}\left(H^\top P_1+\tilde{B}^\top P_2\right)=0,\\
	&\left(\tilde{R}_1+L^\top\tilde{R}_0L\right)^{-1}\left[L^\top\tilde{R}_0\eta+\left(\tilde{B}+\tilde{H}L\right)^\top\Delta\right]-R_1^{-1}\left(H^\top \Pi_1+\tilde{B}^\top \Pi_2\right)=0,\\
	&\Delta(T)=\tilde{G}-\tilde{G}\tilde{\Gamma}_2,\quad\Theta(T)=0,
\end{aligned}\right.
\end{equation}
then there exists the approximation incentive strategy set of the leader $\mathcal{A}_0$ for the stochastic mean field system, given by
\begin{equation}\label{approx incentive strategy}
	\Gamma_i^*\left(t,u_{1i}(\cdot),x_0(\cdot),m(\cdot)\right)=L^*(t)u_{1i}(t)+\zeta^*(t)x_0(t)+\eta^*(t)m(t),\\
\end{equation}
where
\begin{equation*}
\begin{aligned}
	\zeta^*(\cdot)&=-\left(R_0+D^\top P_1D\right)^{-1}\left(B^\top P_1+\tilde{H}^\top P_2+D^\top P_1C\right)+L^*R_1^{-1}\left(H^\top P_1+\tilde{B}^\top P_2\right),\\
	\eta^*(\cdot)&=-\left(R_0+D^\top P_1D\right)^{-1}\left(B^\top \Pi_1+\tilde{H}^\top \Pi_2+D^\top P_1O\right)+L^*R_1^{-1}\left(H^\top \Pi_1+\tilde{B}^\top \Pi_2\right).
\end{aligned}
\end{equation*}
\end{mythm}

\begin{proof}
The proof is direct from what is stated prior to this theorem, we omit it here.
\end{proof}

\section{Asymptotic optimality}

\subsection{Leader's asymptotic robust team optimality}

With $(\bar{u}_0^*(\cdot),\bar{u}_1^*(\cdot),v^*(\cdot))$ given by (\ref{closed-loop rep}) and $(x_0^*(\cdot),m^*(\cdot))$ by (\ref{closed-loop system}), for convenience, we set
\begin{equation*}
\left\{
\begin{aligned}
	\bar{u}_0^*&(\cdot)=\Theta_{11}^*x_0^*(\cdot)+\Theta_{12}^*m^*(\cdot),\\
	\bar{u}_1^*&(\cdot)=\Theta_{21}^*x_0^*(\cdot)+\Theta_{22}^*m^*(\cdot),
\end{aligned}
\right.
\end{equation*}
where
\begin{equation}\label{notation}
\begin{aligned}
	\Theta_{11}^*(\cdot)&:=-\left(R_0+D^\top P_1D\right)^{-1}\left(B^\top P_1+\tilde{H}^\top P_2+D^\top P_1C\right),\\
	\Theta_{12}^*(\cdot)&:=-\left(R_0+D^\top P_1D\right)^{-1}\left(B^\top \Pi_1+\tilde{H}^\top \Pi_2+D^\top P_1O\right),\\
	\Theta_{21}^*(\cdot)&:=-R_1^{-1}\left(H^\top P_1+\tilde{B}^\top P_2\right),\quad \Theta_{22}^*(\cdot):=-R_1^{-1}\left(H^\top \Pi_1+\tilde{B}^\top \Pi_2\right),
\end{aligned}
\end{equation}
and $P_1(\cdot)$, $P_2(\cdot)$, $\Pi_1(\cdot)$, $\Pi_2(\cdot)$ are given by (\ref{P}). Therefore, we may design the following mean-field strategies:
\begin{equation}\label{de-sts}
\left\{
\begin{aligned}
	\hat{u}_{0i}^*(\cdot)&=\Theta_{11}^*\hat{x}_0(\cdot)+\Theta_{12}^*\hat{m}(\cdot),\\
    \hat{u}_{1i}^*(\cdot)&=\Theta_{21}^*\hat{x}_0(\cdot)+\Theta_{22}^*\hat{m}(\cdot),\qquad\qquad 1\leq i\leq N,\\
    \hat{v}^*(\cdot)&=\gamma^{-2}R_2^{-1}E^\top\left(P_1\hat{x}_0(\cdot)+\Pi_1\hat{m}(\cdot)\right),\quad \mathbb{P}\mbox{-}a.s.,
\end{aligned}
\right.
\end{equation}
where $(\hat{x}_0(\cdot),\hat{m}(\cdot))$ is the centralized state satisfying the following state equations:
\begin{equation}\label{asy-state1}
\left\{\begin{aligned}
	d\hat{x}_0(t)&=\left[\left(A+\gamma^{-2}ER_2^{-1}E^\top P_1+B\Theta_{11}^*+H\Theta_{21}^*\right)\hat{x}_0+F\hat{x}^{(N)}\right.\\
	&\qquad \left.+\left(\gamma^{-2}ER_2^{-1}E^\top \Pi_1+B\Theta_{12}^*+H\Theta_{22}^*\right)\hat{m}\right]dt\\
    &\quad +\left[\left(C+D\Theta_{11}^*\right)\hat{x}_0+(O+D\Theta_{12}^*)\hat{m}\right]dW_0(t),\\ d\hat{x}^{(N)}(t)&=\left[\left(\tilde{A}+\tilde{F}\right)\hat{x}^{(N)}+\left(\tilde{B}\Theta_{21}^*+\tilde{H}\Theta_{11}^*\right)\hat{x}_0+\left(\tilde{B}\Theta_{22}^*+\tilde{H}\Theta_{12}^*\right)\hat{m}\right]dt\\
    &\quad +\frac{\sigma}{N}\sum_{i=1}^{N}dW_i(t),\\
	d\hat{m}(t)&=\left[\left(\tilde{A}+\tilde{F}+\tilde{B}\Theta_{22}^*+\tilde{H}\Theta_{12}^*\right)\hat{m}+\left(\tilde{B}\Theta_{21}^*+\tilde{H}\Theta_{11}^*\right)\hat{x}_0\right]dt,\\
	\hat{x}_0(0)&=\xi,\quad\hat{x}^{(N)}(0)=x,\quad\hat{m}(0)=x.
\end{aligned}\right.
\end{equation}

\begin{mylem}
Assume that \textbf{(A1)}-\textbf{(A3)} hold, and $\gamma>\hat{\gamma}$. It holds that
\begin{equation*}
	\sup_{t\in[0,T]}\mathbb{E}\left\vert\hat{x}^{(N)}(t)-\hat{m}(t)\right\vert^2=O\left(\frac{1}{N}\right),
\end{equation*}
where $\hat{x}^{(N)}(\cdot)$ and $\hat{m}(\cdot)$ satisfy (\ref{asy-state1}).
\end{mylem}

\begin{proof}
By (\ref{asy-state1}), we have the following dynamics
\begin{equation*}
\left\{\begin{aligned}
	d\left(\hat{x}^{(N)}-\hat{m}\right)&=\left(\tilde{A}+\tilde{F}\right)\left(\hat{x}^{(N)}-\hat{m}\right)dt+\frac{\sigma}{N}\sum_{i=1}^{N}dW_i(t),\\
	\hat{x}^{(N)}(0)-\hat{m}(0)&=0.
\end{aligned}\right.
\end{equation*}
By the standard estimates of linear SDEs, we get
\begin{equation*}
	\mathbb{E}\left\vert\hat{x}^{(N)}(t)-\hat{m}(t)\right\vert^2\leq\frac{C}{N}.
\end{equation*}
Then the lemma follows.
\end{proof}

\begin{mythm}\label{leader asy thm}
Suppose that \textbf{(A1)}-\textbf{(A3)} hold, and $\gamma>\hat{\gamma}$. The set of mean-field strategies $(\hat{u}_0^*(\cdot),\hat{u}_1^*(\cdot),\hat{v}^*(\cdot))$ given by (\ref{de-sts}) has the asymptotic robust team-optimality, i.e.,
\begin{equation}\label{asymptotic robust team-optimality}
	\bigg\vert \mathcal{J}_0\left(\hat{u}_0^*(\cdot),\hat{u}_1^*(\cdot),\hat{v}^*(\cdot)\right)-\inf_{(u_0(\cdot),u_1(\cdot))\in\,\mathcal{U}_{0c}}\sup_{v(\cdot)\in\,\mathcal{U}_{vc}}
\mathcal{J}_0(u_0(\cdot),u_1(\cdot),v(\cdot))\bigg\vert=O\left(\frac{1}{\sqrt{N}}\right).
\end{equation}
\end{mythm}

\begin{proof}
Let $\tilde{u}_{0i}=u_{0i}-\hat{u}_{0i}^*$, $\tilde{u}_{1i}=u_{1i}-\hat{u}_{1i}^*$, $\tilde{v}=v-\hat{v}^*$, $\tilde{x}_0=x_0-\hat{x}_0$, $\tilde{x}^{(N)}=x^{(N)}-\hat{x}^{(N)}$, $\tilde{u}_0^{(N)}=u_0^{(N)}-\hat{u}_0^{*(N)}$, $\tilde{u}_1^{(N)}=u_1^{(N)}-\hat{u}_1^{*(N)}$. Then by (\ref{state0}), (\ref{statei}) and (\ref{asy-state1}), we have
\begin{equation}
\left\{\begin{aligned}
	d\tilde{x}_0(t)&=\left[A\tilde{x}_0+B\tilde{u}_0^{(N)}+F\tilde{x}^{(N)}+H\tilde{u}_1^{(N)}+E\tilde{v}\right]dt
    +\left[C\tilde{x}_0+D\tilde{u}_0^{(N)}+O\tilde{x}^{(N)}\right]dW_0(t),\\
	d\tilde{x}^{(N)}(t)&=\left[\left(\tilde{A}+\tilde{F}\right)\tilde{x}^{(N)}+\tilde{B}\tilde{u}_1^{(N)}
     +\tilde{H}\tilde{u}_0^{(N)}\right]dt+\frac{\sigma}{N}\sum_{i=1}^{N}dW_i,\\
	\tilde{x}_0(0)&=0,\quad\tilde{x}^{(N)}(0)=0.
\end{aligned}\right.
\end{equation}
It follows from (\ref{cost0}) that
\begin{equation}\label{decomp}
	\mathcal{J}_0(u_0(\cdot),u_1(\cdot),v(\cdot))=\mathcal{J}_0(\hat{u}_0^*(\cdot),\hat{u}_1^*(\cdot),\hat{v}^*(\cdot))+\tilde{\mathcal{J}}_0(\tilde{u}_0(\cdot),\tilde{u}_1(\cdot),\tilde{v}(\cdot))+I,
\end{equation}
where
\begin{equation*}
\begin{aligned}
	\tilde{\mathcal{J}}_0(\tilde{u}_0,\tilde{u}_1,\tilde{v})&:=\mathbb{E}\bigg\{\int_0^T\left[\left\vert \tilde{x}_0-\Gamma_1\tilde{x}^{(N)}\right\vert_Q^2 +\left\vert \tilde{u}_0^{(N)}\right\vert_{R_0}^2 +\left\vert \tilde{u}_1^{(N)}\right\vert_{R_1}^2 -\gamma^2\vert \tilde{v}\vert_{R_2}^2\right](t)dt\\
	&\qquad\quad +\left\vert \tilde{x}_0(T)-\Gamma_2\tilde{x}^{(N)}(T)\right\vert_G^2 \bigg\},\\
	I&:=2\mathbb{E}\bigg\{\int_0^T\left[\left\langle Q\left(\tilde{x}_0-\Gamma_1\tilde{x}^{(N)}\right),\hat{x}_0-\Gamma_1\hat{x}^{(N)}\right\rangle+\left\langle R_0\tilde{u}_0^{(N)},\hat{u}_0^{*(N)}\right\rangle\right.\\
\end{aligned}
\end{equation*}
\begin{equation*}
\begin{aligned}
	&\qquad\qquad\qquad +\left.\left\langle R_1\tilde{u}_1^{(N)},\hat{u}_1^{*(N)}\right\rangle-\gamma^2\langle R_2\tilde{v},\hat{v}^*\rangle\right](t)dt\\
	&\qquad\qquad +\left\langle G\left(\tilde{x}_0(T)-\Gamma_2\tilde{x}^{(N)}(T)\right),\hat{x}_0(T)-\Gamma_2\hat{x}^{(N)}(T)\right\rangle\bigg\}.
\end{aligned}
\end{equation*}
Due to the (uniformly) positive-definiteness conditions on $Q(\cdot)$, $R_0(\cdot)$, $R_1(\cdot)$, $R_2(\cdot)$ and $G$, we can get
\begin{equation}
	\inf_{(u_0(\cdot),u_1(\cdot))\in\,\mathcal{U}_{0c}}\sup_{v(\cdot)\in\,\mathcal{U}_{vc}}\tilde{\mathcal{J}}_0(\tilde{u}_0(\cdot),\tilde{u}_1(\cdot),\tilde{v}(\cdot))\geq0.
\end{equation}
By Proposition \ref{Prop3.3} and Proposition \ref{Prop3.5}, Problem \textbf{(L1)} is concave in $v$, then
\begin{equation}
	\inf_{(u_0(\cdot),u_1(\cdot))\in\,\mathcal{U}_{0c}}\tilde{\mathcal{J}}_0(\tilde{u}_0(\cdot),\tilde{u}_1(\cdot),\tilde{v}(\cdot))\leq\tilde{\mathcal{J}}_0(0,0,\tilde{v}(\cdot))\leq0.
\end{equation}
Taking the supremum with respect to $v$ on both sides of the inequality, we get
\begin{equation}
	\sup_{v(\cdot)\in\,\mathcal{U}_{vc}}\inf_{(u_0(\cdot),u_1(\cdot))\in\,\mathcal{U}_{0c}}\tilde{\mathcal{J}}_0(\tilde{u}_0(\cdot),\tilde{u}_1(\cdot),\tilde{v}(\cdot))\leq0.
\end{equation}
By \cite{Mou06}, we have
\begin{equation}\label{decomp2}
	\inf_{(u_0(\cdot),u_1(\cdot))\in\,\mathcal{U}_{0c}}\sup_{v(\cdot)\in\,\mathcal{U}_{vc}}\tilde{\mathcal{J}}_0(\tilde{u}_0(\cdot),\tilde{u}_1(\cdot),\tilde{v}(\cdot))=0.
\end{equation}
Since $\hat{x}^{(N)}(\cdot)=\hat{x}^{(N)}(\cdot)-\hat{m}(\cdot)+\hat{m}(\cdot)$, we can rewrite the equation of $I$ as follows:
\begin{equation}
\begin{aligned}
	I&=2\mathbb{E}\bigg\{\int_0^T\left[\left\langle Q\left(\tilde{x}_0-\Gamma_1\tilde{x}^{(N)}\right),\hat{x}_0-\Gamma_1\hat{m}\right\rangle +\left\langle R_0\tilde{u}_0^{(N)},\hat{u}_0^{*(N)}\right\rangle +\left\langle R_1\tilde{u}_1^{(N)},\hat{u}_1^{*(N)}\right\rangle\right.\\
	&\qquad\qquad\qquad -\gamma^2\langle R_2\tilde{v},\hat{v}^*\rangle -\left\langle Q\left(\tilde{x}_0-\Gamma_1\tilde{x}^{(N)}\right),\Gamma_1(\hat{x}^{(N)}-\hat{m})\right\rangle\Big](t)dt\\
    &\qquad\quad +\left\langle G\left(\tilde{x}_0(T)-\Gamma_2\tilde{x}^{(N)}(T)\right),\hat{x}_0(T)-\Gamma_2\hat{m}(T)\right\rangle\\
	&\qquad\quad -\left\langle G\left(\tilde{x}_0(T)-\Gamma_2\tilde{x}^{(N)}(T)\right),\Gamma_2(\hat{x}^{(N)}(T)-\hat{m}(T))\right\rangle\bigg\}.
\end{aligned}
\end{equation}
Under the notation (\ref{notation}), the equations of $P_1(\cdot)$, $P_2(\cdot)$, $\Pi_1(\cdot)$ and $\Pi_2(\cdot)$ are equivalent to
\begin{equation}
\left\{\begin{aligned}
	&\dot{P}_1+P_1A+A^\top P_1+C^\top P_1C+\gamma^{-2}P_1ER_2^{-1}E^\top P_1+Q\\
	&\quad +\left(P_1B+\Pi_1\tilde{H}+C^\top P_1D\right)\Theta_{11}^*+\left(P_1H+\Pi_1\tilde{B}\right)\Theta_{21}^*=0,\\
	&\dot{\Pi}_1+\Pi_1\left(\tilde{A}+\tilde{F}\right)+A^\top\Pi_1+\gamma^{-2}P_1ER_2^{-1}E^\top\Pi_1+P_1F-Q\Gamma_1+C^\top P_1O\\
	&\quad +\left(P_1B+\Pi_1\tilde{H}+C^\top P_1D\right)\Theta_{12}^*+\left(P_1H+\Pi_1\tilde{B}\right)\Theta_{22}^*=0,\\
	&\dot{P}_2+P_2A+\left(\tilde{A}+\tilde{F}\right)^\top P_2+\gamma^{-2}P_2ER_2^{-1}E^\top P_1+F^\top P_1-\Gamma_1^\top Q+O^\top P_1C\\
	&\quad +\left(P_2B+\Pi_2\tilde{H}+O^\top P_1D\right)\Theta_{11}^*+\left(P_2H+\Pi_2\tilde{B}\right)\Theta_{21}^*=0,\\
	&\dot{\Pi}_2+\Pi_2\left(\tilde{A}+\tilde{F}\right)+\left(\tilde{A}+\tilde{F}\right)^\top\Pi_2+\gamma^{-2}P_2ER_2^{-1}E^\top \Pi_1+P_2F+O^\top P_1O\\
	&\quad +F^\top\Pi_1+\Gamma_1^\top Q\Gamma_1+\left(P_2B+\Pi_2\tilde{H}+O^\top P_1D\right)\Theta_{12}^*+\left(P_2H+\Pi_2\tilde{B}\right)\Theta_{22}^*=0,\\
	&P_1(T)=G,\quad\Pi_1(T)=-G\Gamma_2,\quad P_2(T)=-\Gamma_2^\top G,\quad\Pi_2(T)=\Gamma_2^\top G\Gamma_2.
\end{aligned}\right.
\end{equation}
Applying It\^{o}'s formula to $P_1(\cdot)\hat{x}_0(\cdot)+\Pi_1(\cdot)\hat{m}(\cdot)$, $P_2(\cdot)\hat{x}_0(\cdot)+\Pi_2(\cdot)\hat{m}(\cdot)$, respectively, we obtain
\begin{equation}
\begin{aligned}
	&d(P_1\hat{x}_0+\Pi_1\hat{m})=-\left[A^\top(P_1\hat{x}_0+\Pi_1\hat{m})+C^\top P_1(C\hat{x}_0+O\hat{m})+Q(\hat{x}_0-\Gamma_1\hat{m})+C^\top P_1D(\Theta_{11}^*\hat{x}_0\right.\\
    &\quad +\left.\Theta_{12}^*\hat{m})+P_1F(\hat{m}-\hat{x}^{(N)})\right]dt+\big[P_1(C\hat{x}_0+O\hat{m})+P_1D(\Theta_{11}^*\hat{x}_0+\Theta_{12}^*\hat{m})\big]dW_0(t),\\
	&d(P_2\hat{x}_0+\Pi_2\hat{m})=-\left[\left(\tilde{A}+\tilde{F}\right)^\top(P_2\hat{x}_0+\Pi_2\hat{m})+F^\top(P_1\hat{x}_0+\Pi_1\hat{m})\right. -\Gamma_1^\top Q(\hat{x}_0-\Gamma_1\hat{m})\\
	&\quad+P_2F(\hat{m}-\hat{x}^{(N)})+O^\top P_1(C\hat{x}_0+O\hat{m})\bigg]dt+\big[P_2(C\hat{x}_0+O\hat{m})+P_2D(\Theta_{11}^*\hat{x}_0+\Theta_{12}^*\hat{m})\big]dW_0(t).
\end{aligned}
\end{equation}
Applying It\^{o}'s formula to $\left\langle\tilde{x}_0(\cdot),P_1(\cdot)\hat{x}_0(\cdot)+\Pi_1(\cdot)\hat{m}(\cdot)\right\rangle$ and $\left\langle\tilde{x}^{(N)}(\cdot),P_2(\cdot)\hat{x}_0(\cdot)+\Pi_2(\cdot)\hat{m}(\cdot)\right\rangle$, we can get
\begin{equation}
\begin{aligned}
	&\mathbb{E}\left\langle\tilde{x}_0(T),G(\hat{x}_0(T)-\Gamma_2\hat{m}(T))\right\rangle=\mathbb{E}\int_0^T\left[\left\langle B\tilde{u}_0^{(N)}+F\tilde{x}^{(N)}+H\tilde{u}_1^{(N)}+E\tilde{v},P_1\hat{x}_0+\Pi_1\hat{m}\right\rangle\right.\\
	&\qquad\qquad +\left.\left\langle D^\top P_1(C\hat{x}_0+O\hat{m}),\tilde{u}_0^{(N)}\right\rangle+\left\langle\tilde{x}^{(N)},O^\top P_1(C\hat{x}_0+O\hat{m})\right\rangle\right.\\
	&\qquad\qquad -\left.\left\langle\tilde{x}_0,Q(\hat{x}_0-\Gamma_1\hat{m})+P_1F\left(\hat{m}-\hat{x}^{(N)}\right)\right\rangle +\left\langle D^\top P_1D\hat{u}_0^{*(N)},\tilde{u}_0^{(N)}\right\rangle\right]dt,\\
	&-\mathbb{E}\left\langle\Gamma_2\tilde{x}^{(N)}(T),G(\hat{x}_0(T)-\Gamma_2\hat{m}(T))\right\rangle=\mathbb{E}\int_0^T\left[\left\langle\tilde{B}\tilde{u}_1^{(N)}+\tilde{H}\tilde{u}_0^{(N)},P_2\hat{x}_0+\Pi_2\hat{m}\right\rangle\right.\\
	&\qquad\qquad -\left. \left\langle F\tilde{x}^{(N)},P_1\hat{x}_0+\Pi_1\hat{m}\right\rangle+\left\langle Q\Gamma_1\tilde{x}^{(N)},\hat{x}_0-\Gamma_1\hat{m}\right\rangle -\left\langle\tilde{x}^{(N)},P_2F(\hat{m}-\hat{x}^{(N)})\right\rangle\right. \\
	&\qquad\qquad\left. -\left\langle\tilde{x}^{(N)},O^\top P_1(C\hat{x}_0+O\hat{m})\right\rangle\right]dt.
\end{aligned}
\end{equation}
Adding the above two equations together, we have
\begin{equation}
\begin{aligned}
    &\mathbb{E}\left\langle\tilde{x}_0(T)-\Gamma_2\tilde{x}^{(N)}(T),G(\hat{x}_0(T)-\Gamma_2\hat{m}(T))\right\rangle\\
    &=\mathbb{E}\int_0^T\left[\left\langle\tilde{u}_0^{(N)},B^\top(P_1\hat{x}_0+\Pi_1\hat{m})+\tilde{H}^\top(P_2\hat{x}_0+\Pi_2\hat{m})+D^\top P_1(C\hat{x}_0+O\hat{m})+D^\top P_1D\hat{u}_0^{*(N)}\right\rangle\right.\\
    &\qquad\qquad +\left\langle\tilde{u}_1^{(N)},H^\top(P_1\hat{x}_0+\Pi_1\hat{m})+\tilde{B}^\top(P_2\hat{x}_0+\Pi_2\hat{m})\right\rangle +\left\langle\tilde{v},E^\top(P_1\hat{x}_0+\Pi_1\hat{m})\right\rangle\\
    &\qquad\qquad -\left.\left\langle Q(\tilde{x}_0-\Gamma_1\tilde{x}^{(N)}),\hat{x}_0-\Gamma_1\hat{m}\right\rangle -\left\langle P_1\tilde{x}_0+P_2\tilde{x}^{(N)},F(\hat{m}-\hat{x}^{(N)})\right\rangle\right]dt.
\end{aligned}
\end{equation}
Substituting it into $I$, we have
\begin{equation}
\begin{aligned}
	I&=2\mathbb{E}\bigg\{\int_0^T\left[\left\langle\tilde{u}_0^{(N)},\left(R_0+D^\top P_1D\right)\hat{u}_0^{*(N)}+B^\top(P_1\hat{x}_0+\Pi_1\hat{m})+\tilde{H}^\top(P_2\hat{x}_0+\Pi_2\hat{m})\right.\right.\\
	&\qquad\qquad\quad +\left.D^\top P_1(C\hat{x}_0+O\hat{m})\right\rangle +\left\langle\tilde{u}_1^{(N)},R_1\hat{u}_1^{*(N)}+H^\top(P_1\hat{x}_0+\Pi_1\hat{m})+\tilde{B}^\top(P_2\hat{x}_0+\Pi_2\hat{m})\right\rangle\\
	&\qquad\qquad\quad +\left.\left\langle\tilde{v},E^\top(P_1\hat{x}_0+\Pi_1\hat{m})-\gamma^2R_2\hat{v}^*\right\rangle -\left\langle P_1\tilde{x}_0+P_2\tilde{x}^{(N)},F(\hat{m}-\hat{x}^{(N)})\right\rangle\right]dt\\
	&\qquad\quad +\left\langle G\left(\tilde{x}_0(T)-\Gamma_2\tilde{x}^{(N)}(T)\right),\Gamma_2\left(\hat{m}(T)-\hat{x}^{(N)}(T)\right)\right\rangle\bigg\}.
\end{aligned}
\end{equation}
From (\ref{de-sts}), it implies that
\begin{equation}
\left\{
\begin{aligned}
	&\left(R_0+D^\top P_1D\right)\hat{u}_0^{*(N)}+B^\top(P_1\hat{x}_0+\Pi_1\hat{m})+\tilde{H}^\top(P_2\hat{x}_0+\Pi_2\hat{m})+D^\top P_1(C\hat{x}_0+O\hat{m})=0,\\
	&R_1\hat{u}_1^{*(N)}+H^\top(P_1\hat{x}_0+\Pi_1\hat{m})+\tilde{B}^\top(P_2\hat{x}_0+\Pi_2\hat{m})=0,\\
	&\gamma^2R_2\hat{v}^*-E^\top(P_1\hat{x}_0+\Pi_1\hat{m})=0,\quad \mathbb{P}\mbox{-}a.s.,
\end{aligned}
\right.
\end{equation}
therefore we can obtain
\begin{equation}
\begin{aligned}
	I=2\mathbb{E}\bigg\{\int_0^T\left\langle P_1\tilde{x}_0+P_2\tilde{x}^{(N)},F\left(\hat{x}^{(N)}-\hat{m}\right)\right\rangle dt+\left\langle G\left(\tilde{x}_0-\Gamma_2\tilde{x}^{(N)}\right)(T),\Gamma_2\left(\hat{m}-\hat{x}^{(N)}\right)(T)\right\rangle\bigg\}.
\end{aligned}
\end{equation}
By using H\"{o}lder's inequality and the boundedness of coefficients and solutions, we have
\begin{equation}\label{decomp3}
	\vert I\vert\leq C\left(\mathbb{E}\int_0^T\left\vert\hat{x}^{(N)}(t)-\hat{m}(t)\right\vert^2dt\right)^{\frac{1}{2}} +C\left(\mathbb{E}\left\vert\hat{x}^{(N)}(T)-\hat{m}(T)\right\vert^2\right)^{\frac{1}{2}}\leq\frac{C}{\sqrt{N}}.
\end{equation}
Therefore, from (\ref{decomp}), (\ref{decomp2}) and (\ref{decomp3}), we can get (\ref{asymptotic robust team-optimality}). The theorem follows.
\end{proof}

\subsection{Followers' asymptotic Nash equilibrium}

Theorem \ref{thm4.3} established the leader's (approximate) incentive strategy $\Gamma_i^*$ given by (\ref{approx incentive strategy}) that attains the leader's desired performance (\ref{achieve}) and it sustains followers' Nash equilibrium for the auxiliary limiting Problem \textbf{(F2)'}. Based on the decentralized strategy of Problem \textbf{(F2)'} given by (\ref{ui1+feedback}), we may design the following mean-field strategy:\vspace{-1mm}
\begin{equation}\hspace{-1mm}
\left\{
\begin{aligned}
	\check{u}_{1i}^+(\cdot)&=-\left(\tilde{R}_1+L^{*\top}\tilde{R}_0L^*\right)^{-1}\left\{\left(\tilde{B}+\tilde{H}L^*\right)^\top\Sigma^* \check{x}_i+\left[L^{*\top}\tilde{R}_0\zeta^*\right.\right.\\
    &\quad +\left.\left(\tilde{B}+\tilde{H}L^*\right)^\top\Psi^*\right] \check{x}_0+\left.\left[L^{*\top}\tilde{R}_0\eta^*+\left(\tilde{B}+\tilde{H}L^*\right)^\top\Phi^*\right] \check{m}\right\},\\
    \bar{\check{u}}_1^+(\cdot)&=\lim\limits_{N\to\infty}\frac{1}{N}\sum_{i=1}^{N}\check{u}_{1i}^+(\cdot)=\lim\limits_{N\to\infty}\check{u}_1^{+(N)}(\cdot)=-\left(\tilde{R}_1+L^{*\top}\tilde{R}_0L^*\right)^{-1}\\
    &\quad \times\left\{\left[L^{*\top}\tilde{R}_0\zeta^*+\left(\tilde{B}+\tilde{H}L^*\right)^\top\Theta^*\right] \check{x}_0\right.+\left.\left[L^{*\top}\tilde{R}_0\eta^*+\left(\tilde{B}+\tilde{H}L^*\right)^\top\Delta^*\right] \check{m}\right\},
\end{aligned}
\right.
\end{equation}
where
$(\check{x}_0(\cdot),\check{m}(\cdot),\check{x}_i(\cdot))$ is the centralized state satisfying the following equations (for convenience, we have suppressed the superscript *):\vspace{-1mm}
\begin{equation}\label{asy optimal state}
\left\{\begin{aligned}
	d\check{x}_i&=\left\{\left[\tilde{A}-\left(\tilde{B}+\tilde{H}L\right)\left(\tilde{R}_1+L^\top\tilde{R}_0L\right)^{-1}\left(\tilde{B}+\tilde{H}L\right)^\top\Sigma\right]\check{x}_i+\tilde{F}\check{x}^{(N)}\right.\\
	&\quad\ +\left[\tilde{H}\eta-\left(\tilde{B}+\tilde{H}L\right)\left(\tilde{R}_1+L^\top\tilde{R}_0L\right)^{-1}\left[L^\top\tilde{R}_0\eta+\left(\tilde{B}+\tilde{H}L\right)^\top\Phi\right]\right]\check{m}\\
	&\quad\ +\left.\left[\tilde{H}\eta-\left(\tilde{B}+\tilde{H}L\right)\left(\tilde{R}_1+L^\top\tilde{R}_0L\right)^{-1}\left[L^\top\tilde{R}_0\zeta
    +\left(\tilde{B}+\tilde{H}L\right)^\top\Psi\right]\right]\check{x}_0\right\}dt+\sigma dW_i,\\
	d\check{x}^{(N)}&=\left\{\left[\tilde{A}+\tilde{F}-\left(\tilde{B}+\tilde{H}L\right)\left(\tilde{R}_1+L^\top\tilde{R}_0L\right)^{-1}\left(\tilde{B}+\tilde{H}L\right)^\top\Sigma\right]\check{x}^{(N)}+\left[\tilde{H}\eta\right.\right.\\
	&\quad\ -\left.\left(\tilde{B}+\tilde{H}L\right)\left(\tilde{R}_1+L^\top\tilde{R}_0L\right)^{-1}\left[L^\top\tilde{R}_0\eta+\left(\tilde{B}+\tilde{H}L\right)^\top\Phi\right]\right]\check{m}+\left[\tilde{H}\zeta\right.\\
	&\quad\ \left.\left.-\left(\tilde{B}+\tilde{H}L\right)\left(\tilde{R}_1+L^\top\tilde{R}_0L\right)^{-1}\left[L^\top\tilde{R}_0\zeta+\left(\tilde{B}+\tilde{H}L\right)^\top\Psi\right] \right]\check{x}_0\right\}dt
    +\frac{\sigma}{N}\sum_{i=1}^{N}dW_i,\\
	d\check{m}&=\left\{\left[\tilde{A}+\tilde{F}+\tilde{H}\eta-\left(\tilde{B}+\tilde{H}L\right)\left(\tilde{R}_1+L^\top\tilde{R}_0L\right)^{-1}\left[L^\top\tilde{R}_0\eta
    +\left(\tilde{B}+\tilde{H}L\right)^\top\Delta\right]\right]\check{m}\right.\\
	&\quad\ +\left.\left[\tilde{H}\zeta-\left(\tilde{B}+\tilde{H}L\right)\left(\tilde{R}_1+L^\top\tilde{R}_0L\right)^{-1}\left[L^\top\tilde{R}_0\zeta
    +\left(\tilde{B}+\tilde{H}L\right)^\top\Theta\right]\right]\check{x}_0\right\}dt,\\
	d\check{x}_0&=\left\{\left[A+\gamma^{-2}ER_2^{-1}E^\top P_1+B\zeta-(H+BL)\left(\tilde{R}_1+L^\top\tilde{R}_0L\right)^{-1}\right.\right.\\
	&\quad\ \left.\times\left[L^\top\tilde{R}_0\zeta+\left(\tilde{B}+\tilde{H}L\right)^\top\Psi\right]\right]\check{x}_0+\left[F+\gamma^{-2}ER_2^{-1}E^\top \Pi_1+B\eta\right.\\
    &\quad\ -\left.\left.(H+BL)\left(\tilde{R}_1+L^\top\tilde{R}_0L\right)^{-1}\left[L^\top\tilde{R}_0\eta+\left(\tilde{B}+\tilde{H}L\right)^\top\Delta\right]\right]\check{m}\right\}dt\\
	&\quad\ +\left\{\left[C+D\zeta-DL\left(\tilde{R}_1+L^\top\tilde{R}_0L\right)^{-1}\left[L^\top\tilde{R}_0\zeta+\left(\tilde{B}+\tilde{H}L\right)^\top\Psi\right]\right]\check{x}_0\right.\\
	&\quad\ +\left.\left[O+D\eta-DL\left(\tilde{R}_1+L^\top\tilde{R}_0L\right)^{-1}\left[L^\top\tilde{R}_0\eta+\left(\tilde{B}+\tilde{H}L\right)^\top\Delta\right]\right]\check{m}\right\}dW_0,\\
	\check{x}_i(0)&=x,\quad\check{x}^{(N)}(0)=x,\quad\check{m}(0)=x,\quad \check{x}_0(0)=\xi.
\end{aligned}\right.\vspace{-1mm}
\end{equation}
Moreover, we have
\begin{equation}\label{cost1}
\begin{aligned}
	&\mathcal{J}_i(u_{0i}\left(\check{u}_{1i}^+)(\cdot),\check{u}_{1i}^+(\cdot)\right)
	=\mathbb{E}\biggl\{\int_0^T\left[\left\langle\tilde{Q}\check{x}_i,\check{x}_i\right\rangle +\left\langle\tilde{\Gamma}_1^\top\tilde{Q}\tilde{\Gamma}_1\check{x}^{(N)},\check{x}^{(N)}\right\rangle +\left\langle\eta^\top\tilde{R}_0\eta\check{m},\check{m}\right\rangle \right.\\
	&\quad +\left\langle\zeta^\top\tilde{R}_0\zeta \check{x}_0,\check{x}_0\right\rangle-2\left\langle\tilde{Q}\tilde{\Gamma}_1\check{x}^{(N)},\check{x}_i\right\rangle
	+2\left\langle\eta^\top\tilde{R}_0\zeta\check{x}_0,\check{m}\right\rangle +\left\langle\left(\tilde{R}_1+L^\top\tilde{R}_0L\right)\check{u}_{1i}^+,\check{u}_{1i}^+\right\rangle\\
	&\quad +\left.2\left\langle L^\top\tilde{R}_0\zeta\check{x}_0+L^\top\tilde{R}_0\eta\check{m},\check{u}_{1i}^+\right\rangle\right]dt +\left\langle\tilde{G}\check{x}_i(T),\check{x}_i(T)\right\rangle\\
	&\quad -2\left\langle\tilde{G}\tilde{\Gamma}_2\check{x}^{(N)}(T),\check{x}_i(T)\right\rangle +\left\langle\tilde{\Gamma}_2^\top\tilde{G}\tilde{\Gamma}_2\check{x}^{(N)}(T),\check{x}^{(N)}(T)\right\rangle\bigg\},
\end{aligned}
\end{equation}
where $u_{0i}(\check{u}_{1i}^+)(\cdot)=\Gamma_i^*(\cdot,\check{u}_{1i}^+,\check{x}_0,\check{m})$ is given by (\ref{approx incentive strategy}).

By the law of large numbers, it follows unequivocally that $\lim\limits_{N\to\infty}\check{x}^{(N)}(\cdot)=\check{m}(\cdot)$. Moreover, by (\ref{Hamiltioni}), we have $\lim\limits_{N\to\infty}\check{x}_i(\cdot)=x_i^+(\cdot)$, $\check{m}(\cdot)=m^+(\cdot)$ and $\check{x}_0(\cdot)=x_0^+(\cdot)$. Then $\lim\limits_{N\to\infty}\check{u}_{1i}^+(\cdot)=u_{1i}^+(\cdot)$ and $\bar{\check{u}}_1^+(\cdot)=\bar{u}_1^+(\cdot)$ given by (\ref{ui1+feedback}) and (\ref{u1+feedback}).
For the convenience of subsequent discussion, the equations satisfied by $(x_i^+(\cdot),m^+(\cdot),x_0^+(\cdot))$ are recalled:\vspace{-1mm}
\begin{equation}\label{asy lim-optimal state}\vspace{-4mm}
\left\{\begin{aligned}
	dx_i^+&=\left\{\left[\tilde{A}-\left(\tilde{B}+\tilde{H}L\right)\left(\tilde{R}_1+L^\top\tilde{R}_0L\right)^{-1}\left(\tilde{B}+\tilde{H}L\right)^\top\Sigma\right]x_i^+\right.\\
	&\quad\ +\left[\tilde{F}+\tilde{H}\eta-\left(\tilde{B}+\tilde{H}L\right)\left(\tilde{R}_1+L^\top\tilde{R}_0L\right)^{-1}\left[L^\top\tilde{R}_0\eta+\left(\tilde{B}+\tilde{H}L\right)^\top\Phi\right]\right]m^+\\
	&\quad\ +\left.\left[\tilde{H}\zeta-\left(\tilde{B}+\tilde{H}L\right)\left(\tilde{R}_1+L^\top\tilde{R}_0L\right)^{-1}\left[L^\top\tilde{R}_0\zeta+\left(\tilde{B}+\tilde{H}L\right)^\top\Psi\right]\right]x_0^
	+\right\}dt+\sigma dW_i,\\
	dm^+&=\left\{\left[\tilde{A}+\tilde{F}+\tilde{H}\eta-\left(\tilde{B}+\tilde{H}L\right)\left(\tilde{R}_1+L^\top\tilde{R}_0L\right)^{-1}\left[L^\top\tilde{R}_0\eta
	+\left(\tilde{B}+\tilde{H}L\right)^\top\Delta\right]\right]m^+\right.\\
	&\quad\ +\left.\left[\tilde{H}\zeta-\left(\tilde{B}+\tilde{H}L\right)\left(\tilde{R}_1+L^\top\tilde{R}_0L\right)^{-1}\left[L^\top\tilde{R}_0\zeta+\left(\tilde{B}+\tilde{H}L\right)^\top\Theta\right]\right]x_0^+\right\}dt,\\
	dx_0^+&=\left\{\left[A+\gamma^{-2}ER_2^{-1}E^\top P_1+B\zeta-(H+BL)\left(\tilde{R}_1+L^\top\tilde{R}_0L\right)^{-1}\right.\right.\\
	&\qquad \left.\times\left[L^\top\tilde{R}_0\zeta+\left(\tilde{B}+\tilde{H}L\right)^\top\Psi\right]\right]x_0^+ +\left[F+\gamma^{-2}ER_2^{-1}E^\top \Pi_1+B\eta\right.\\
	&\qquad -\left.\left.(H+BL)\left(\tilde{R}_1+L^\top\tilde{R}_0L\right)^{-1}\left[L^\top\tilde{R}_0\eta+\left(\tilde{B}+\tilde{H}L\right)^\top\Delta\right]\right]m^+\right\}dt\\
	&\quad +\left\{\left[C+D\zeta-DL\left(\tilde{R}_1+L^\top\tilde{R}_0L\right)^{-1}\left[L^\top\tilde{R}_0\zeta+\left(\tilde{B}+\tilde{H}L\right)^\top\Psi\right]\right]x_0^+\right.\\
	&\qquad +\left.\left[O+D\eta-DL\left(\tilde{R}_1+L^\top\tilde{R}_0L\right)^{-1}\left[L^\top\tilde{R}_0\eta+\left(\tilde{B}+\tilde{H}L\right)^\top\Delta\right]\right]m^+\right\}dW_0,\\
	x_i^+(0)&=x,\quad m^+(0)=x,\quad x_0^+(0)=\xi,
\end{aligned}\right.
\end{equation}
and
\begin{equation}\label{cost2}
\begin{aligned}
	&J_i(u_{0i}(u_{1i}^+)(\cdot),u_{1i}^+(\cdot))=\mathbb{E}\left\{\int_0^T\left[\left\langle\tilde{Q}x_i^+,x_i^+\right\rangle +\left\langle\left(\tilde{\Gamma}_1^\top\tilde{Q}\tilde{\Gamma}_1+\eta^\top\tilde{R}_0\eta\right)m^+,m^+\right\rangle\right.\right.\\
	&\quad +\left\langle\zeta^\top\tilde{R}_0\zeta x_0^+,x_0^+\right\rangle -2\left\langle\tilde{Q}\tilde{\Gamma}_1m^+,x_i^+\right\rangle +2\left\langle\eta^\top\tilde{R}_0\zeta x_0^+,m^+\right\rangle \\
	&\quad +\left.\left\langle\left(\tilde{R}_1+L^\top\tilde{R}_0L\right)u_{1i}^+,u_{1i}^+\right\rangle +2\left\langle L^\top\tilde{R}_0\zeta x_0^++L^\top\tilde{R}_0\eta m^+,u_{1i}^+\right\rangle\right]dt\\
	&\quad +\left\langle\tilde{G}x_i^+(T),x_i^+(T)\right\rangle -2\left\langle\tilde{G}\tilde{\Gamma}_2m^+(T),x_i^+(T)\right\rangle +\left\langle\tilde{\Gamma}_2^\top\tilde{G}\tilde{\Gamma}_2m^+(T),m^+(T)\right\rangle\bigg\}.
\end{aligned}
\end{equation}

\begin{mylem}\label{est lemma1}
Let \textbf{(A1)}-\textbf{(A4)} hold, and $\gamma>\hat{\gamma}$. It follows\vspace{-1mm}
\begin{equation}\label{est diff1}
\begin{aligned}
	&\sup_{0\leq t\leq T}\mathbb{E}\left\vert\check{x}^{(N)}(t)-m^+(t)\right\vert^2=O\left(\frac{1}{N}\right),\\
	&\sup_{0\leq t\leq T}\mathbb{E}\left\vert\check{x}_i(t)-x_i^+(t)\right\vert^2=O\left(\frac{1}{N}\right),
\end{aligned}
\end{equation}
\begin{equation}\label{est diff2}
	\mathbb{E}\int_0^T\left\vert\check{u}_{1i}^+(t)-u_{1i}^+(t)\right\vert^2dt=O\left(\frac{1}{N}\right).
\end{equation}
\end{mylem}

\begin{proof}
From (\ref{asy optimal state}) and (\ref{asy lim-optimal state}), the difference $\check{x}_i(\cdot)-x_i^+(\cdot)$ satisfies
\begin{equation}
\left\{\begin{aligned}
	d(\check{x}_i-x_i^+)&=\left\{\left[\tilde{A}-\left(\tilde{B}+\tilde{H}L\right)\left(\tilde{R}_1+L^\top\tilde{R}_0L\right)^{-1}\left(\tilde{B}+\tilde{H}L\right)^\top\Sigma\right](\check{x}_i-x_i^+)\right.\\
    &\qquad +\tilde{F}\left(\check{x}^{(N)}-m^+\right)\bigg\}dt,\\
	d(\check{x}^{(N)}-m^+)&=\left\{\left[\tilde{A}+\tilde{F}-\left(\tilde{B}+\tilde{H}L\right)\left(\tilde{R}_1+L^\top\tilde{R}_0L\right)^{-1}\left(\tilde{B}+\tilde{H}L\right)^\top\Sigma\right]\right.\\
	&\qquad \times\left(\check{x}^{(N)}-m^+\right)\bigg\}dt+\frac{\sigma}{N}\sum_{i=1}^{N}dW_i,\\
	\check{x}_i(0)-x_i^+(0)&=0,\quad \check{x}^{(N)}(0)-m^+(0)=0.\\
\end{aligned}\right.
\end{equation}
By linear SDEs' estimates, we can obtain (\ref{est diff1}).
Moreover,
\begin{equation*}
	\check{u}_{1i}^+(\cdot)-u_{1i}^+(\cdot)=-\left(\tilde{R}_1+L^{\top}\tilde{R}_0L\right)^{-1}\left(\tilde{B}+\tilde{H}L\right)^\top\Sigma\left(\check{x}_i(\cdot)-x_i^+(\cdot)\right).
\end{equation*}
Consequently, by (\ref{est diff1}), we have (\ref{est diff2}).
\end{proof}

\begin{mylem}\label{est lemma2}
Let \textbf{(A1)}-\textbf{(A4)} hold, and $\gamma>\hat{\gamma}$. Then
\begin{equation}
	\big\vert\mathcal{J}_i(u_{0i}(\check{u}_{1i}^+)(\cdot),\check{u}_{1i}^+(\cdot))-J_i(u_{0i}(u_{1i}^+)(\cdot),u_{1i}^+(\cdot))\big\vert=O\left(\frac{1}{\sqrt{N}}\right).
\end{equation}
\end{mylem}

\begin{proof}
By (\ref{cost1}) and (\ref{cost2}), the difference $\mathcal{J}_i(u_{0i}(\check{u}_{1i}^+)(\cdot),\check{u}_{1i}^+(\cdot))-J_i(u_{0i}(u_{1i}^+)(\cdot),u_{1i}^+(\cdot))$ can be rewritten as
\begin{equation}
\begin{aligned}
	&\mathcal{J}_i(u_{0i}(\check{u}_{1i}^+)(\cdot),\check{u}_{1i}^+(\cdot))-J_i(u_{0i}(u_{1i}^+)(\cdot),u_{1i}^+(\cdot))\\
    &=\mathbb{E}\bigg\{\int_0^T\left[\left\langle\tilde{Q}(\check{x}_i-x_i^+),\check{x}_i-x_i^+\right\rangle +2\left\langle\tilde{Q}x_i^+,\check{x}_i-x_i^+\right\rangle +\left\langle\tilde{\Gamma}_1^\top\tilde{Q}\tilde{\Gamma}_1\left(\check{x}^{(N)}-m^+\right),\check{x}^{(N)}-m^+\right\rangle\right.\\
    &\qquad\quad +2\left\langle\tilde{\Gamma}_1^\top\tilde{Q}\tilde{\Gamma}_1m^+,\check{x}^{(N)}-m^+\right\rangle -2\left\langle\tilde{Q}\tilde{\Gamma}_1(\check{x}^{(N)}-m^+),\check{x}_i\right\rangle -2\left\langle\tilde{Q}\tilde{\Gamma}_1m^+,\check{x}_i-x_i^+\right\rangle\\
	&\qquad\quad +\left\langle\left(\tilde{R}_1+L^\top\tilde{R}_0L\right)(\check{u}_{1i}^+-u_{1i}^+),\check{u}_{1i}^+-u_{1i}^+\right\rangle
    +2\left\langle\left(\tilde{R}_1+L^\top\tilde{R}_0L\right)u_{1i}^+,\check{u}_{1i}^+-u_{1i}^+\right\rangle\\
	&\qquad\quad +\left.2\left\langle L^\top\tilde{R}_0\zeta x_0^++L^\top\tilde{R}_0\eta m^+,\check{u}_{1i}^+-u_{1i}^+\right\rangle\right]dt
    +\left\langle\tilde{G}\left(\check{x}_i(T)-x_i^+(T)\right),\check{x}_i(T)-x_i^+(T)\right\rangle\\
	&\qquad +2\left\langle\tilde{G}x_i^+(T),\check{x}_i(T)-x_i^+(T)\right\rangle +\left\langle\tilde{\Gamma}_2^\top\tilde{G}\tilde{\Gamma}_2\left(\check{x}^{(N)}(T)-m^+(T)\right),\check{x}^{(N)}(T)-m^+(T)\right\rangle\\
	&\qquad +2\left\langle\tilde{\Gamma}_2^\top\tilde{G}\tilde{\Gamma}_2m^+(T),\check{x}^{(N)}(T)-m^+(T)\right\rangle -2\left\langle\tilde{G}\tilde{\Gamma}_2\left(\check{x}^{(N)}(T)-m^+(T)\right),\check{x}_i(T)\right\rangle\\
	&\qquad -2\left\langle\tilde{G}\tilde{\Gamma}_2m^+(T),\check{x}_i(T)-x_i^+(T)\right\rangle\bigg\}.
\end{aligned}
\end{equation}
Noting the fact that $\mathbb{E}\int_0^T\big[\vert x_i^+(t)\vert^2+\vert m^+(t)\vert^2\big]dt<\infty$, $\mathbb{E}\int_0^T\vert\check{x}_i(t)\vert^2dt<\infty$ and from (\ref{est diff1}), it holds
\begin{equation}
	\big\vert\mathcal{J}_i(u_{0i}(\check{u}_{1i}^+)(\cdot),\check{u}_{1i}^+(\cdot))-J_i(u_{0i}(u_{1i}^+)(\cdot),u_{1i}^+(\cdot))\big\vert\leq\frac{C}{\sqrt{N}}.
\end{equation}
The lemma follows.
\end{proof}

For any $u_{1i}(\cdot)\in\mathcal{U}_{ic}$, consider the set of strategies $(u_{1i}(\cdot),\check{u}_{-1i}^+(\cdot))$. Obviously,
\begin{equation*}
	\lim\limits_{N\to\infty}\check{u}_{-1}^{(N)}(\cdot):=\lim\limits_{N\to\infty}\frac{1}{N}\left(\sum_{j\neq i}^{N}\check{u}_{1j}^+(\cdot)+u_{1i}(\cdot)\right)=\bar{\check{u}}_1^+(\cdot)=\bar{u}_1^+(\cdot).
\end{equation*}
The corresponding state processes are
\begin{equation}\label{asy state}\hspace{-2mm}
\left\{\begin{aligned}
	dx_i&=\left[\tilde{A}x_i+\tilde{F}x^{(N)}+\left(\tilde{B}+\tilde{H}L\right)u_{1i}+\tilde{H}\zeta x_0+\tilde{H}\eta m\right]dt+\sigma dW_i,\\
	dx_j&=\left[\tilde{A}x_j+\tilde{F}x^{(N)}+\left(\tilde{B}+\tilde{H}L\right)\check{u}_{1j}^++\tilde{H}\zeta x_0+\tilde{H}\eta m\right]dt+\sigma dW_j,\\
	dx^{(N)}&=\left[\left(\tilde{A}+\tilde{F}\right)x^{(N)}+\left(\tilde{B}+\tilde{H}L\right)\check{u}_1^{(N)}+\frac{1}{N}\left(\tilde{B}+\tilde{H}L\right)(u_{1i}-\check{u}_{1i}^+)\right.\\
	&\qquad +\tilde{H}\zeta x_0+\tilde{H}\eta m\bigg]dt+\frac{\sigma}{N}\sum_{i=1}^{N}dW_i,\\
	dm&=\left[\left(\tilde{A}+\tilde{F}+\tilde{H}\eta\right)m+\left(\tilde{B}+\tilde{H}L\right)\bar{u}_1^++\tilde{H}\zeta x_0\right]dt,\\
	dx_0&=\left[\left(A+\gamma^{-2}ER_2^{-1}E^\top P_1+B\zeta\right)x_0+\left(F+\gamma^{-2}ER_2^{-1}E^\top \Pi_1+B\eta\right)m\right.\\
    &\qquad +(H+BL)\bar{u}_1^+\Big]dt+\left[(C+D\zeta)x_0+(O+D\eta) m+DL\bar{u}_1^+\right]dW_0,\\
	x_i(0)&=x,\quad x_j(0)=x,\quad x^{(N)}(0)=x,\quad m(0)=x,\quad x_0(0)=\xi,\quad 1\leq j\leq N,\quad j\neq i,
\end{aligned}\right.
\end{equation}
and the cost functional is
\begin{equation}
\begin{aligned}
	&\mathcal{J}_i\left(u_{0i}(u_{1i})(\cdot),u_{-0i}(\check{u}_{-1i}^+)(\cdot),u_{1i}(\cdot),\check{u}_{-1i}^+(\cdot)\right)\\
	&=\mathbb{E}\bigg\{\int_0^T\left[\left\langle\tilde{Q}x_i,x_i\right\rangle +\left\langle\tilde{\Gamma}_1^\top\tilde{Q}\tilde{\Gamma}_1x^{(N)},x^{(N)}\right\rangle
    +\left\langle\eta^\top\tilde{R}_0\eta m,m\right\rangle +\left\langle\zeta^\top\tilde{R}_0\zeta x_0,x_0\right\rangle\right.\\
	&\qquad\quad -2\left\langle\tilde{Q}\tilde{\Gamma}_1x^{(N)},x_i\right\rangle +2\left\langle\eta^\top\tilde{R}_0\zeta x_0,m\right\rangle +\left\langle\left(\tilde{R}_1+L^\top\tilde{R}_0L\right)u_{1i},u_{1i}\right\rangle\\
	&\qquad\quad +\left.2\left\langle L^\top\tilde{R}_0\zeta x_0+L^\top\tilde{R}_0\eta m,u_{1i}\right\rangle\right]dt +\left\langle\tilde{G}x_i(T),x_i(T)\right\rangle\\
	&\qquad\quad -2\left\langle\tilde{G}\tilde{\Gamma}_2x^{(N)}(T),x_i(T)\right\rangle +\left\langle\tilde{\Gamma}_2^\top\tilde{G}\tilde{\Gamma}_2x^{(N)}(T),x^{(N)}(T)\right\rangle\bigg\}.
\end{aligned}
\end{equation}
Here, since $\bar{\check{u}}_1^+(\cdot)=\bar{u}_1^+(\cdot)$, thus $m(\cdot)=m^+(\cdot)$ and $x_0(\cdot)=x_0^+(\cdot)$.
When $N$ tends to infinity, $\check{u}_{1j}^+(\cdot)\to u_{1j}^+(\cdot)$, and the corresponding limiting system is as follows:
\begin{equation}\label{asy lim state}
\left\{\begin{aligned}
	d\bar{x}_i&=\left[\tilde{A}\bar{x}_i+\left(\tilde{F}+\tilde{H}\eta\right)m+\left(\tilde{B}+\tilde{H}L\right)u_{1i}+\tilde{H}\zeta x_0\right]dt+\sigma dW_i,\\
	d\bar{x}_j&=\left[\tilde{A}\bar{x}_j+\left(\tilde{F}+\tilde{H}\eta\right)m+\left(\tilde{B}+\tilde{H}L\right)u_{1j}^++\tilde{H}\zeta x_0\right]dt+\sigma dW_j,\\
	dm&=\left[\left(\tilde{A}+\tilde{F}+\tilde{H}\eta\right)m+\left(\tilde{B}+\tilde{H}L\right)\bar{u}_1^++\tilde{H}\zeta x_0\right]dt,\\
	dx_0&=\left[\left(A+\gamma^{-2}ER_2^{-1}E^\top P_1+B\zeta\right)x_0+\left(F+\gamma^{-2}ER_2^{-1}E^\top \Pi_1+B\eta\right)m\right.\\
	&\qquad +(H+BL)\bar{u}_1^+\Big]dt +\left[(C+D\zeta)x_0+(O+D\eta) m+DL\bar{u}_1^+\right]dW_0,\\
	\bar{x}_i(0)&=x,\quad \bar{x}_j(0)=x,\quad m(0)=x,\quad x_0(0)=\xi,\quad 1\leq j\leq N,\quad j\neq i,
\end{aligned}\right.
\end{equation}
with
\begin{equation}
\begin{aligned}
	&J_i\left(u_{0i}(u_{1i})(\cdot),u_{-0i}(u_{-1i}^+)(\cdot),u_{1i}(\cdot),u_{-1i}^+(\cdot)\right)\\
	&=\mathbb{E}\bigg\{\int_0^T\left[\left\langle\tilde{Q}\bar{x}_i,\bar{x}_i\right\rangle +\left\langle\left(\eta^\top\tilde{R}_0\eta+\tilde{\Gamma}_1^\top\tilde{Q}\tilde{\Gamma}_1\right)m,m\right\rangle
    +\left\langle\zeta^\top\tilde{R}_0\zeta x_0,x_0\right\rangle\right.\\
	&\qquad\quad -2\left\langle\tilde{Q}\tilde{\Gamma}_1m,\bar{x}_i\right\rangle +2\left\langle\eta^\top\tilde{R}_0\zeta x_0,m\right\rangle +\left\langle\left(\tilde{R}_1+L^\top\tilde{R}_0L\right)u_{1i},u_{1i}\right\rangle\\
	&\qquad\quad +\left.2\left\langle L^\top\tilde{R}_0\zeta x_0+L^\top\tilde{R}_0\eta m,u_{1i}\right\rangle\right]dt +\left\langle\tilde{G}\bar{x}_i(T),\bar{x}_i(T)\right\rangle\\
	&\qquad\quad -2\left\langle\tilde{G}\tilde{\Gamma}_2m(T),\bar{x}_i(T)\right\rangle +\left\langle\tilde{\Gamma}_2^\top\tilde{G}\tilde{\Gamma}_2m(T),m(T)\right\rangle\bigg\}.
\end{aligned}
\end{equation}

\begin{mylem}
Let \textbf{(A1)}-\textbf{(A4)} hold, and $\gamma>\hat{\gamma}$. Then the following estimations hold
\begin{equation}\label{est diff3}
\begin{aligned}
	&\sup_{0\leq t\leq T}\mathbb{E}\left\vert x^{(N)}(t)-m(t)\right\vert^2=O\left(\frac{1}{N}\right),\\
	&\sup_{0\leq t\leq T}\mathbb{E}\left\vert x_i(t)-\bar{x}_i(t)\right\vert^2=O\left(\frac{1}{N}\right),\\
	&\sup_{0\leq t\leq T}\mathbb{E}\left\vert x_j(t)-\bar{x}_j(t)\right\vert^2=O\left(\frac{1}{N}\right),\quad 1\leq j\leq N,\quad j\neq i.
\end{aligned}
\end{equation}
\end{mylem}

\begin{proof}
Subtract (\ref{asy state}) from (\ref{asy lim state}), we have
\begin{equation}
\left\{\begin{aligned}
	d(x_i-\bar{x}_i)&=\left[\tilde{A}(x_i-\bar{x}_i)+\tilde{F}\left(x^{(N)}-m\right)\right]dt,\\
	d(x_j-\bar{x}_j)&=\left[\tilde{A}(x_j-\bar{x}_j)+\tilde{F}\left(x^{(N)}-m\right)+\left(\tilde{B}+\tilde{H}L\right)(\check{u}_{1j}^+-u_{1j}^+)\right]dt,\\
	d(x^{(N)}-m)&=\left[\left(\tilde{A}+\tilde{F}\right)\left(x^{(N)}-m\right)+\left(\tilde{B}+\tilde{H}L\right)\left(\check{u}_1^{(N)}-\bar{u}_1^+\right)\right.\\
	&\qquad +\left.\frac{1}{N}\left(\tilde{B}+\tilde{H}L\right)(u_{1i}-\check{u}_{1i}^+)\right]dt+\frac{\sigma}{N}\sum_{i=1}^{N}dW_i,\\
	x_i(0)-\bar{x}_i(0)&=0,\quad x_j(0)-\bar{x}_j(0)=0,\quad x^{(N)}(0)-m(0)=0.
\end{aligned}\right.
\end{equation}
Since
\begin{equation}
\left\{
\begin{aligned}
	\check{u}_1^{(N)}(\cdot)&=-\left(\tilde{R}_1+L^\top\tilde{R}_0L\right)^{-1}\left\{\left(\tilde{B}+\tilde{H}L\right)^\top\Sigma \check{x}^{(N)}+\left[L^\top\tilde{R}_0\zeta+\left(\tilde{B}+\tilde{H}L\right)^\top\Psi\right]x_0^+\right.\\
	&\qquad +\left.\left[L^\top\tilde{R}_0\eta+\left(\tilde{B}+\tilde{H}L\right)^\top\Phi\right] m^+\right\},\\
	\bar{u}_1^+(\cdot)&=-\left(\tilde{R}_1+L^\top\tilde{R}_0L\right)^{-1}\left\{\left[L^\top\tilde{R}_0\zeta+\left(\tilde{B}+\tilde{H}L\right)^\top\Theta\right]x_0^+\right.\\
    &\qquad +\left.\left[L^\top\tilde{R}_0\eta+\left(\tilde{B}+\tilde{H}L\right)^\top\Delta\right] m^+\right\},
\end{aligned}
\right.
\end{equation}
and the difference is
\begin{equation}
	\check{u}_1^{(N)}(\cdot)-\bar{u}_1^+(\cdot)=-\left(\tilde{R}_1+L^{\top}\tilde{R}_0L\right)^{-1}\left(\tilde{B}+\tilde{H}L\right)^\top\Sigma\left(\check{x}^{(N)}-m^+\right).
\end{equation}
Thus, by Lemma \ref{est lemma1} and the boundedness of the coefficients, we obtain
\begin{equation*}
	\mathbb{E}\int_0^T\left\vert\check{u}_1^{(N)}(t)-\bar{u}_1^+(t)\right\vert^2dt=O\left(\frac{1}{N}\right).
\end{equation*}
By (\ref{est diff2}) and linear SDEs' estimates, the lemma follows.
\end{proof}

\begin{mylem}\label{est lemma3}
Let \textbf{(A1)}-\textbf{(A4)} hold, and $\gamma>\hat{\gamma}$. It holds that
\begin{equation}
\begin{aligned}
	&\left\vert\mathcal{J}_i\left(u_{0i}(u_{1i})(\cdot),u_{-0i}(\check{u}_{-1i}^+)(\cdot),u_{1i}(\cdot),\check{u}_{-1i}^+(\cdot)\right)\right.\\
&\quad -\left.J_i\left(u_{0i}(u_{1i})(\cdot),u_{-0i}(u_{-1i}^+)(\cdot),u_{1i}(\cdot),u_{-1i}^+(\cdot)\right)\right\vert=O\left(\frac{1}{\sqrt{N}}\right).
\end{aligned}
\end{equation}
\end{mylem}

\begin{proof}
The proof is similar to Lemma \ref{est lemma2} and omitted.
\end{proof}
From the above, we have the following result.
\begin{mythm}\label{followers asy thm}
Let \textbf{(A1)}-\textbf{(A4)} hold, and $\gamma>\hat{\gamma}$. Assume that the CC-incentive system (\ref{cc-incentive}) admits a solution. For any alternative $u_{1i}(\cdot)$ and the mean-field strategy $(\check{u}_{11}^+(\cdot),\cdots,\check{u}_{1N}^+(\cdot))$, we have
\begin{equation}
	\bar{\check{u}}_1^+(\cdot)=\lim\limits_{N\to\infty}\frac{1}{N}\sum_{i=1}^{N}\check{u}_{1i}^+(\cdot)=\bar{u}_1^+(\cdot)=\bar{u}_1^*(\cdot),
\end{equation}
\begin{equation}
\begin{aligned}
	&\mathcal{J}_i\left(u_{0i}(\check{u}_{1i}^+)(\cdot),\check{u}_{1i}^+(\cdot)\right)\\
&\leq\mathcal{J}_i\left(u_{0i}(u_{1i})(\cdot),u_{-0i}(\check{u}_{-1i}^+)(\cdot),u_{1i}(\cdot),\check{u}_{-1i}^+(\cdot)\right)+O\left(\frac{1}{\sqrt{N}}\right),\quad\forall u_{1i}(\cdot)\in\mathcal{U}_{ic}.
\end{aligned}
\end{equation}
\end{mythm}

\begin{proof}
Using the results mentioned in Lemmas \ref{est lemma2} and \ref{est lemma3}, we have
\begin{equation}
\begin{aligned}
	&\mathcal{J}_i(u_{0i}(\check{u}_{1i}^+)(\cdot),\check{u}_{1i}^+(\cdot))= J_i(u_{0i}(u_{1i}^+)(\cdot),u_{1i}^+(\cdot))+O\left(\frac{1}{\sqrt{N}}\right)\\
	&\leq J_i\left(u_{0i}(u_{1i})(\cdot),u_{-0i}(u_{-1i}^+)(\cdot),u_{1i}(\cdot),u_{-1i}^+(\cdot)\right)+O\left(\frac{1}{\sqrt{N}}\right)\\
	&\leq \mathcal{J}_i\left(u_{0i}(u_{1i})(\cdot),u_{-0i}(\check{u}_{-1i}^+)(\cdot),u_{1i}(\cdot),\check{u}_{-1i}^+(\cdot)\right)+O\left(\frac{1}{\sqrt{N}}\right),
\end{aligned}
\end{equation}
thus the mean-field strategies $\check{u}_1^+(\cdot)=(\check{u}_{11}^+(\cdot),\cdots,\check{u}_{1N}^+(\cdot))$ has the $\hat{\epsilon}$-followers' Nash equilibrium property with $\hat{\epsilon}=O\left(\frac{1}{\sqrt{N}}\right)$.
\end{proof}
From Definition \ref{def}, Theorem \ref{leader asy thm} and \ref{followers asy thm}, we now state the main result of this section.

\begin{mythm}
Let \textbf{(A1)}-\textbf{(A4)} hold, and $\gamma>\hat{\gamma}$. Assume that the CC-incentive system (\ref{cc-incentive}) admits a solution. Then $(\hat{u}_0^*(\cdot),\hat{u}_1^*(\cdot),\hat{v}^*(\cdot))$ given by (\ref{de-sts}) is an $\epsilon$-robust incentive strategy.
\end{mythm}

\section{A numerical example}

Since the solvability of the CC-incentive system (\ref{cc-incentive}) is challenging and entirely new, we attempt to analyze its solvability under special cases before presenting a numerical example. We begin with a special case to illustrate that (\ref{cc-incentive}) has a solution on $[0,T]$ under some conditions.

\begin{Example}
Assume that $\gamma\textgreater\hat{\gamma}$, and system coefficients satisfy $\tilde{R}_0(\cdot)=E(\cdot)=0$,  $H(\cdot)=\tilde{B}(\cdot)$, $A(\cdot)+F(\cdot)=\tilde{A}(\cdot)+\tilde{F}(\cdot)$, $C(\cdot)+O(\cdot)=0$,  $\Gamma_1(\cdot)=\Gamma_2(\cdot)=I$, and $\mathcal{R}(\tilde{B}(t))\subseteq\mathcal{R}(\tilde{H}(t))$, $t\in[0,T]$, then the CC-incentive system (\ref{cc-incentive}) admits a unique solution $(\Delta^*(\cdot),\Theta^*(\cdot),L^*(\cdot))$.
\end{Example}
Consider $\tilde{R}_0(\cdot)=E(\cdot)=0$, from (\ref{notations}) and (\ref{notation}), the equation (\ref{cc-incentive}) can be simplified to:
\begin{equation}\label{eq.1}
\left\{\begin{aligned}
	&\dot{\Delta}+\Delta\big(\tilde{A}+\tilde{F}+\tilde{H}\eta\big)+\tilde{A}^\top\Delta-\Delta\big(\tilde{B}+\tilde{H}L\big)\tilde{R}_1^{-1}(\tilde{B}+\tilde{H}L)^\top\Delta\\
	&\quad+\Theta\big[F+B\eta-\big(H+BL\big)\tilde{R}_1^{-1}\big(\tilde{B}+\tilde{H}L\big)^\top\Delta\big]+\tilde{Q}-\tilde{Q}\tilde{\Gamma}_1=0,\quad\Delta(T)=\tilde{G}-\tilde{G}\tilde{\Gamma}_2,\\
	&\dot{\Theta}+\Theta\big(A+B\zeta\big)+\tilde{A}^\top\Theta-\Theta\big(H+BL\big)\tilde{R}_1^{-1}\big(\tilde{B}+\tilde{H}L\big)^\top\Theta\\
	&\quad+\Delta\big[\tilde{H}\zeta-\big(\tilde{B}+\tilde{H}L\big)\tilde{R}_1^{-1}\big(\tilde{B}+\tilde{H}L\big)^\top\Theta\big]=0,\quad\Theta(T)=0,\\
	&\big(\tilde{B}+\tilde{H}L\big)^\top\Theta+\tilde{R}_1\Theta_{21}^*=0,\qquad \big(\tilde{B}+\tilde{H}L\big)^\top\Delta+\tilde{R}_1\Theta_{22}^*=0.\\
\end{aligned}\right. 
\end{equation}

Substituting $\zeta=\Theta_{11}^*-L\Theta_{21}^*$ and $\eta=\Theta_{12}^*-L\Theta_{22}^*$ into equation (\ref{eq.1}) and using its two equality constraints yields
\begin{equation}\label{eq.2}
\left\{\begin{aligned}
	&\dot{\Delta}+\Delta\big(\tilde{A}+\tilde{F}+\tilde{H}\Theta_{12}^*+\tilde{B}\Theta_{22}^*\big)+\tilde{A}^\top\Delta+\Theta\big(F+B\Theta_{12}^*+H\Theta_{22}^*\big)+\tilde{Q}-\tilde{Q}\tilde{\Gamma}_1=0,\\
	&\dot{\Theta}+\Theta\big(A+B\Theta_{11}^*+H\Theta_{21}^*\big)+\tilde{A}^\top\Theta+\Delta\big(\tilde{H}\Theta_{11}^*+\tilde{B}\Theta_{21}^*\big)=0,\\
	&\big(\tilde{B}+\tilde{H}L\big)^\top\Theta+\tilde{R}_1\Theta_{21}^*=0,\qquad \big(\tilde{B}+\tilde{H}L\big)^\top\Delta+\tilde{R}_1\Theta_{22}^*=0,\\
	&\Delta(T)=\tilde{G}-\tilde{G}\tilde{\Gamma}_2,\qquad\Theta(T)=0.\\
\end{aligned}\right. 
\end{equation}

Let $\tilde{A}(\cdot)+\tilde{F}(\cdot)=A(\cdot)+F(\cdot)$, $C(\cdot)+O(\cdot)=0$, $\Gamma_1(\cdot)=\Gamma_2(\cdot)=I$. From the Riccati equation (\ref{P1})-(\ref{Pi2}), we have
\begin{equation*}
\begin{aligned}
	&\dot{\big(P_1+P_2\big)}+\big(P_1+P_2\big)A+\big(A+F\big)^\top\big(P_1+P_2\big)-\big[\big(P_1+P_2\big)B+\big(\Pi_1+\Pi_2\big)\tilde{H}\big]\\
	&\quad\times\big(R_0+D^\top P_1D\big)^{-1}\big(B^\top P_1+\tilde{H}^\top P_2+D^\top P_1C\big)-\big[\big(P_1+P_2\big)H+\big(\Pi_1+\Pi_2\big)\tilde{B}\big]\\
	&\quad\times\tilde{R}_1^{-1}\big(H^\top P_1+\tilde{B}^\top P_2\big)=0,\qquad\big(P_1+P_2\big)(T)=0,\\
\end{aligned}
\end{equation*}
and
\begin{equation*}
\begin{aligned}
	&\dot{\big(\Pi_1+\Pi_2\big)}+\big(\Pi_1+\Pi_2\big)\big(A+F\big)+\big(P_1+P_2\big)F+\big(A+F\big)^\top\big(\Pi_1+\Pi_2\big)\\
	&\quad-\big[\big(P_1+P_2\big)B+\big(\Pi_1+\Pi_2\big)\tilde{H}\big]\big(R_0+D^\top P_1D\big)^{-1}\big(B^\top \Pi_1+\tilde{H}^\top \Pi_2+D^\top P_1O\big)\\
	&\quad-\big[\big(P_1+P_2\big)H+\big(\Pi_1+\Pi_2\big)\tilde{B}\big]\tilde{R}_1^{-1}\big(H^\top \Pi_1+\tilde{B}^\top \Pi_2\big)=0,\qquad\big(\Pi_1+\Pi_2\big)(T)=0,\\
\end{aligned}
\end{equation*}
Obviously, $P_1(\cdot)+P_2(\cdot)=\Pi_1(\cdot)+\Pi_2(\cdot)=0$ satisfies the above two equations. Similarly, we get $P_1(\cdot)+\Pi_1(\cdot)=P_2(\cdot)+\Pi_2(\cdot)=0$. Then, the relation $P_2(\cdot)=\Pi_1(\cdot)=-P_1(\cdot)=-\Pi_2(\cdot)$ holds.

Therefore, we get
\begin{equation}\label{eq.3}
\begin{aligned}
	\Theta_{11}^*&=-\left(R_0+D^\top P_1D\right)^{-1}\left[\big(B-\tilde{H}\big)^\top P_1+D^\top P_1C\right],\\
    \Theta_{12}^*&=\left(R_0+D^\top P_1D\right)^{-1}\left[\big(B-\tilde{H}\big)^\top P_1+D^\top P_1C\right],\\
    \Theta_{21}^*&=-R_1^{-1}\big(H-\tilde{B}\big)^\top P_1,\quad \Theta_{22}^*=R_1^{-1}\big(H-\tilde{B}\big)^\top P_1,\\
\end{aligned}
\end{equation}
and  $\Theta_{11}^*+\Theta_{12}^*=\Theta_{21}^*+\Theta_{22}^*=0$.

When $H=\tilde{B}$, we have $\Theta_{21}^*=\Theta_{22}^*=0$. If $\mathcal{R}(\tilde{B}(t))\subseteq\mathcal{R}(\tilde{H}(t))$, then there exists a matrix-valued function $N(t)\in L^\infty(0,T;\mathbb{R}^{m_L\times m_F})$ such that $\tilde{B}(t)=\tilde{H}(t)N(t)$. Taking $L^*(t)=-N(t)$, $t\in[0,T]$, from (\ref{approx incentive strategy}), the approximation incentive strategy set of the leader $\mathcal{A}_0$ is 
\begin{equation*}
	\Gamma_i^*(t,u_{1i}(\cdot),x_0(\cdot),m(\cdot))=-N(t)u_{1i}(t)+\Theta_{11}^*(t)x_0(t)+\Theta_{12}^*(t)m(t),\\
\end{equation*}
and both equality constraints of equation (\ref{eq.2}) are satisfied. Moreover, (\ref{eq.2}) reduces to 
\begin{equation}\label{eq.4}
\left\{\begin{aligned}
	&\dot{\Delta}+\Delta\big(\tilde{A}+\tilde{F}+\tilde{H}\Theta_{12}^*\big)+\tilde{A}^\top\Delta+\Theta\big(F+B\Theta_{12}^*\big)+\tilde{Q}-\tilde{Q}\tilde{\Gamma}_1=0,\\
	&\dot{\Theta}+\Theta\big(A+B\Theta_{11}^*\big)+\tilde{A}^\top\Theta+\Delta\tilde{H}\Theta_{11}^*=0,\\
	&\Delta(T)=\tilde{G}-\tilde{G}\tilde{\Gamma}_2,\qquad\Theta(T)=0,\\
\end{aligned}\right. 
\end{equation}
which is equivalent to
\begin{equation}\label{eq.5}
\left\{\begin{aligned}
	&\dot{\big(\Delta+\Theta\big)}+\big(\Delta+\Theta\big)\big(A+F\big)+\tilde{A}^\top\big(\Delta+\Theta\big)+\tilde{Q}-\tilde{Q}\tilde{\Gamma}_1=0,\quad\big(\Delta+\Theta\big)(T)=\tilde{G}-\tilde{G}\tilde{\Gamma}_2,\\
	&\dot{\big(\Delta-\Theta\big)}+\big(\Delta-\Theta\big)A+\tilde{A}^\top\big(\Delta-\Theta\big)+\big(\Delta+\Theta\big)F+\big(\Delta\tilde{H}+\Theta B\big)\big(\Theta_{12}^*-\Theta_{11}^*\big)\\
	&\quad+\tilde{Q}-\tilde{Q}\tilde{\Gamma}_1=0,\quad\big(\Delta-\Theta\big)(T)=\tilde{G}-\tilde{G}\tilde{\Gamma}_2.\\
\end{aligned}\right. 
\end{equation}
Since 
\begin{equation}\label{trans}
	\Delta=\frac{\big(\Delta+\Theta\big)+\big(\Delta-\Theta\big)}{2},\qquad\Theta=\frac{\big(\Delta+\Theta\big)-\big(\Delta-\Theta\big)}{2},\\
\end{equation}
the equation (\ref{eq.5}) can be simplified to
\begin{equation}\label{eq.6}
\left\{\begin{aligned}
	&\dot{\big(\Delta+\Theta\big)}+\big(\Delta+\Theta\big)\big(A+F\big)+\tilde{A}^\top\big(\Delta+\Theta\big)+\tilde{Q}-\tilde{Q}\tilde{\Gamma}_1=0,\quad\big(\Delta+\Theta\big)(T)=\tilde{G}-\tilde{G}\tilde{\Gamma}_2,\\
	&\dot{\big(\Delta-\Theta\big)}+\big(\Delta-\Theta\big)\big[A+\big(B-\tilde{H}\big)\Theta_{11}^*\big]+\tilde{A}^\top\big(\Delta-\Theta\big)+\big(\Delta+\Theta\big)\big[F-\big(B+\tilde{H}\big)\Theta_{11}^*\big]\\
	&\quad+\tilde{Q}-\tilde{Q}\tilde{\Gamma}_1=0,\quad\big(\Delta-\Theta\big)(T)=\tilde{G}-\tilde{G}\tilde{\Gamma}_2.\\
\end{aligned}\right. 
\end{equation}

In fact, the first equation in equation (\ref{eq.6}) is a differential Sylvester equation in $\Delta+\Theta$, which admits a unique explicit solution (see Behr et al. \cite{Behr19})
\begin{equation}\label{eq.7}
\begin{aligned}
	\Delta(t)^*+\Theta^*(t)=\Phi_{\tilde{A}}(T,t)^\top\tilde{G}\big(I-\tilde{\Gamma}_2\big)\Phi_{(A+F)}(T,t)+\int_{t}^{T}\Phi_{\tilde{A}}(s,t)^\top\tilde{Q}\big(I-\tilde{\Gamma}_1\big)\Phi_{(A+F)}(s,t)ds,\\
\end{aligned}
\end{equation}
where $\Phi_{\tilde{A}}(t,s)$ and $\Phi_{(A+F)}(t,s)$ are the unique state-transition matrices of $\tilde{A}(t)$ and $A(t)+F(t)$ with respect to $s\in[0,T]$, respectively.
Then substituting it into the second equation of equation (\ref{eq.6}), which is also a linear differential Sylvester equation, we obtain
\begin{equation}\label{eq.8}
\begin{aligned}
	\Delta(t)^*-\Theta^*(t)=\Phi_{\tilde{A}}(T,t)^\top\tilde{G}\big(I-\tilde{\Gamma}_2\big)\Phi_{N_1}(T,t)+\int_{t}^{T}\Phi_{\tilde{A}}(s,t)^\top N_2(s)\Phi_{N_1}(s,t)ds,\\
\end{aligned}
\end{equation}
where $N_1:=A+\big(B-\tilde{H}\big)\Theta_{11}^*$, $N_2:=\big(\Delta^*+\Theta^*\big)\big[F-\big(B+\tilde{H}\big)\Theta_{11}^*\big]+\tilde{Q}-\tilde{Q}\tilde{\Gamma}_1$, and $\Phi_{N_1}(t,s)$ is the unique state-transition matrix of $N_1(t)$ with respect to $s\in[0,T]$. Thus, using (\ref{trans}), (\ref{eq.7}) and (\ref{eq.8}), we derive the explicit representations of solutions $\Delta^*$ and $\Theta^*$ to equation (\ref{eq.4}). 

In what follows, we present a numerical example to better demonstrate the efficacy of the proposed mean-field strategies. The parameters are listed in Table~\ref{2}. The population of the followers is $N=100$. The horizon length for the simulation is selected as $T=10$.
\begin{table}[H]
\centering
\caption{Simulation parameters}
\begin{tabular}{c c c c c c c c c c c c c c c}
	\hline
	$A$ & $B$ & $F$ & $H$ & $E$ & $C$ & $D$ & $O$ & $\tilde{A}$ & $\tilde{B}$ & $\tilde{F}$ & $\tilde{H}$ & $\sigma$ & $\xi$ & $x$\\
	\hline
	0.3 & 0.5 & 0.6 & 0.7 & -0.5 & 1 & 0.5 & 0 & 0.25 & 0.5 & 0.4 & 0.2 & 0.6 & 1 & 1 \\
	\hline\hline
	$\Gamma_1$ & $Q$ & $R_0$ & $R_1$ & $R_2$ & $\Gamma_2$ & $G$ & $\tilde{\Gamma}_1$ & $\tilde{Q}$ & $\tilde{R}_0$ & $\tilde{R}_1$ & $\tilde{\Gamma}_2$ & $\tilde{G}$ & $\gamma$\\
	\hline
	1 & 0.4 & 0.6 & 0.5 & 0.4 & 0.01 & 1& 0.99998 & 0.01 & 0.15 &0.6 & 1 & 0.01 & 5 \\
	\hline
	\label{2}
\end{tabular}
\end{table}

Through computation, Figure \ref{3} gives the numerical solutions of Riccati equations $P_1(\cdot)$, $\Pi_1(\cdot)$, $P_2(\cdot)$ and $\Pi_2(\cdot)$.

\begin{figure}[H]
	\centering
	\includegraphics[width=0.66\linewidth]{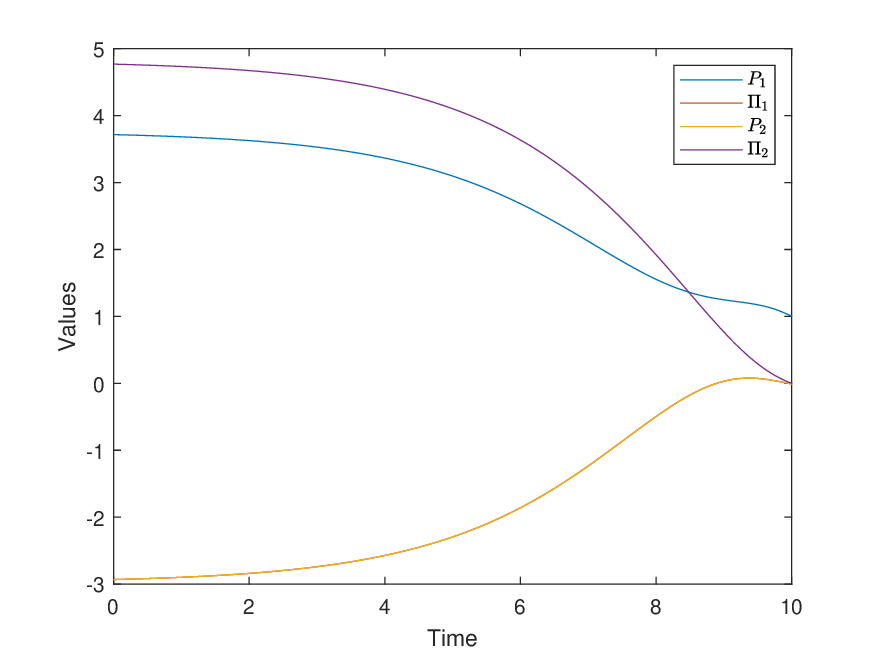}
	\caption{The numerical solutions of Riccati equations of $P_1(\cdot)$, $\Pi_1(\cdot)$, $P_2(\cdot)$ and $\Pi_2(\cdot)$}
	%\label{fig.}
	\label{3}
\end{figure}

The decentralized optimal state trajectories $x_0^*(\cdot)$ and $m^*(\cdot)$ of Problem \textbf{(L2)} are shown in Figure \ref{4}.

\begin{figure}[H]
	\centering
	\includegraphics[width=0.66\linewidth]{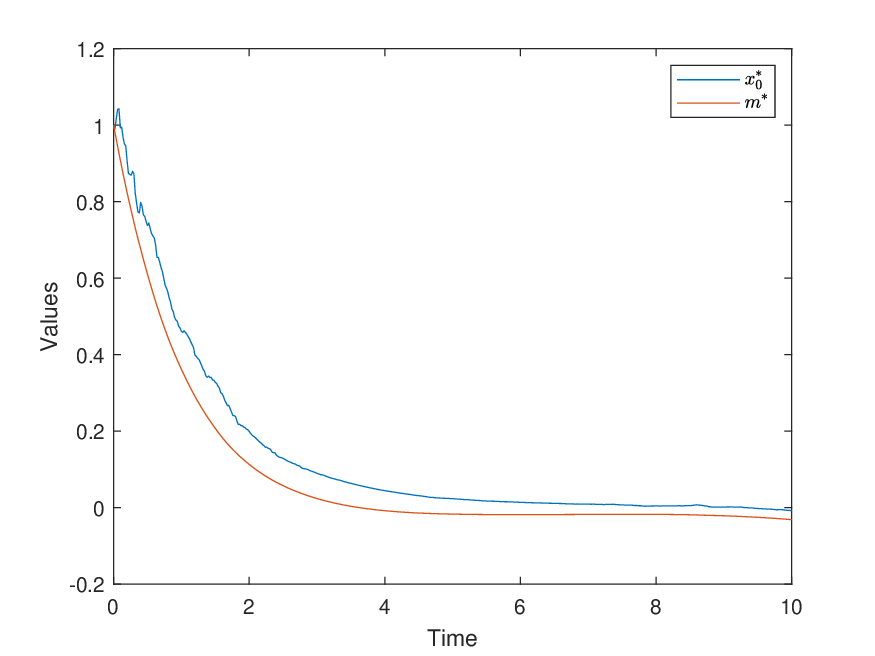}
	\caption{The optimal state trajectories $x_0^*(\cdot)$ and $m^*(\cdot)$ of Problem \textbf{(L2)}}
	%\label{fig.}
	\label{4}
\end{figure}

Figure \ref{5} shows the open-loop saddle points $(\bar{u}^*_0(\cdot),\bar{u}^*_1(\cdot))$ of Problem \textbf{(L2)}.

\begin{figure}[H]
	\centering
	\includegraphics[width=0.66\linewidth]{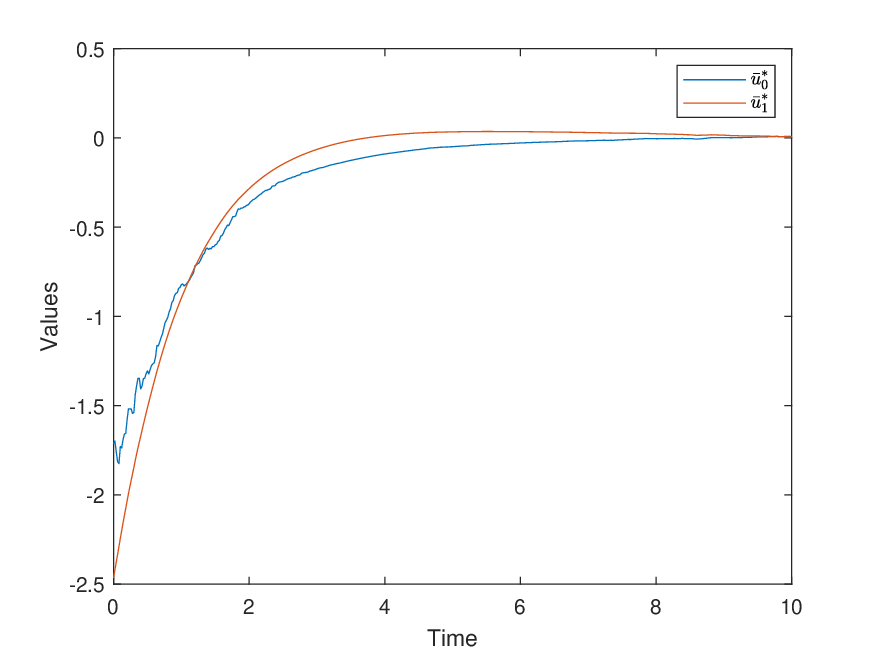}
	\caption{The decentralized optimal control $(\bar{u}^*_0(\cdot),\bar{u}^*_1(\cdot))$ of Problem \textbf{(L2)}}
	%\label{fig.}
	\label{5}
\end{figure}

The worst-case disturbance $v^*(\cdot)$ of Problem \textbf{(L2)} is shown in Figure \ref{6}.

\begin{figure}[H]
	\centering
	\includegraphics[width=0.66\linewidth]{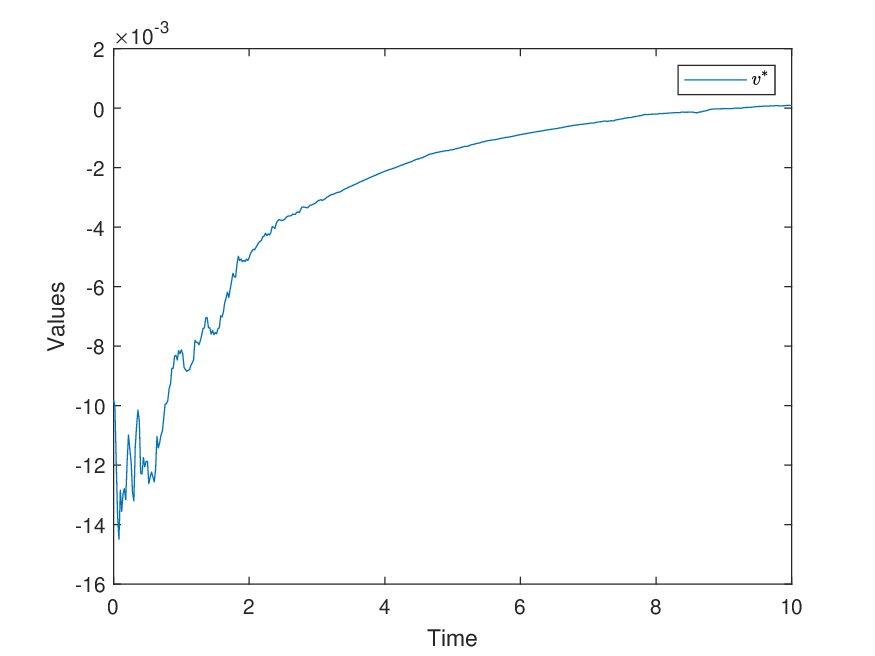}
	\caption{The worst-case disturbance $v^*(\cdot)$ for the leader}
	%\label{fig.}
	\label{6}
\end{figure}

Figure \ref{7} gives the numerical solutions of equations $\Delta^*(\cdot)$ and $\Theta^*(\cdot)$.

\begin{figure}[H]
	\centering
	\includegraphics[width=0.66\linewidth]{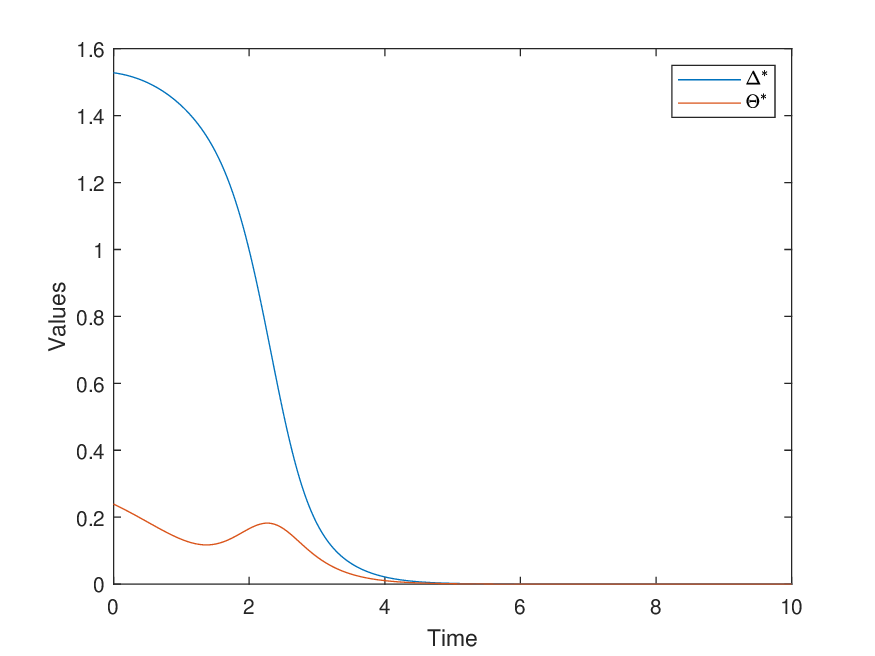}
	\caption{The numerical solutions of equations $\Delta^*(\cdot)$ and $\Theta^*(\cdot)$}
	%\label{fig.}
	\label{7}
\end{figure}

Figure \ref{8} draws the numerical solutions of incentive parameters $L^*(\cdot)$, $\eta^*(\cdot)$ and $\zeta^*(\cdot)$.

\begin{figure}[H]
	\centering
	\includegraphics[width=0.66\linewidth]{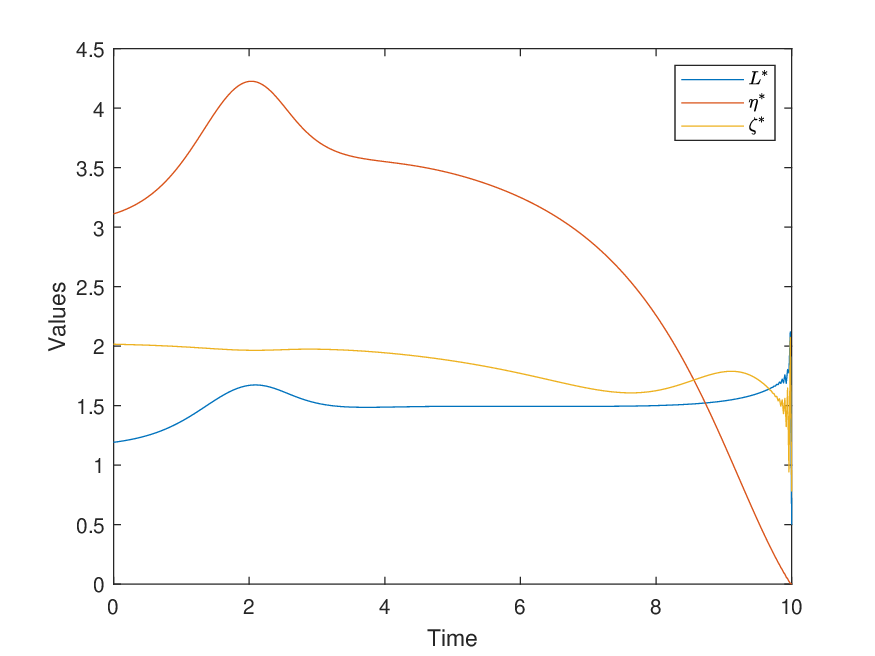}
	\caption{The numerical solutions of incentive parameters $L^*(\cdot)$, $\eta^*(\cdot)$ and $\zeta^*(\cdot)$}
	%\label{fig.}
	\label{8}
\end{figure}

The decentralized optimal state trajectories $x_i^+(\cdot)$ of 100 followers are shown in Figure \ref{9}.

\begin{figure}[H]
	\centering
	\includegraphics[width=0.65\linewidth]{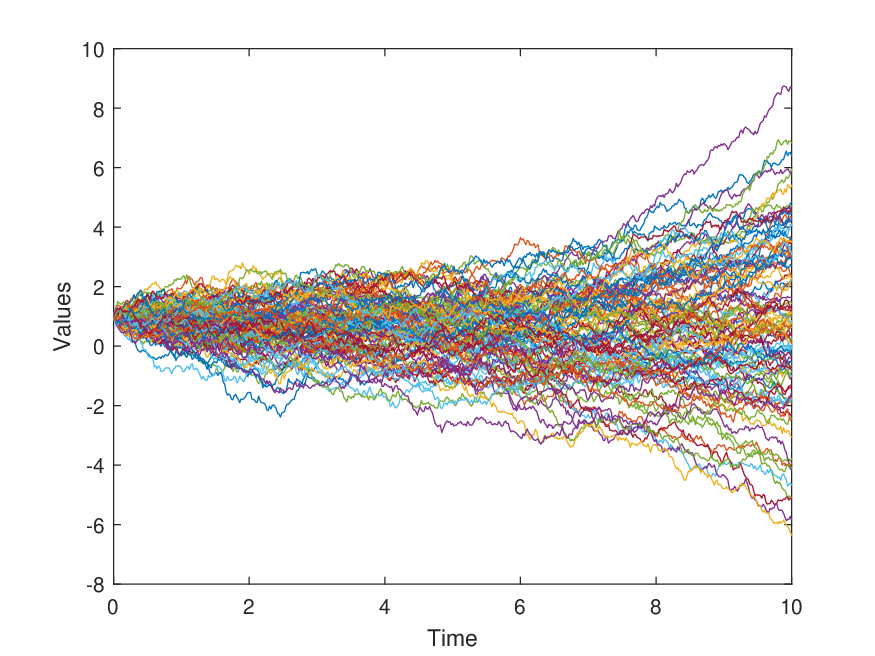}
	\caption{The decentralized state trajectories of 100 followers}
	%\label{fig.}
	\label{9}
\end{figure}
\vspace{-2mm}
Figure \ref{10} gives the decentralized strategies of $N$ followers, their average limiting strategy $\bar{u}_1^+(\cdot)$ and the optimal decentralized strategy $\bar{u}^*_1(\cdot)$ that the leader $\mathcal{A}_0$ expects followers to achieve.\vspace{-2mm}
\begin{figure}[H]
	\centering
	\includegraphics[width=0.66\linewidth]{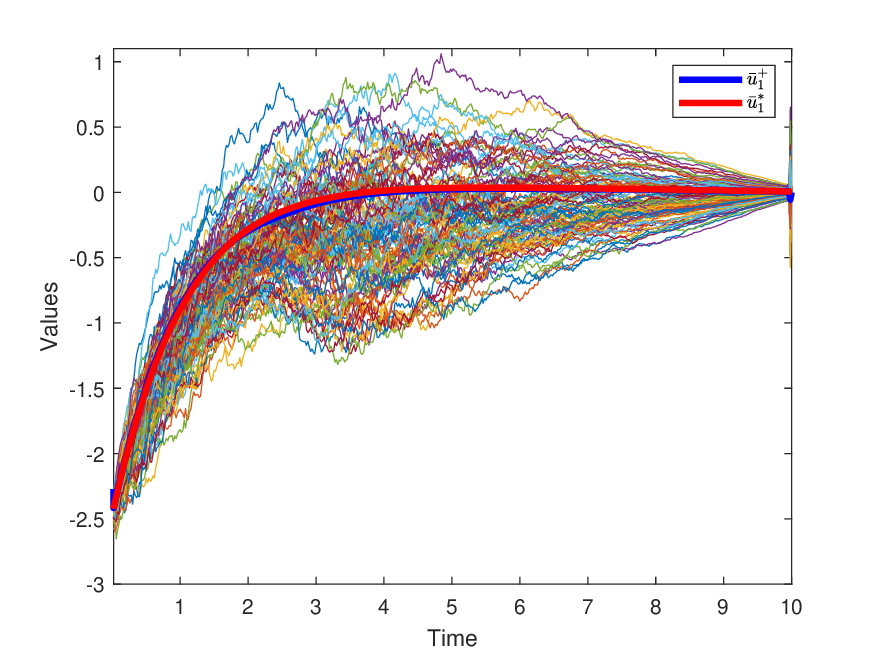}
	\caption{The mean-field strategies of 100 followers, $\bar{u}_1^+(\cdot)$ and $\bar{u}_1^*(\cdot)$}
	%\label{fig.}
	\label{10}
\end{figure}\vspace{-2mm}
Based on the previous theoretical analysis, under incentive matrices $L^*(\cdot)$, $\eta^*(\cdot)$ and $\zeta^*(\cdot)$, the leader $\mathcal{A}_0$ can induce $N$ followers' average limiting strategy $\bar{u}_1^+(\cdot)$ to equal its desired strategy $\bar{u}_1^*(\cdot)$, i.e., $\bar{u}_1^+(\cdot)=\bar{u}_1^*(\cdot)$. It can be seen from Figure \ref{10} that the thick blue line and the thick red line coincide throughout the entire time interval.

\section{Conclusions}

In this paper, we have studied a robust incentive SLQ Stackelberg MFGs with model uncertainty, where the external disturbance appears in the drift term of the leader's dynamics and the state-average and control-averages enter into the leader's dynamics and cost functional. By $H_\infty$ control theory, zero-sum game approach and duality theory, we have given the representation of the leader's limiting cost functional and the feedback decentralized open-loop saddle points via Riccati equations. Moreover, we have discussed the interrelations among uniform concavity of the limiting cost functional, solvability of the corresponding Riccati equation, and disturbance attenuation parameter. By convex analysis, the variational method and decoupling method, the followers' decentralized strategies and the consistency condition system have been derived, then the leader's approximate incentive strategy set was obtained. We have also demonstrated the asymptotical robust incentive optimality of the decentralized mean field strategy. At last, we have discussed the solvability of the CC-incentive system under some special conditions, and given a numerical example as further applications of theoretical results.

In the future, it is interesting to consider robust incentive SLQ mean field Stackelberg differential games with partial observation.

\end{document}